\documentclass[a4paper]{article}
\usepackage{amsmath}
\usepackage{amsthm}
\usepackage{amssymb}
\usepackage{graphicx}
\usepackage{authblk}
\usepackage{hyperref}
\usepackage{multirow}
\usepackage{cite}
\usepackage{color} % for comments
\usepackage[dvipsnames]{xcolor}
\usepackage{comment}
\usepackage{enumitem}
\newtheorem{theorem}{Theorem}[section]
\newtheorem{proposition}[theorem]{Proposition}
\newtheorem{lemma}[theorem]{Lemma}
\newtheorem{assumption}{\rsp{Hypothesis}}

\newtheorem{remark}{Remark}

\usepackage{fullpage}

\newcommand{\Wert}{{\vert\kern-0.25ex\vert\kern-0.25ex\vert}}

\newcommand{\figref}[1]{Fig.~\ref{fig:#1}}
\newcommand{\figlab}[1]{\label{fig:#1}}

\newcommand{\lemmaref}[1]{Lemma~\ref{lemma:#1}}
\newcommand{\lemmalab}[1]{\label{lemma:#1}}

\newcommand{\remlab}[1]{\label{remark:#1}}
\newcommand{\thmref}[1]{Theorem~\ref{theorem:#1}}
\newcommand{\thmlab}[1]{\label{theorem:#1}}

\newcommand{\appref}[1]{Appendix~\ref{app:#1}}

\newcommand{\applab}[1]{\label{app:#1}}
\newtheorem{cor}[theorem]{Corollary}
\newcommand{\corlab}[1]{\label{cor:#1}}
\newcommand{\corref}[1]{Corollary~\ref{cor:#1}}
\newcommand\KUK[1]{{\color{purple}{#1}}}

\newcommand\rsp[1]{{\color{black}{#1}}}
\newcommand\rspp[1]{{\color{black}{#1}}}
\author[1]{Kristian Uldall Kristiansen}
\author[2]{Peter Szmolyan}

\affil[1]{Department of Applied Mathematics and Computer Science, Technical University of Denmark, 2800 Kgs. Lyngby, Denmark}
\affil[2]{Institute of Analysis and Scientific Computing, Technical University of Vienna, Austria}
\affil[1]{{\tt krkri@dtu.dk}}
\affil[2]{{\tt peter.szmolyan@tuwien.ac.at}}
\title{Analytic weak-stable manifolds in unfoldings of saddle-nodes}

\begin{document}
\maketitle
\date{}
\begin{center}
%     \thispagestyle{empty}
%     \vspace*{\fill}
   \textit{Dedicated to the memory of Claudia Wulff. A dear friend and a respected colleague.}
%     \vspace*{\fill}
\end{center}

\begin{abstract}
Any attracting, hyperbolic and proper node of a two-dimensional analytic vector-field has a unique strong-stable manifold. This manifold is analytic. The corresponding weak-stable manifolds are, on the other hand, not unique, but in the nonresonant case there is a unique weak-stable manifold that is analytic. As the system approaches a saddle-node (under parameter variation), a sequence of resonances (of increasing order) occur. In this paper, we give a detailed description of the analytic weak-stable manifolds during this process. In particular, we relate a ``flapping-mechanism'', corresponding to a dramatic change of the position of the analytic weak-stable manifold as the parameter passes through the infinitely many resonances, to the lack of analyticity of the center manifold at the saddle-node. Our work is motivated and inspired by the work of Merle, Rapha\"{e}l, Rodnianski, and
              Szeftel, where this flapping mechanism is the crucial ingredient in the construction of $C^\infty$-smooth self-similar solutions of the compressible Euler equations. 
            
%               
%                constructed finite time energy blowup solutions of the isentropic ideal compressible Navier-Stokes equations. These blowup solutions were constructed through certain $C^\infty$-smooth self-similar solutions to the compressible Euler equations. The $C^\infty$-property of these self-similar solutions was obtained by carefully tracking a weak-stable $C^\infty$ manifold of a hyperbolic node close to a saddle-node bifurcation.  Our approach is different from (but inspired by) MRRS's and more in tune with dynamical systems theory. We anticipate that these results will be useful in different areas of dynamical systems theory, including in the description of weak canards of folded nodes in slow-fast systems. 
\end{abstract}

% \ffs{test}

\tableofcontents
% \begin{document}
\section{Introduction}
% The following system is an analytic normal form
\rsp{In this paper, we consider the following analytic normal form (based upon \cite[Theorem 2.2]{rousseau2005a}) for the unfolding of a planar saddle-node bifurcation:
\begin{equation}\label{eq:normalform0}
\begin{aligned}
 \dot x &= (x-\epsilon)x,\\
 \dot y &=-y(1+a^\epsilon x)+g^\epsilon(x,y),
\end{aligned}
\end{equation}
with $g^\epsilon(x,y)=\mathcal O(x^2,x^2y,xy^2)$, see Theorem \ref{thm:nf} below. Here it is important to emphasize that this formulation of the unfolding of the saddle-node is only valid on the singularity side of the saddle-node; this means that the unfolding parameter $\epsilon\ge 0$ is treated nonnegative only. In particular, the functions $a^\epsilon$ and $g^\epsilon$ depend analytically on the unfolding parameter $\epsilon\in [0,\epsilon_0)$, $\epsilon_0>0$. For $\epsilon\in (0,\epsilon_0)$, there is a saddle at $(\epsilon,\rsp{\mathcal O(\epsilon^2)})$ and a node at the origin.} It is well-known that the saddle's stable and unstable manifolds, $W^{s}$ and $W^u$, are analytic. The linearization of the node has eigenvalues $-\epsilon$ and $-1$. It is therefore resonant for $\epsilon^{-1}\in \mathbb N$. When the node is nonresonant ($\epsilon^{-1}\notin \mathbb N$) it is known \cite[Theorem 2.15]{dumortier2006a} that the node can be linearized locally by an analytic change of coordinates to the form
\begin{equation}\label{eq:xynode}
\begin{aligned}
 \dot x &=-\epsilon x,\\
 \dot y &=-y.
\end{aligned}
\end{equation}
Here $x=0$ is the strong-stable manifold $W^{ss}$, which is analytic. The invariant curves $y=c \vert x\vert^{\epsilon^{-1}}$, $c\ne 0$, tangent to the weak eigendirection at $x=0$, are all weak-stable invariant manifolds with finite smoothness at $x=0$. \rspp{The set $\{y=0\}$ is therefore the unique analytic weak-stable manifold $W^{ws}$ of \eqref{eq:xynode}}. At a resonance $\epsilon^{-1}=N\in \mathbb N$, the node can be brought into the analytic normal form
\begin{equation}\label{xynodeb}
\begin{aligned}
 \dot x &=-N^{-1} x,\\
 \dot y &=-y+b x^N,
\end{aligned}
\end{equation}
see \cite[Theorem 2.15]{dumortier2006a}. 
\rsp{All weak-stable manifolds of \eqref{xynodeb} take the graph form}
\begin{align*}
 \rsp{y=x^N c- bN x^N \log \vert x\vert,\quad c\in \mathbb R,}
\end{align*}
and therefore have finite smoothness at the \rsp{origin} in the generic case $b\ne 0$. Specifically, there is no analytic weak-stable manifold in this case. Note that this classification in the context of the normal form \eqref{eq:normalform0} is (clearly) nonuniform with respect to $\epsilon>0$. 

\rsp{For further backgound on normal forms, including formal and analytic linearizations, center manifolds and stable and unstable manifolds, we refer to the excellent book \cite{dumortier2006a}.}

\begin{figure}[h!]
\begin{center}
\includegraphics[width=.3\textwidth]{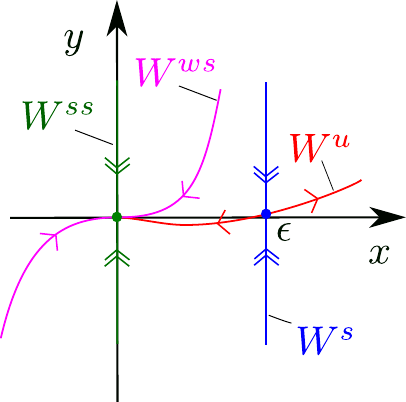}
\end{center}
\caption{Phaseportrait of \eqref{eq:normalform0} for \rsp{$0<\epsilon\ll 1$}, $\epsilon^{-1}\notin \mathbb N$, with a hyperbolic saddle at $(\epsilon,\rsp{\mathcal O(\epsilon^2)})$ with stable and unstable manifolds ($W^s$ and $W^u$ in blue and red, respectively) and a hyperbolic proper node at the origin. The node always has a unique strong-stable manifold ($W^{ss}$ in green) and in the nonresonant case ($\epsilon^{-1}\notin \mathbb N$) a unique analytic weak-stable manifold ($W^{ws}$ in magenta). }
\label{node}
% Remark: If $U$ smooth for $|u|<\delta$, the compact manifolds lie inside $|u|<\delta$ for large $n$
\end{figure}

In the present paper, we provide a detailed description of the analytic weak-stable manifold $W^{ws}$ of \eqref{eq:normalform0} for all $0<\epsilon\ll 1$ (see our \rsp{Hypotheses} \ref{assa} and \ref{deltacond} below). Our overall strategy follows \cite{MR4445442}. Here the authors constructed $C^\infty$-smooth invariant manifolds (for a specific polynomial system) by matching a global unstable manifold with an analytic weak-stable manifold close to a saddle-node $\epsilon\rightarrow 0$. These invariant manifolds correspond to $C^\infty$-smooth self-similar {solutions} of the isentropic ideal compressible Euler equations that were used in \cite{merle2022a} to determine finite time energy blowup solutions of Navier-Stokes equations (isentropic ideal compressible), see also \cite{merle2022b} for applications to the defocusing nonlinear Schr{\"o}dinger equation.

In order to control the analytic weak-stable manifolds, the authors of \cite{MR4445442} first apply a new approach for the center manifold $W^c$ at $\epsilon=0$. In particular, they define a number $S_\infty^0$, which depends on the nonlinearity (in our case, it will depend on the full jet of $g^0$), and show that if this quantity is nonzero $S_\infty^0\ne 0$, then a ``leading order term'' of the analytic weak-stable manifold can be determined. The proof of the main result of \cite{MR4445442} is not based upon dynamical systems theory but rather on careful estimation and boot-strapping arguments in order to bound the growth of the coefficients of certain series expansions. 

\rsp{In the context of \eqref{eq:normalform0}, the center manifold $W^c$ is defined for $\epsilon=0$:
% 
% 
% 
% 
% 
% 
% At the same time, for $\epsilon=0$:
\begin{equation}\label{eq:normalformeps0}
\begin{aligned}
 \dot x &= x^2,\\
 \dot y &= -y(1+a^0 x)+g^0(x,y),
\end{aligned}
\end{equation}
as an invariant manifold of the graph form $y=m^0(x)$ tangent to $y=0$. This means that $y=m^0(x)$ solves 
\begin{align*}
 x^2 y' &= -y(1+a^0 x)+g^0(x,y),
\end{align*}
obtained from \eqref{eq:normalformeps0} by eliminating time. It is well-known that although $m^0$ has a well-defined formal series expansion, which we denote by
\begin{align}\label{eq:formalseriesm0}
\widehat m^0(x)=\sum_{k=2}^\infty m_k^0 x^k,
\end{align}
it is in general only $C^\infty$-smooth, see e.g. \cite[Theorem 2.19]{dumortier2006a}.  As an example of a nonanalytic center manifold, consider $a^0=0$, $g^0(x,y)=x^2$:
\begin{align*}
 x^2 \frac{dy}{dx}=-y+x^2.
\end{align*}
This $y$-linear case corresponds to Euler's famous example. Here one can easily show that $y= \sum_{k=2}^\infty m_k^0 x^k$ (by term-wise differentiation of the series) leads to
\begin{align*}
 m_k^0 = (-1)^k (k-1)!\quad \forall\, k\ge 2.
\end{align*}
Consequently, we have $m_k^0  x^k\not \rightarrow 0$ as $k\rightarrow \infty$ for any $x\ne 0$ and the center manifold is therefore nonanalytic. The nonanalyticity of center manifolds is also intrinsically related to their nonuniqueness  (see e.g. \figref{saddle_node} below for $x<0$). 

In general, it is also known, see e.g. \cite{bonckaert2008a}, that 
the formal series expansion $\widehat m^0(x)=\sum_{k=2}^\infty m_k^0 x^k$ is Gevrey-$1$: 
% \begin{align}
%  y = \sum_{k=2}^\infty m^0_k x^k,\label{eq:formalseriesm0}
% \end{align}
% % with $m^0_k$ satisfying
% is Gevrey-1: 
\begin{align}
\vert m^0_k\vert \le C D^{-k} k! \quad \forall\, k\ge 2,\label{eq:gevrey1}
\end{align}
for some $C,D>0$. (In fact, this formal series is $1$-summable
along any sector that is not centered along the negative real axis, see \cite[Chapter 3]{balser1994a} and \cite{bonckaert2008a} for further details.) The Gevrey-$1$ property of the formal series in \eqref{eq:gevrey1} also holds true for one-dimensional center manifolds of higher dimensional saddle-nodes (with one single zero eigenvalue of the linearization). We refer to \cite{bonckaert2008a} for further details. In contrast, higher dimensional center manifolds may only have finite $C^k$-smoothness, see \cite{Strien1979}.}

\subsection{\rsp{Informal statement of the main results}}
\rsp{In this paper, we consider the general case \eqref{eq:normalform0} (as opposed to a specific $g^\epsilon$ as in \cite{MR4445442}), while adding the following technical hypotheses: 
\begin{align}
 a^0>-2,\quad g^\epsilon(x,y)- g^\epsilon(x,0)=\mathcal O(\mu),\label{eq:gepsmu0}
\end{align}
with $0\le \mu<\mu_0$ small enough, see further details below. Here $\mu_0$ is independent of $\epsilon\ge 0$. We conjecture that our results are true without these assumptions (i.e. for any analytic and generic unfolding of a saddle-node), but leave this to future work (see Section \ref{sec:discussion}). %Our approach only has to be adapted in some places and we firmly believe that the general strategy carries over. At the same time, 
We feel that \eqref{eq:gepsmu0} gives a suitable forum to present the phenomenon in an accessible way. }
%\rsp{Hypothesis} \ref{deltacond} below.  
\rsp{We then summarize our results as follows:
 \begin{theorem}\thmlab{main0}
 Suppose \eqref{eq:gepsmu0} and let $y=\sum_{k=2}^\infty m_k^0 x^k$ denote the formal series expansion of the center manifold $W^c$ for $\epsilon=0$. Then
 \begin{align}
  S_\infty^0 := \lim_{k\rightarrow \infty} \frac{(-1)^k m_k^0}{\Gamma(k+a^0)},\label{eq:Sinfty00}
 \end{align}
where $\Gamma$ is the gamma function (see Section \ref{sec:Gamma} below), is well-defined for all $\mu \in [0,\mu_0)$, with $\mu_0>0$ small enough. Moreover, if $S_\infty^0\ne 0$ then the following holds true:
  \begin{enumerate}
   \item The center manifold $W^{cs}$ for $\epsilon=0$ is nonanalytic (see \thmref{main1}).
   \item\label{nonint} $W^{ws} \cap W^u =\emptyset$ for all $0<\epsilon\ll 1,\,\epsilon^{-1}\notin \mathbb N$ (see \thmref{main2}).
   \item  The sign of $S_\infty^0\ne 0$  determines on what side of $W^u$ the analytic weak-stable manifold $W^{ws}$ lands (\rsp{within the strip defined by $x\in (0,\epsilon)$}, see \thmref{main2} and Fig. \ref{node}.)
   \item \label{flapping} On the far side of the saddle (i.e. for $x<0$), the weak-stable analytic manifold $W^{ws}$ transitions from intersecting $y=c>0$ to intersecting $y=-c<0$ ($c>0$ small but $\mathcal O(1)$)  as $\epsilon^{-1}$ transverses (sufficiently large) integers (the resonances) (see \corref{main2b} and \figref{node_cases}). The position of $W^{ws}$ close to each resonance $\epsilon^{-1}\in \mathbb N$, $\epsilon^{-1}\gg 1$, is determined by the sign of $S_\infty^0$ and on whether $\epsilon^{-1}$ is even or odd. 
  \end{enumerate}
 \end{theorem}
 The results of \thmref{main0} are (as indicated) stated precisely in our main result section, see Section \ref{sec:main}. 
  The (dramatic) movement of $W^{ws}$ as $\epsilon^{-1}$ transverses (sufficiently large) integers (described informally in item \ref{flapping} of \thmref{main0}) was central in \cite{MR4445442} for the construction of their special $C^\infty$ self-similar solutions. We will refer to this phenomenon as ``flapping'', see also Section \ref{sec:examples}.}
 
\rsp{While the discovery of the phenomenon in a specific problem is due to \cite{MR4445442}, our treatment of the underlying general mechanism is novel and more in the spirit of dynamical systems theory. We also feel that our proof streamlines the approach of \cite{MR4445442}, used for their specific nonlinearity (rational with numerator cubic, denominator quadratic, see \cite[Eqs. (1.9)--(1.10)]{MR4445442}). Moreover, we will perform the important estimates not by brute force calculations but by using appropriate arguments (including fixed-point theorems) in suitable normed spaces.}

\rspp{Analyticity will throughout refer to real analyticity and we will work with $x\in \mathbb R$. This is in contrast to related studies of exponentially small phenomena, see e.g. \cite{baldom2013a,hayes2015a}, where extensions into the complex plane $x\in \mathbb C$ play a crucial role.}
 
% In this section, 

\subsection{Further background}
% We now proceed to discuss our results in the context of the existing literature. 
 
We emphasize that item \ref{nonint} of \thmref{main0} is in line with the statement in \cite[Example 3.1, p. 13]{rousseau2023a}, which says that 
\begin{align*}
W^c\notin  C^\omega  \Longrightarrow W^{ws} \cap W^u =\emptyset\quad \forall \, 0<\epsilon\ll 1,\,\epsilon^{-1}\notin \mathbb N.
\end{align*}
This is of course the generic situation. \rsp{(Keep in mind that, as the branches (away from the equilibria) of $W^{ws}$ and $W^{u}$ are orbits of the planar problem \eqref{eq:normalform0},  $W^{ws}$ and $W^{u}$ either coincide or do not intersect)}.  \rsp{However, a novel aspect of our results is that we show that the sign of $S_\infty^0\ne 0$ determines (along with the parity of $\lfloor \epsilon^{-1}\rfloor \gg 1$) on what side of $W^u$ the analytic weak-stable manifold $W^{ws}$ lands. }

\rsp{In \cite{martinet1983a}, Martinet and Ramis presented their analytic characterization of saddle-nodes through \textit{analytic invariants}. The quantity $a^0$ is here known as \textit{the formal analytic invariant}. The name is motivated by the fact that there exists a \textit{formal} transformation (a formal series with respect to $x$ and $y$) that brings the saddle-node for $\epsilon=0$ into
\begin{align*}
 \dot x &= x^2,\\
 \dot y &=-y(1+a^0x).
\end{align*}
In general (also for Poincar\'e ranks $r\in \mathbb N\setminus \{1\}$ where $\dot x = x^{r+1}$), the formal transformation can be summed (in the sense of Borel-Laplace) along complex sectors (whose union cover the origin) and in essence the analytic invariants encode the relationship between these on overlapping domains. Here there is a deep connection to the Ecalle-Voronin classification of tangent-to-the-identity maps, see \cite{ecalle1,voronin1981a,loray2004a}.}

\rsp{The quantity $S_\infty^0$ (somehow) relates to the so-called \textit{translational part of the analytic invariants of Martinet and Ramis}. Indeed, the center manifold is analytic if and only if this translational part is trivial, see also \cite[Section 1]{loray2004a}. (Alternatively, the center manifold is analytic if and only if the Borel transform of $\widehat m^0$ is entire). However, the details of this connection is still not clear to the authors. For example, we do not know whether $S_\infty^0= 0$ implies that the center manifold is analytic (we will show below that it holds in the $y$-linear case). At the same time, $S_\infty^0$ also acts like a Stokes constant that determines the properties of the unfolding. This is reminiscent of the Stokes constant in \cite{baldom2013a} for the unfolding of the zero-Hopf.}

\rsp{On the other hand, \eqref{eq:Sinfty00} also provides an estimate of the growth of the coefficients $m_k^0$ as $k\rightarrow \infty$, insofar that there is a constant $F>0$ such that 
\begin{align}\label{eq:m0kest0}
 \vert m_k^0\vert \le F\Gamma(k+a^0)\quad \forall k\ge 2.
\end{align}
By Stirling's formula (see \eqref{stirling} in Appendix \ref{sec:Gamma}):
\begin{align}\label{eq:stirling00}
 \Gamma(k+a^0) = (1+o(1)) k! k^{a^0-1},
\end{align}
for $k\gg 1$, it follows that $m_k^0$ is Gevrey-$1$ (see \eqref{eq:gevrey1}) for any $D<1$, also $D=1$ if $a^0\le 1$.} %The bound \eqref{eq:m0kest0} agrees with \cite[Lemma 5.4]{MR4445442} on their specific nonlinearity. }

\subsection{Outline}
\rsp{The paper is structured as follows: In Section \ref{sec:examples}, we provide a first glimpse of the phenomena  that we study  through simple examples. We also use this section to introduce our terminology and (parts of our) notation. Subsequently, in Section \ref{sec:main} we present our hypotheses and our main results in full technical details. In Section \ref{sec:cm}, we then prove the statements relating to the center manifold. Statements relating to $\epsilon>0$ are proven in Section \ref{sec:node}. (Section \ref{sec:overview} has a more detailed overview of the proofs.) We conclude the paper with a discussion section, see Section \ref{sec:discussion}. Appendix \ref{sec:Gamma} contains some relevant properties of the gamma function that will be used throughout. }
\section{\rsp{Motivating examples}}\label{sec:examples}
% \KUK{$u$ not polynomial, otherwise no resonance}
\rsp{Consider first the following simple example}:
\begin{equation}\label{eq:node_init}
\begin{aligned}
 \dot x &=-\epsilon x,\\
 \dot y &=-y+u(x),
\end{aligned}
\end{equation}
with $$u(x) = \sum_{k=2}^\infty u_k x^k,\quad \vert u_k\vert\le B \rho^{-k},$$
being analytic on the open disc $\vert x\vert<\rho$. 
% Here $\epsilon$ is the eigenvalue ratio of the two eigenvalues $-1,-\epsilon$. 
Here $x=0$ is the strong stable manifold $W^{ss}$ and for $\epsilon^{-1}\notin \mathbb N$ the analytic weak-stable manifold $W^{ws}$ exists and takes the graph form 
\begin{align}
 y = m^\epsilon(x),\quad m^{\rsp{\epsilon}}(x) = \sum_{k=2}^\infty \frac{u_k}{1-\epsilon k}x^k\quad \forall\,0\le \vert \rsp{x}\vert<\rho.\label{eq:m0}
\end{align}
This follows from a simple calculation. 
Notice that there are small divisors in the expression for $m^\epsilon$ for $\epsilon\approx \frac{1}{N}$, $N\in \mathbb N$ (the resonances). Let $$\epsilon^{-1} = N^\epsilon+\alpha^\epsilon\notin \mathbb N, \quad N^\epsilon:=\lfloor \epsilon^{-1}\rfloor, \quad \alpha^\epsilon\in (0,1),$$ and define
\begin{align}\label{eq:Veps00}
 V^\epsilon(x) = N^\epsilon \frac{u_{N^\epsilon}}{\alpha^\epsilon} x^{N^\epsilon}  - (N^\epsilon+1) \frac{u_{N^\epsilon+1}}{1-\alpha^\epsilon} x^{N^\epsilon+1}.
\end{align}
Then the sum of the terms in \eqref{eq:m0} with $k=N^\epsilon$ and $k=N^\epsilon+1$ takes the  following form
\begin{equation}\nonumber
\begin{aligned}
  \sum_{k=N^\epsilon}^{N^\epsilon+1} \frac{u_k}{1-\epsilon k}x^k &= \epsilon^{-1} \frac{u_{N^\epsilon}}{\alpha^\epsilon}x^{N^\epsilon}  -\epsilon^{-1} \frac{u_{N^\epsilon+1}}{1-\alpha^\epsilon}x^{N^\epsilon+1}\\
 &= \left(N^\epsilon \frac{u_{N^\epsilon}}{\alpha^\epsilon}+u_{N^\epsilon}\right) x^{N^\epsilon}  -\left((N^\epsilon+1) \frac{u_{N^\epsilon+1}}{1-\alpha^\epsilon} -u_{N^\epsilon+1} \right) x^{N^\epsilon+1}\\
 &=V^\epsilon(x)+u_{N^\epsilon} x^{N^\epsilon} + u_{N^\epsilon+1} x^{N^\epsilon+1}.
\end{aligned}
\end{equation}
It follows that $B^\epsilon:=m^\epsilon-V^{\epsilon}$ is uniformly bounded with respect to $\alpha^\epsilon\in [0,1)$, and for any $\upsilon>0$ there is a $\delta>0$ such that 
\begin{align}\label{eq:mVinit}
\vert B^\epsilon(x)\vert \le \upsilon \quad \forall\, 0\le \vert x\vert \le \delta,\,\alpha^\epsilon\in (0,1).
\end{align}
The function $V^{\epsilon}$, on the other hand, is not uniformly bounded if $u_{N^\epsilon}\ne 0$ or $u_{N^\epsilon+1}\ne 0$. Specifically, if $u_{N^\epsilon}u_{N^\epsilon+1}\ne 0$ then it follows that \textit{we can track $W^{ws}:\,y=m^\epsilon(x)$ through $y=V^{\epsilon}(x)$ for $x\ne 0$, $\alpha^\epsilon\rightarrow 0^+$  and $\alpha^\epsilon\rightarrow 1^-$} (since $V^{\epsilon}(x)$, $x\ne 0$, goes unbounded in these limits). Here by ``track'' we will mean that the position of $W^{ws}$ can qualitatively be determined as follows:
\begin{lemma}\label{lemma:baby1}
Fix $N^\epsilon\in \mathbb N$, $K>0$, suppose that $u_{N^\epsilon}\ne 0$ and define $s=\operatorname{sign}(u_{N^\epsilon})$. Let
$W^{ws}:\,y=m^\epsilon(x),\,0\le\vert x\vert <\rho,$ denote the analytic weak-stable manifold of \eqref{eq:node_init}, $\epsilon^{-1}\notin \mathbb N$.
Then the following holds true for all $0<\delta\ll 1$: 

The position of $W^{ws}$ for all $0<\alpha^\epsilon <N^\epsilon \frac{ \vert u_{N^\epsilon}\vert }{K}\delta^{N^\epsilon}$ can be determined as follows:
\begin{enumerate}
 \item Suppose that $N^\epsilon$ is even. Then $W^{ws}$ intersects $\{y = \frac{sK}{2}\}$ for both $-\delta<x<0$ and $0<x<\delta$. 
 \item Suppose that $N^\epsilon$ is odd.  Then $W^{ws}$ intersects $\{y = -\frac{sK}{2}\}$ for $-\delta<x<0$ and $\{y = \frac{sK}{2}\}$ for $0<x<\delta$. 
\end{enumerate}

% the analytic weak-stable manifold $y=m^\epsilon(x)$, $0\le\vert x\vert\le \delta$, of \eqref{eq:node_init} intersects the sections $y=\pm \frac{K}{2}$ in to points  Here the sign of $y$ on the section is determined by the sign of $V^{\epsilon}(x)$.
\end{lemma}
\begin{proof}
For $u_{N^\epsilon}\ne 0$, we have
\begin{align*}
 \left|N^\epsilon  \frac{u_{N^\epsilon}}{\alpha^\epsilon}\delta^{N^\epsilon}\right|>K\quad \forall\, 0<\alpha^\epsilon<N^\epsilon \frac{ \vert u_{N^\epsilon}\vert }{K}\delta^{N^\epsilon}.
\end{align*}
Then from \eqref{eq:Veps00}, 
% \begin{align*}
%  \overline V^\epsilon(x) = N^\epsilon \frac{u_{N^\epsilon}}{\alpha^\epsilon} x^{N^\epsilon} +\left(u_{N^\epsilon} x^{N^\epsilon} -\epsilon^{-1} \frac{u_{N^\epsilon}}{1-}
% \end{align*}
we obtain that the following holds true for $x=-\delta$ and $x=\delta$:
\begin{align*}
 \vert V^{\epsilon}(x)\vert \ge \frac{3K}{4}\quad \forall\, 0<\alpha^\epsilon <N^\epsilon\frac{ \vert u_{N^\epsilon}\vert}{K} \delta^{N},\, 0<\delta\ll 1.
\end{align*}
% with 
% Moreover, for these values of $\alpha^\epsilon$, we have that $\frac{d}{dx}V^{\epsilon}(x)=0$, $0\le \vert x\vert\le \delta$, only if $x=0$. 
Then by using $m^\epsilon = B^\epsilon+V^\epsilon$ and \eqref{eq:mVinit} with $0<\upsilon\ll K$ for $0<\delta\ll 1$ the result follows.
\end{proof}
A similar result holds for $0<1-\alpha^\epsilon\ll 1$ if $u_{N^\epsilon+1}\ne 0$, which we state without proof. 
\begin{lemma}\label{lemma:baby2}
Fix $N^\epsilon\in \mathbb N$, $K>0$, suppose that $u_{N^\epsilon+1}\ne 0$ and define $s=\operatorname{sign}(u_{N^\epsilon+1})$. Let
$W^{ws}:\,y=m^\epsilon(x),\,0\le\vert x\vert <\rho,$ denote the analytic weak-stable manifold of \eqref{eq:node_init}, $\epsilon^{-1}\notin \mathbb N$.
Then the following holds for all $0<\delta\ll 1$:

The position of $W^{ws}$ for all $0<1-\alpha^\epsilon <(N^\epsilon+1) \frac{ \vert u_{N^\epsilon+1}\vert }{K}\delta^{N^\epsilon+1}$ can be determined as follows: 
\begin{enumerate}
 \item Suppose that $N^\epsilon$ is even. Then $W^{ws}$ intersects $\{y = \frac{sK}{2}\}$ for $-\delta<x<0$ and $\{y=-\frac{sK}{2}\}$ for $0<x<\delta$. 
 \item Suppose that $N^\epsilon$ is odd.  Then $W^{ws}$ intersects $\{y = -\frac{sK}{2}\}$ for both $-\delta<x<0$ and $0<x<\delta$. 
\end{enumerate}
\end{lemma}
% In particular, if $u_{N^\epsilon}u_{N^\epsilon+1}\ne 0$ then the first term (second term) of $V$ ``dominates'' $m$ for some $x\ne 0$ for all $0<\alpha^\epsilon\ll 1$ ($0<1-\alpha^\epsilon\ll 1$, respectively). Although ``dominate term'' is clearly very vague here, we expect that the use of the word is clear enough in this informal discussion of the paper. 

By Lemma \ref{lemma:baby1} and Lemma \ref{lemma:baby2}, we obtain a ``flapping phenomenon'' when $u_{N^\epsilon}u_{N^\epsilon+1}\ne 0$, \textit{whereby the position of $W^{ws}$ (at least on one side of the node) changes dramatically as $\alpha^\epsilon$ transverses the interval $(0,1)$}.  We illustrate this flapping phenomena in \figref{node_cases_init} for $u_{N^\epsilon}>0,u_{N^\epsilon+1}<0$ (which is relevant for \eqref{eq:normalform} with $0<\epsilon\ll 1$; \rsp{the reader should compare the figure} with \figref{node_cases}). 
\begin{remark}
\rspp{It is essentially the flapping mechanism that (together with a basic continuity argument) allows the authors of \cite{MR4445442} to connect their analytic weak-stable manifolds with a global analytic manifold (that does not ``flap'') and construct $C^\infty$-smooth self-similar solutions close to resonances (and close to a saddle-node where the resonances accumulate).} %'
\end{remark}
\begin{figure}[h!]
\begin{center}
\includegraphics[width=.65\textwidth]{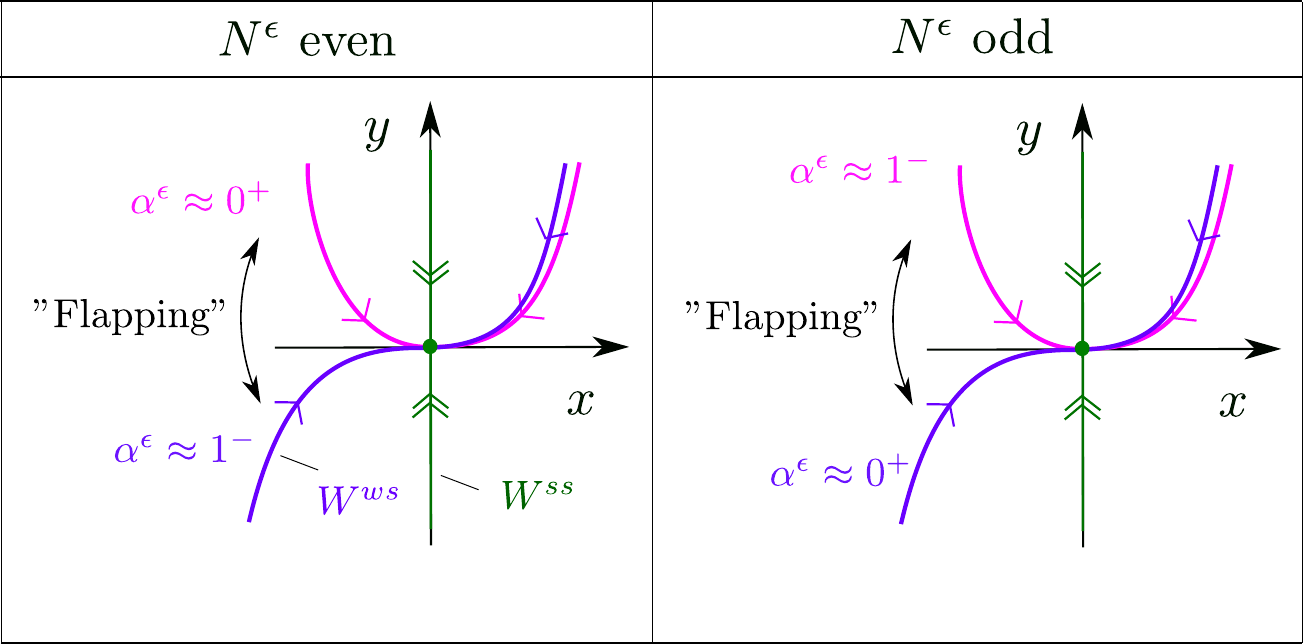}
\end{center}
\caption{The ``flapping phenomenon'' of the analytic weak-stable manifolds ($W^{ws}$ in magenta and purple) of \eqref{eq:node_init} for $u_{N^\epsilon}>0, u_{N^\epsilon+1}<0$.}
\figlab{node_cases_init}
% Remark: If $U$ smooth for $|u|<\delta$, the compact manifolds lie inside $|u|<\delta$ for large $n$
\end{figure}

For a general (fully nonlinear) analytic system,  quantities corresponding to $u_{N^\epsilon}$ and $u_{N^\epsilon+1}$ for a hyperbolic node  can in principle be computed for any fixed $N^\epsilon$ in terms of the jet of the nonlinearity (through normal form computations \cite[Chapter 2]{dumortier2006a}). But in the context of \eqref{eq:normalform0}, our results show (see Section \ref{sec:main}) that the condition $u_{N^\epsilon}u_{N^\epsilon+1}\ne 0$ can be related to the lack of analyticity \rspp{(through $S_\infty^0$)} of the center manifold $y=m^0(x)$ of the origin for $\epsilon=0$.

\begin{remark}\label{rem:upol}
\rsp{Notice that in the context of the $y$-linear system \eqref{eq:node_init}, the flapping phenomena for $\epsilon\rightarrow 0$ does not appear if $u$ is a polynomial. Indeed, the flapping is caused by an accumulation of resonances for $\epsilon\rightarrow 0$ and if $u$ is a polynomial then there are only finitely many (possible) resonances for \eqref{eq:node_init}.}
\end{remark}

Next, to illustrate how \eqref{eq:Sinfty00} and the bound \eqref{eq:m0kest0} occur, consider the case $g^0(x,y)=f^0(x)$ (so that $g^0$ is independent of $y$) in \eqref{eq:normalformeps0}. \rsp{As we are interested in invariant manifolds, we eliminate time to obtain the following differential equation for $y=y(x)$}:
\begin{align}
 x^2 \frac{dy}{dx}+y(1+a^0x)=f^0(x),\quad f^0(x)=\sum_{k=2}^\infty f^0_k x^k.\label{eq:eps0mu01}
\end{align}
This equation is linear in $y$ and one can solve explicitly for the $m_k^0$'s of the formal series. Indeed, inserting the formal series \eqref{eq:formalseriesm0} into \eqref{eq:eps0mu01}
leads to 
\begin{align*}
 \sum_{k=2}^\infty \left( (k +a^0)m_{k}^{\rsp{0}} x^{k+1}+ m_k^0 x^k\right)=\sum_{k=2}^\infty f_k x^k,
\end{align*}
and therefore to 
the recursion relation:
\begin{align}
 m_k^0 + (k-1+a^0)m_{k-1}^0 = f^0_k\quad \forall \,k\ge 2,\label{eq:recmk0}
\end{align}
with $m_1^0= 0$. 
\begin{lemma}\lemmalab{lemma:mk0linearcase}
Suppose that $a^0>-2$ and define 
\begin{align}
 S_k^0 :=\sum_{j=2}^k \frac{(-1)^j f^0_j}{\Gamma(j+a^0)}.\label{eq:S0k}
\end{align}
Then the solution of the recursion relation \eqref{eq:recmk0} with $m_1^0=0$ is
 \begin{align}
 m_k^0  = (-1)^k \Gamma(k+a^0) S_k^0.\label{eq:mk0linearcase}
\end{align}
\end{lemma}
\begin{proof}
 The result can easily be proven by induction using the base case $m_1^0=0$ and the basic property of the gamma function: $\Gamma(z+1)=z \Gamma(z)$, see \eqref{eq:Gamma}, in the induction step. 
%  
%  However, it is useful to proceed in a more constructive way: 
%  
%  Write
%  \begin{align}\label{eq:S0k}
%  w_k^0:= \Gamma(k+a^0),\quad m^0_k = (-1)^k w^0_k S^0_k, \quad S^0_1=0. 
%  \end{align}Then 
% \begin{align}
%  S^0_k - S^0_{k-1} = (-1)^k \frac{f^0_k}{w^0_k},\label{eq:xi0keq}
% \end{align}
% using that 
% \begin{align*}
%  w^0_k = (k-1+a^0) w^0_{k-1},
% \end{align*}
% see \eqref{eq:Gamma}. 
% % for all $k\ge 2$. 
% % 
% % \fbox{If $\xi_k\le C$ then $\vert g^0_k\vert \le K(C)w_{k-2}^0$}
% % 
% % \fbox{$\Rightarrow $ rhs of \eqref{xi0keq} goes to zero on compact sets}
% % 
% % \fbox{Majorant equation for \eqref{xi0keq}?}
% It is elementary to show that
% \begin{align}
%  S_k^0 = \sum_{j=2}^k \frac{(-1)^j f_j^0}{w_j^0},\quad k\ge 2.\label{eq:S0kf}
% \end{align}
% Indeed, the statement is true for $k=2$ using $S_1^0=0$ in \eqref{eq:xi0keq}. Suppose therefore that it is true for $k-1$. Then by \eqref{eq:xi0keq}:
% \begin{align*}
%  S^0_k &= \sum_{j=2}^{k-1} \frac{(-1)^j f_j^0}{w_j^0}+(-1)^k \frac{f^0_k}{w^0_k}\\
%  &= \sum_{j=2}^{k} \frac{(-1)^j f_j^0}{w_j^0},
% \end{align*}
% as claimed. 
\end{proof}
Seeing that $f^0$ is analytic, we have 
\begin{align*}
 \vert f_k^0\vert \le B\rho^{-k},
\end{align*}
for some $B>0$, $\rho>0$,
and the sum
\begin{align*}
 S_\infty^0 := \lim_{k\rightarrow \infty} S_k^0=\lim_{k\rightarrow \infty} \frac{(-1)^k m_k^0}{\Gamma(k+a^0)}= \sum_{j=2}^\infty \frac{(-1)^j f_j^0}{\Gamma(j+a^0)},
\end{align*}
see \eqref{eq:S0k},
is therefore absolutely convergent for any $a^0>-2$:
\begin{align*}
 \vert S_\infty^0\vert \le \sum_{j=2}^\infty \frac{\vert f_j^0\vert}{\Gamma(j+a^0)}\le F:=B\sum_{j=2}^\infty \frac{\rho^{-j}}{\Gamma(j+a^0)}<\infty.
\end{align*}
The property \eqref{eq:m0kest0} therefore follows from \eqref{eq:mk0linearcase} in the context of \eqref{eq:eps0mu01}.
% \begin{align*}
%  F = \sum_{j=2}^\infty \frac{\vert f_j^0\vert}{\Gamma(j+a^0)}<\infty,
% \end{align*}

% \begin{remark}
Notice that if $S_\infty^0\ne 0$, then by \eqref{eq:stirling00} we have
\begin{align}
m_k^0 = (-1)^k \Gamma(k+a^0) S_k^0 = (-1)^k (1+o(1)) S_\infty^0 \Gamma(k+a^0) =(-1)^k (1+o(1))S_\infty^0 k^{a^0-1} k!\label{eq:mkmu0}
\end{align}
for all $k\gg 1$. 
This implies that the center manifold is nonanalytic. In the linear case \eqref{eq:eps0mu01}, it is also possible to go the other way. We collect this in the following lemma:
\begin{lemma}\label{lemma:linear}
 {Suppose that $a^0>-2$.} Then the center manifold of the linear system \eqref{eq:eps0mu01} is analytic if and only if $S_\infty^0=0$. 
\end{lemma}
\begin{proof}
 $\Rightarrow$: From \eqref{eq:mkmu0}, we have that for any $x\ne 0$, $ m_k^0 x^k \not \rightarrow 0$ for $k\rightarrow0$. Consequently, if $S_\infty^0\ne 0$ then the center manifold is nonanalytic. 
 
 $\Leftarrow$: If $S^0_\infty =0 $ then
 \begin{align*}
  \vert S_k^0 \vert = \vert S_\infty^0 -S_k^0 \vert \le B\sum_{j=k+1}^\infty \frac{ \rho^{-j}}{\Gamma(j+a^0)},
 \end{align*}
and for any $k\ge k_0(a^0)$, $j\in \mathbb N_0$:
\begin{align*}
 \frac{\Gamma(k+a^0)}{\Gamma(k+1+j+a^0)} = \frac{(k-1)!}{(k+j)!}(1+o_{k_0\rightarrow \infty}(1))\left(\frac{k}{k+1+j}\right)^{a^0}\le \frac{2}{k^j}\left(\frac{k}{k+1+j}\right)^{-2}\le 8(1+j)^2k^{-j},
\end{align*}
with $k_0\gg 1$,
using Stirling's formula (see \eqref{stirling} below)
and
\begin{align*}
\left(\frac{k+1+j}{k}\right)^{2} \le \left(\frac{2k(1+j)}{k}\right)^{2} = 4(1+j)^2.
\end{align*}
Then upon using $\sum_{j=0}^\infty2^{-j}(1+j)^2 = 12$ it follows that 
\begin{align*}
 \vert m_k^0 \vert \le 8B\rho^{-k-1}\sum_{j=0}^\infty (\rho k)^{-j}(1+j)^2\le 96B\rho^{-k-1}\quad \forall\,k\ge k_0\ge 2\rho^{-1}.
\end{align*}
We conclude that $\sum_{k=2}^\infty m_k^0 x^k$ converges absolutely for all $0\le \vert x\vert < \rho$ if $S_\infty^0=0$.
\end{proof}
% \end{remark}
% \KUK{expand this}
% \begin{remark}
% Consider 
% \begin{align*}
%  x^2 \frac{dy}{dx}=-y(1+a^0x) + \frac{x^2}{1-x}.
% \end{align*}
% In this case, $S_\infty^0$ can be computed:
% \begin{align*}
%  S_\infty^0 = (-1)^{-a^0} e^{-1} \frac{\int_0^{a^0} e^{-t} t^{a^0-1} dt}{\Gamma(1+a^0)},
% \end{align*}
% (which has a removable singularity for $a^0=-1$). 
% \end{remark}

% (an incomplete Gamma-function). 

A first important step of our approach is to carry the classification of the analyticity of the center manifold for $\epsilon=0$ over to the nonlinear case. For this, we will use a fixed-point argument in an appropriate Banach space of formal series. This leads to the definition of $S_\infty^0$ for a nonlinearity $g^0$, satisfying the \rsp{Hypotheses} \ref{assa} and \ref{deltacond} below. %If $S_\infty^0\ne 0$ then the center manifold will be nonanalytic. 
% In \cite{} it was observed  that the nonanalyticity of the center manifold is related to a mis-match for $\epsilon>0$ of the analytic unstable manifold of saddle and the analytic weak-stable manifold of the node for \eqref{eq:normalform0}.
\begin{figure}[h!]
\begin{center}
\includegraphics[width=.30\textwidth]{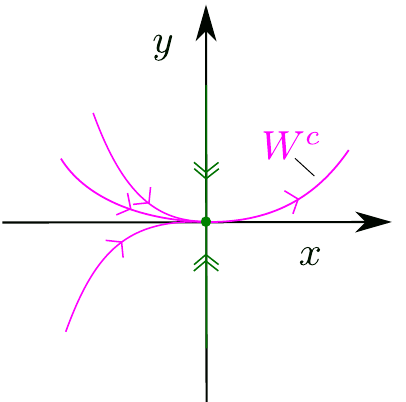}
\end{center}
\caption{The saddle node for $\epsilon=0$. The center manifold ($W^c$ in magenta) is only unique on the positive side of $x=0$.}
\label{fig:saddle_node}
% Remark: If $U$ smooth for $|u|<\delta$, the compact manifolds lie inside $|u|<\delta$ for large $n$
\end{figure}
% In this paper, we provide a detailed description of the analytic weak-stable invariant manifold as $\epsilon\rightarrow 0$ (under some additional assumptions listed above, see \rsp{Hypothesis} \ref{assa} and \rsp{Hypothesis} \ref{deltacond}). In particular, under a certain condition 

Subsequently, for $\epsilon>0$ and $S_\infty^0\ne 0$, we (essentially) expand the analytic weak-stable invariant manifold $y=m^\epsilon(x)$ into the form 
\begin{align}
 m^\epsilon = B^\epsilon+(-1)^{N^\epsilon}S_\infty^0 V^{\epsilon},\quad N^\epsilon = \lfloor \epsilon^{-1}\rfloor,\nonumber%\label{eq:mepsExpansion0}
\end{align}
on a subset $x\in I^\epsilon$, where (in essence, see Theorem \ref{theorem:main2} for details) \textit{only $B^\epsilon$ is uniformly bounded with respect to $\alpha^\epsilon = \epsilon^{-1}-\lfloor \epsilon^{-1}\rfloor\in (0,1)$}. We will therefore \textit{track $y=m^\epsilon(x)$ for $S_\infty^0\ne 0$ using $y=(-1)^{N^\epsilon} S_\infty^0 V^{\epsilon}(x)$ for $\alpha^\epsilon \rightarrow 0^+$ and $\alpha^\epsilon\rightarrow 1^-$} as in the example \eqref{eq:node_init} above. (It would be more accurate to say that the tracking will first be done in scaled coordinates, see \eqref{scalingxy}, and that $V^\epsilon(x)=\epsilon \overline V^\epsilon(\epsilon^{-1}x)$, see \eqref{eq:Veps0}. Moreover, $x>0$ and $x<0$ will be treated slightly different, but we refer the reader to further details and the precise statements below.) In this context, it is (again) worth pointing out that $S_\infty^0\ne 0$ essentially ensures that a condition like $u_{N^\epsilon}u_{N^\epsilon+1}\ne  0$ holds true near all resonances $\epsilon^{-1}\in \mathbb N$ for $0<\epsilon\ll 1$, see \thmref{main2} and \corref{main2b}.  %The condition \eqref{eq:S0infcond} will imply that the center manifold is nonanalytic.

\section{Main results}\label{sec:main}
% Following \cite{}, we have the following analytic normal form for a generic unfolding of an analytic saddle-node in $\mathbb R^2$:
%
% where
We first state a general result (based upon  \cite[Theorem 2.2]{rousseau2005a}) on saddle-nodes. 
\begin{theorem}\label{thm:nf}
 For any analytic and generic family of two-dimensional vector-fields unfolding a saddle-node, there exists a locally defined analytic change of coordinates, parameters and time, such that on the singularity-side ($\epsilon\ge 0$) of the bifurcation, the system takes the following normal form:
\begin{equation}\label{eq:normalform}
\begin{aligned}
 \dot x &= (x-\epsilon)x,\\
 \dot y &=-y(1+a^\epsilon x)+g^\epsilon(x,y),
\end{aligned}
\end{equation}
% by , and
where 
\begin{equation}\label{eq:gcondapp}
\begin{aligned} 
 g^\epsilon(x,y) &= f^\epsilon(x)+u^\epsilon(x,y)\\
 f^\epsilon(x) = \sum_{k=2}^\infty f_{k}^\epsilon x^k, \quad & u^\epsilon(x,y) = \sum_{k=2}^\infty u_{k,1}^\epsilon x^k y+\sum_{k=1}^\infty\sum_{l=2}^\infty u_{k,l}^\epsilon x^k y^l.
\end{aligned}
\end{equation}
In particular, the following holds regarding the absolutely convergent power series expansions of $f^\epsilon$ and $u^\epsilon$ for all $\rho>0$ small enough: {Let $$D_1:=[0,\epsilon_0)\times \{0\le \vert x\vert <\rho\},\quad D_2:=[0,\epsilon_0)\times \{0\le \vert x\vert <\rho\}\times \{0\le \vert y\vert <\rho\},$$
and define
\begin{align}
 B := \sup_{(\epsilon,x)\in D_1} \vert f^\epsilon(x)\vert,\quad \mu: =\sup_{(\epsilon,x,y)\in D_2}\vert u^\epsilon(x,y)\vert.\label{eq:Bmudefn}
\end{align}}
Then 
\begin{align}
 \vert f_{k}^\epsilon\vert \le B \rho^{-k},\quad \vert u_{k,l}^\epsilon\vert \le  \mu \rho^{-k-l}\quad \mbox{and}\quad u_{k,1}^0= 0\quad \forall\,k,l\in \mathbb N,\,\epsilon \in[0,\epsilon_0)\label{eq:gkcond0app}
\end{align}
% where

% \KUK{here $\sup$ is taken over the domain defined by: $$\epsilon\in [0,\epsilon_0), \quad 0\le \vert x\vert <\rho\quad \mbox{and}\quad 0\le \vert y\vert <\rho.$$}
% where $R=$.
% for all $k,l$, 
% all $\epsilon \in [0,\epsilon_0)$ 
% for some $B>0,\mu>0$ and $\rho>0$, all independent of $\epsilon$. 
% % \begin{align}\label{eq:gcond}
% % g_{0,0}^\epsilon= g_{1,0}^\epsilon=g_{0,1}^\epsilon=g_{1,1}^\epsilon=0,
% % \end{align}
% for all $\epsilon\in [0,\epsilon_0)$. 
\end{theorem}

The proof of Theorem \ref{thm:nf} (available in  \appref{app:normalform}) is obtained by 
applying elementary transformations to the normal form in \cite[Theorem 2.2]{rousseau2005a}.

In the remainder of the paper, we will assume the following \rsp{conditions on $a^\epsilon$ and $g^\epsilon$}:
\begin{assumption}\label{assa}  The following inequality holds true:
		$$a^0:=\lim_{\epsilon\rightarrow 0} a^\epsilon >-2.$$
	\end{assumption}
%assume the following:
\begin{assumption}
 \label{deltacond}$B$ and $\rho>0$ are fixed and $\mu\ge 0$ in \eqref{eq:Bmudefn} is a parameter that is small enough (see details below). 
% Then it is assumed that  while $\mu>0$ is a small parameter.% in \eqref{eq:gkcond0app}. 
\end{assumption}

% We will write 
% \begin{align*}
%  p^\epsilon = \mu h^\epsilon,
% \end{align*}
% 
% 
% 

 Following \rsp{Hypothesis} \ref{deltacond}, 
 we will henceforth write
  \begin{align*}
  u^\epsilon = \mu h^\epsilon \quad\mbox{and}\quad u^\epsilon_{k,l} = \mu h^\epsilon_{k,l},%\quad \mbox{such that}\quad  \vert h^\epsilon_{k,l}\vert \le \rho^{-k-l}\quad \forall k,l\in \mathbb N,
 \end{align*}
 so that $g^\epsilon$ in \eqref{eq:normalform} becomes
%  \begin{equation}\label{eq:normalform}
% \begin{aligned}
%  \dot x &= (x-\epsilon)x,\\
%  \dot y &=-y(1+a^\epsilon x)+g^\epsilon(x,y),
% \end{aligned}
% \end{equation}
% of the analytic generic unfolding of the saddle node $\epsilon=0$. 
\begin{align}
 g^\epsilon(x,y) =: f^\epsilon(x) + \mu h^\epsilon(x,y), \label{eq:gexpansion}
\end{align}
% with $0<\mu\ll 1$ and
where 
\begin{align}\label{eq:fhexpansion}
f^\epsilon(x) &=\sum_{k=2}^\infty f^\epsilon_k x^k,\quad
h^\epsilon(x,y) =\sum_{k=2}^\infty h^\epsilon_{k,1} x^k y+\sum_{k=1}^\infty \sum_{l=2}^\infty h^\epsilon_{k,l} x^k y^l,
\end{align}
with
% on the polydisc $R:=\{\vert x\vert < \rho\}\times \{\vert y\vert <\rho\}$:
\begin{align}
 \vert f_k^\epsilon\vert \le B \rho^{-k},\quad \vert h^\epsilon_{k,l}\vert \le \rho^{-k-l} \quad \mbox{and}\quad h_{k,1}^0 =0\quad \forall\,k,l\in \mathbb N,\,\epsilon\in [0,\epsilon_0).\label{eq:gkcond0}
\end{align}
\rsp{The results below will be stated for \eqref{eq:normalform} with $g^\epsilon$ given by \eqref{eq:gexpansion} for $0<\mu\ll 1$ (in accordance with \rsp{Hypothesis} \ref{deltacond}).}

\rsp{In \thmref{main1}, when we treat $f_2^0$ as a parameter, we will fix a compact interval $I$ so that \eqref{eq:gkcond0app} holds (with $B>0$ large enough) for all $f_2^0\in I$.}

% for $\rho>0$ independent of $\epsilon$
% and
% \begin{align*}
%  B = \sup_{\epsilon\in [0,\epsilon_0),0\le \vert x\vert<\rho} \vert f^\epsilon\vert,\quad \mu = \sup_{\epsilon\in [0,\epsilon_0),(x,y)\in R} \vert \mu h^\epsilon\vert,
% \end{align*}
% by Cauchy's estimates. 
% The last property $h_{k,1}^0=0\,\forall\,k\ge 2$ in \eqref{eq:gkcond0} follows from Theorem \ref{thm:nf}. The quantity $a^\epsilon\in \mathbb R$ is an $\epsilon$-dependent constant; $a^0$ denotes its value at $\epsilon=0$.
% Here $g^\epsilon$ is analytic in a neighborhood $B$ of $(x,y)=0$, for all $\epsilon\in [0,\epsilon_0)$, $\epsilon_0>0$, and satisfies
% \begin{align}
%  g^\epsilon(x,y) = \mathcal O(x^2,x^2 y,y^2).\label{eq:geps0}
% \end{align}
% The dependency of all functions (including $a^\epsilon$) on $\epsilon\in [0,\epsilon_0)$ is continuous.
% See \appref{app:normalform} for further details of the derivation of this normal form from the one in \cite{rousseau2005a}. 

% \KUK{Somewhere: }
% 
% The reason for emphasizing $\mu$ in the expansion \eqref{eq:gexpansion} ($p^\epsilon=\mu h^\epsilon$ in comparison with \appref{app:normalform}) is that \textit{we will treat it as a parameter, independent of $\rho$}. In particular, we will make the following assumptions: 
% 
% \begin{assumption}
% 		\label{assa} $a^0>-2$.
% 	\end{assumption}
% %assume the following:
% \begin{assumption}
%  \label{deltacond} $B,\rho>0$ are fixed and $\mu>0$ in \eqref{eq:gexpansion} is small enough.
% \end{assumption}
The reference \cite{MR4445442} also assumes a condition like \rsp{Hypothesis} \ref{assa} (see \cite[Eq. 5.3]{MR4445442}) in the context of their specific rational example of an analytic unfolding, see  \cite[Eqs. (1.9)--(1.10)]{MR4445442}. On the other hand, a condition like \rsp{Hypothesis} \ref{deltacond}, which can also be viewed as \eqref{eq:gepsmu0}, does not appear in \cite{MR4445442}. We conjecture that our results are true without  \rsp{Hypotheses} \ref{assa} and \ref{deltacond} (and therefore hold true for any analytic and generic unfolding of a saddle-node), but leave this extension to future work. Whereas \rsp{Hypothesis} \ref{assa} seems relatively easy to relax, \rsp{Hypothesis} \ref{deltacond} requires extra work. We will discuss the matter further in Section \ref{sec:discussion}. 
\begin{remark}\remlab{proph1}
\rsp{Hypothesis} \ref{deltacond} is only an assumption on the nonlinearity in $y$. This follows from the last equality in \eqref{eq:gkcond0} and continuity with respect to $\epsilon$ (i.e. $h_{k,1}^\epsilon=o(1)$).
\end{remark}
\begin{remark}
Obviously, from \eqref{eq:Bmudefn} we have $\mu= \mathcal O(\rho^3)$ as $\rho\rightarrow 0$ in general (since $u^\epsilon$ starts with cubic terms) and in this sense one can achieve $\mu$ small by taking $\rho>0$ small. But this will not be helpful to us (and we do not expect it to be useful in general). {This is in contrast to arguments based upon Nagumo norms (see e.g. \cite{de2020a}), where the size of the domain can be used as a small parameter to obtain the appropriate contraction of a fixed-point formulation of Gevrey-properties of formal series.} At this stage, our approach in the present paper requires $B$ and $\rho>0$ fixed and $\mu>0$ small enough, as stated in \rsp{Hypothesis} \ref{deltacond}. %We refer to Section \ref{sec:discussion} for further discussion of this important aspect.
% Indeed, for $\epsilon=0$, we can remove the $y$-linear part $\sum_{k=2}^\infty h_{k,1}^0 x^k y$ of $h^0$ through an $x$-fibered diffeomorphism defined by
% \begin{align*}
%  (x,y)\mapsto \widetilde y = e^{-\psi(x)} y,\quad \psi(x) := \sum_{k=2}^\infty \frac{h_{k,1}^0 }{k-1} x^{k-1}\Longrightarrow x^2\phi'(x) = \sum_{k=2}^\infty h^0_{k,1} x^{k}.
% \end{align*}
% (In fact, the normal form can be simplied further by analytically linearizing the system within the invariant set $x=0$. In this way, $h_{0,l}^\epsilon \equiv 0$, but this will not be important to use.)
% This gives
% the following system 
% \begin{align*}
%  \dot x&= x^2,\\
%  \dot y &=-y(1+a^0x) + e^{-\psi(x)} f^0(x)+\sum_{k=0}
% \end{align*}

\end{remark}

Our first main result relates to the center manifold.
\begin{theorem}\thmlab{main1}
 Consider \eqref{eq:normalform} with $g^\epsilon$ given by \eqref{eq:gexpansion} for $\epsilon=0$:
 \begin{equation}\label{eq:normalformcm}
 \begin{aligned}
  \dot x &=x^2,\\
  \dot y &=-y(1+a^0x)+g^0(x,y),
 \end{aligned}
\end{equation} 
and suppose that \rsp{Hypotheses} \ref{assa} and \ref{deltacond} hold true. Let $W^c:\,y=m^0(x),\,m^0(0)=\rspp{\frac{dm^0}{dx}}(0)=0$, with $m^0$ defined in a neighborhood of $x=0$, denote the center manifold of $(x,y)=(0,0)$.
 Then there is a $\mu_0>0$ such that for all $0\le \mu< \mu_0$ the following statements hold true:
 \begin{enumerate}
 \item \rspp{There exists a number $S_\infty^0$, which depends upon the full jet of $g^0$, such that}:
 \begin{enumerate}
 \item \begin{align*}
  \rspp{\frac{(-1)^k}{\Gamma(k+a^0)} \frac{1}{k!}\frac{d^k m^0}{dx^k}(0) \rightarrow S_\infty^{\rsp{0}}} \quad \mbox{for}\quad k\rightarrow \infty.
 \end{align*}
\item \rspp{The center manifold $W^c$ is nonanalytic if $S_\infty^0\ne 0$}.
\end{enumerate}
  \item $S_\infty^0=S_\infty^0(f_{2}^0)$ is a $C^1$-function with respect to $f_2^0\in I$ (as well as all other parameters of the system), recall \eqref{eq:fhexpansion}, satisfying
  \begin{align*}
  \frac{\partial S_\infty^0 }{\partial f_{2}^0}(f_{2}^0)=\frac{1}{\Gamma(2+a^0)}+\mathcal O(\mu)\ne 0\quad \forall\,\mu\in [0,\mu_0),\quad \mu_0=\mu_0(I)>0.
\end{align*}
   \end{enumerate}  
\end{theorem}
The second statement shows that the center manifold being nonanalytic for \eqref{eq:normalform}, under the \rsp{Hypotheses} \eqref{assa} and \eqref{deltacond}, is a generic condition. {We exemplify this as follows:}
\begin{cor}\corlab{genf20}
 Suppose that the conditions of \thmref{main1} hold true, in particular $0\le \mu\ll 1$ so that $\frac{\partial S_\infty^0}{\partial f_2^0}( f_2^0)\ne 0$, $f_2^0\in I$, and suppose that the center manifold of \eqref{eq:normalformcm} is analytic ($\Rightarrow S_\infty^0(f_2^0)=0$). Then the center manifold of the perturbed system
\begin{equation}\nonumber
 \begin{aligned}
  \dot x &=x^2,\\
  \dot y &=-y(1+a^0x)+g^0(x,y)+qx^2,
 \end{aligned}
\end{equation} 
is nonanalytic for all $q\ne0$ small enough\rsp{.}  
\end{cor}
\begin{remark}\label{rem:comp}
\rsp{The property that $\frac{\partial S_\infty^0}{\partial f_2^0}( f_2^0)\ne 0$ in \corref{genf20} is a perturbative result (obtained by perturbing away from $\mu=0$) and it is (obviously) not expected to hold true in general ($\mu=1$). In fact, in Fig. \ref{fig:Sinfty0} we illustrate the locus 
\begin{align*}
 S_\infty^0(f_2^0,p)  = 0,
\end{align*}
computed numerically, see further details below, on the domain $(f_2^0,p)\in [0,10]\times [0,2]$ 
for the following $(f_2^0,p)$-family 
\begin{align}
 x^2 \frac{dy}{dx} = -y + f_2^0 x^2 + p\left(\frac{x^3}{1-x}+3 xy^2 + xy^3\right).\label{family}
\end{align}
Here $p$ plays the role of $\mu$, but as it includes parts of $f^0$ (the $y$-independent part) in the formulation of \eqref{eq:normalformcm} we give it a different name.
We clearly see that the curve has a fold (``far away'' from $p=0$) where necessarily $\frac{\partial S_\infty^0}{\partial f_2^0}=0$ (we find that $\frac{\partial S_\infty^0}{\partial p}<0$ at this point). At the same time, $\frac{\partial S_\infty^0}{\partial f_2^0}>0$ (on the lower branch) for all $0\le p\lesssim 1.94$, in agreement with the statement of \corref{genf20}. }

\begin{figure}[h!]
\begin{center}
\includegraphics[width=.54\textwidth]{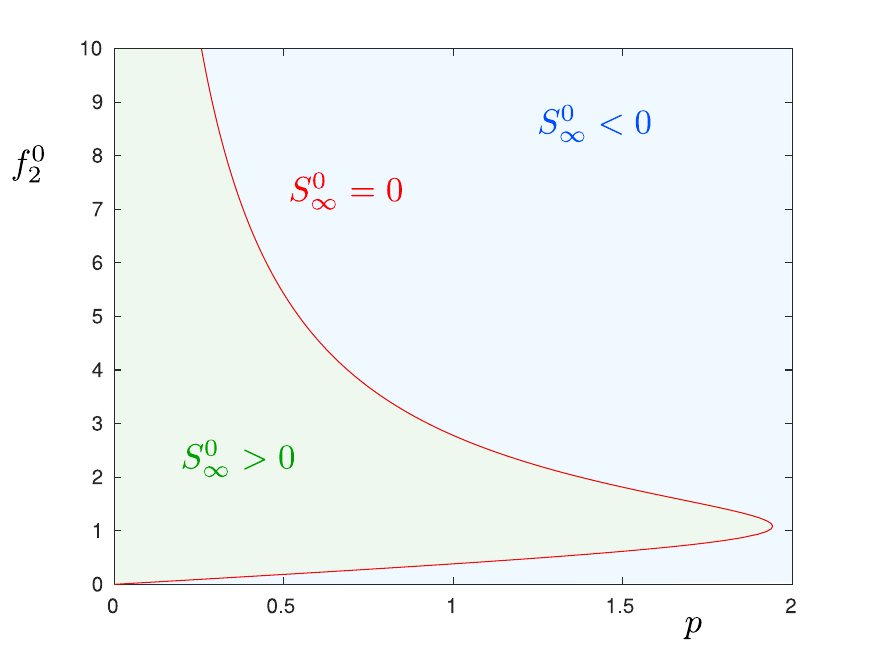}
\end{center}
\caption{\rsp{The locus $S_\infty^0(f_2^0,p)=0$ (in red) for the family \eqref{family}. The curve, which is computed numerically, see text for further details, has a fold point at $(p,f_2^0)\approx (1.94, 1.09)$, where $\frac{\partial S_\infty^0}{\partial f_2^0}$ changes sign, from $\frac{\partial S_\infty^0}{\partial f_2^0}>0$ on the lower branch (in agreement with the statement of \corref{genf20}) to $\frac{\partial S_\infty^0}{\partial f_2^0}<0$ on the upper one.}  }
\label{fig:Sinfty0}
% Remark: If $U$ smooth for $|u|<\delta$, the compact manifolds lie inside $|u|<\delta$ for large $n$
\end{figure}

\rsp{As an approximation of $S_\infty^0=S_\infty^0(f_2^0,p)$ we used $S_{100}^0$ (as a finite sum). In fact, we observed that $\vert S_{101}^0-S_{100}^0\vert \sim 10^{-16}$ (i.e at the order of machine precision). To determine $m_2^0,\ldots,m_{97}^0$ (that are necessary to determine $S_{100}^0$ as a finite sum, see \eqref{eq:S0k} and \eqref{eq:recmk} below), we used the recursion relation offered by \eqref{eq:recmk0}, starting from $m_1^0=0$.
}
% Notice that this general form can be obtained from \eqref{eq:S0k} by using \eqref{eq:recmk0}. To solve $S_{100}^0=S_{100}^0(f_2^0,p)=0$ we used bisection. }
\end{remark}
Our next result relates to the analytic weak-stable invariant manifold $W^{ws}$ of \eqref{eq:normalform}. To present this, we first write \eqref{eq:normalform} in the form (\rsp{upon elimating time}) 
\begin{align}
 x(x-\epsilon) \frac{dy}{dx} + y(1+a^\epsilon x) = g^\epsilon(x,y).\label{dydx}
\end{align}
For all $\epsilon^{-1}\notin \mathbb N$, $W^{ws}$ takes the graph form
\begin{align*}
 y = m^\epsilon(x) = \sum_{k=2}^\infty m_k^\epsilon x^k;
\end{align*}
with the last equality valid locally $x\in (-\delta^\epsilon,\delta^\epsilon)$, $\lim_{\epsilon\rightarrow 0} \delta^\epsilon= 0$. 
 In particular, $y=m^\epsilon(x)$ is a (locally defined) solution of \eqref{dydx}.
% We will need $\mu>0$ (not $B>0$) to be small enough in the following (in order to solve certain fixed-point problems by the contraction mapping theorem). Notice that we separate $g^\epsilon$ into the two terms:
% \begin{align*}
%  g^\epsilon(x,y) = f^\epsilon(x) + h^\epsilon(x,y),
% \end{align*}
% where $h^\epsilon(x,0)\equiv 0$ and $h^\epsilon=\mathcal O(\mu)$. 

Now, the blowup transformation defined by 
\begin{align}
 x = \epsilon \overline x,\quad y=\epsilon \overline y,\label{scalingxy}
\end{align}
for all $\epsilon>0$, separates the node and the saddle, so that the latter is at $\overline x=1$. \rsp{By applying the change of variables defined by \eqref{scalingxy}, \eqref{dydx} becomes}
\begin{align}
 \epsilon \overline x(\overline x-1) \frac{d\overline y}{d\overline x} + \overline y(1+\epsilon  a^\epsilon \overline x) = \epsilon^{-1} g^\epsilon(\epsilon \overline x,\epsilon \overline y),\label{eq:sp}
\end{align}
where 
\begin{align*}
 \epsilon^{-1} g^\epsilon(\epsilon \overline x,\epsilon \overline y) &=:\epsilon\overline f^\epsilon(\overline x)+\epsilon \mu \overline h^\epsilon(\overline x,\overline y),\\
 \overline f^\epsilon(\overline x) = \sum_{k=2}^\infty f_k^\epsilon \epsilon^{k-2} \overline x^k,\quad
 &\overline h^\epsilon(\overline x,\overline y)= \sum_{k=2}^\infty h_{k,1}^\epsilon \epsilon^{k-1} \overline x^k \overline y+\sum_{k=1}^\infty \sum_{l=2}^\infty h_{k,l}^\epsilon \epsilon^{k+l-2} \overline x^k \overline y^l.
\end{align*}
%  
%  
 %\sum_{k=2}^\infty f_k^\epsilon \epsilon^{k-1} \overline x^k + \mu \left(\sum_{k=2}^\infty h_{k,1}^\epsilon \epsilon^{k} \overline x^k \overline y+\sum_{k=1}^\infty \sum_{l=2}^\infty h_{k,l}^\epsilon \epsilon^{k+l-1} \overline x^k\overline y^l\right),
% \end{align*}
In these coordinates, \eqref{eq:sp} is a singularly perturbed system with respect to $0<\epsilon\ll 1$ and  $W^{ws}$ takes the following form $$\overline y = \overline m^\epsilon(\overline x):=\epsilon^{-1} m^\epsilon(\epsilon \overline x) = \sum_{k=2}^\infty \epsilon^{k-1} m_k^\epsilon \overline x^k,$$
where the last equality again holds true locally ($\overline x\in (-\epsilon^{-1} \delta^\epsilon,\epsilon^{-1} \delta^\epsilon)$).
In the language of Geometric Singular Perturbation Theory (GSPT) \cite{fen3,jones_1995}, the set $\{\overline y=0\}$ is a normally hyperbolic and attracting critical manifold \rsp{of \eqref{eq:sp} for $\epsilon=0$}. Therefore there is a (nonunique) slow manifold as a graph $\overline y = \mathcal O(\epsilon) $ over a compact subset $\overline x\in I$. This slow manifold only has finite smoothness (with respect to $\overline x$) in general, see \cite{fen3}. However, the unstable manifold $W^u$ of the saddle $(\overline x,\overline y)=(1,0)$ in the $(\overline x,\overline y)$-coordinates is an example of an analytic slow manifold of the following graph form:
\begin{align}\label{eq:Wugraph}
 W^u:\quad \overline y = \epsilon \overline H^\epsilon(\overline x), \quad \overline x\in (0,2],\quad \overline H^\epsilon(0^+)=0;
\end{align}
here $\overline H^\epsilon$ extends $C^k$-smoothly ($1\le k<\infty$, specifically not analytically, see \corref{main2b} item \ref{2new}) to $\overline x=0$ for all $0<\epsilon\ll 1$.
% \end{lemma}
We will also need the following lemma (which we prove in Section \ref{sec:nodefinal}). 
\begin{lemma}\lemmalab{lemma:Veps}
  Suppose that $\epsilon^{-1}\notin \mathbb N$, $0<\epsilon\ll 1$, and write 
  \begin{align*}
  \epsilon^{-1}=:N^\epsilon+\alpha^\epsilon, \quad N^\epsilon:=\lfloor \epsilon^{-1}\rfloor \quad \mbox{and}\quad \alpha^\epsilon\in (0,1).
  \end{align*}
  Then the following holds true. 
  \begin{enumerate}
  \item \label{Vepsprop} The series 
%   \begin{align*}
   \begin{align}\label{eq:Veps0}
 \overline V^\epsilon(\overline x):=\frac{\Gamma(\alpha^\epsilon)\Gamma(1-\alpha^\epsilon) }{\epsilon\Gamma(\epsilon^{-1})}\sum_{k=N^\epsilon}^\infty \frac{\Gamma(k+a^\epsilon)}{\Gamma(k+1-\epsilon^{-1})}\overline x^k,
\end{align}
is absolutely convergent for all $0\le \vert \overline x\vert <1$; in particular $\overline V^\epsilon(0)=0$ and 
\begin{align}\label{eq:Vepsprop}
\overline V^\epsilon(\overline x)>0,\quad \frac{d}{d\overline x}\overline V^\epsilon(\overline x)>0\quad \forall\, \overline x\in (0,1),
\end{align}
\item \label{Vepsprop2}  Lower bound:
\begin{align}\label{eq:Vepsprop2}
\overline V^\epsilon(\overline x) \ge 
 \epsilon \left(\frac{\overline x}{1- \overline x}\right)^{N^\epsilon+1}\quad \forall\, 0\le \overline x\le \frac34.                                                                                                                                                                
\end{align}
% for some $C_1>0$.
% (u^{-1}\delta_2\epsilon)^{1-\alpha^\epsilon}\quad \forall\, \overline x\in (0,1),
% \end{align}
%$\overline V^\epsilon(\overline x)\rightarrow \infty$ for $\overline x\rightarrow 1^-$. 
\item \label{Vepsblowup} At the same time, for any $0<\vert \overline x\vert <1$,
$$\vert  \overline V^\epsilon(\overline x)\vert \rightarrow \infty\quad \mbox{for}\quad \alpha^\epsilon \rightarrow 0^+\,\,\mbox{and}\, \,1^-.$$
\item \label{expansionV} Asymptotics for $\overline x=\mathcal O(\epsilon)$: Let $\overline x=\epsilon\overline x_2\in [-\epsilon\delta_2,\epsilon\delta_2]$, $\delta_2>0$ fixed. Then for all  $0<\epsilon\ll 1$, $\epsilon^{-1}\notin \mathbb N$:
\begin{equation}\label{eq:expansionV}
\begin{aligned}
\overline V^\epsilon(\epsilon\overline x_2) &=(1+o(1)) \Gamma(\alpha^\epsilon) (N^\epsilon)^{a^\epsilon+1-\alpha^\epsilon} (\epsilon \overline x_2)^{N^\epsilon} \\
 &\times \left(1+ \frac{\overline x_2}{1-\alpha^\epsilon}  \left[1+\overline x_2\int_0^1 e^{(1-v)\overline x_2}  v^{1-\alpha^\epsilon} dv+o(1)\right]\right),
%  \overline V^\epsilon(\overline x) = (1+o(1))\Gamma(\alpha^\epsilon) (N^\epsilon)^{a^\epsilon+1-\alpha^\epsilon}\overline x^{N^\epsilon}\bigg(1 +\mathcal O(1)\Gamma(1-\alpha^\epsilon) N^\epsilon\overline x\bigg)\quad \forall\,0\le \vert \overline x\vert \le \delta_2\epsilon,
\end{aligned}
\end{equation}
with each $o(1)$ being uniform with respect to $\alpha^\epsilon\in (0,1)$. %Moreover, the quantity $\mathcal O(1)\ge c>0$ is strictly positive for all $\alpha^\epsilon\in (0,1)$. 
\end{enumerate}
\end{lemma}

% For our analysis, we will need that $\mu>0$ defined in \eqref{eq:gkcond0} is sufficiently small. 
% Seeing that $g^\epsilon$ is analytic in a neighborhood $U$ (independent of $\epsilon$) of $(x,y)=(0,0)$, we have 
% \begin{align*}
%  g^\epsilon(x,y)=\sum_{k,l}g^\epsilon_{k,l} x^k y^l, \quad \vert
%   g^{\epsilon}_{k,l}\vert \le \xi \rho^{-k-l},
%  \end{align*}
%  for some $\xi>0$ and $\rho>0$. 
% % \end{align*}

% \subsection{Main result}
% In this paper, we will assume the following:
% \begin{enumerate}
%  \item \label{assa} $a^0>0$.
%  \item \label{deltacond} $\mu>0$ in \eqref{eq:gkcond0} is small enough.
% \end{enumerate}
% The reference \cite{MR4445442} also assumes \rsp{Hypothesis} \ref{assa}, but we do believe that our approach can be extended to $a^0< 0$. We will leave this to future work. Based on \cite{MR4445442}, which do not require an assumption of the form \rsp{Hypothesis} \ref{deltacond} for their specific nonlinearity, we believe that the assumption \rsp{Hypothesis} \ref{deltacond} can be relaxed. In future work, we will pursue this. 

% In the present paper, by assuming \rsp{Hypothesis} \ref{assa} and \rsp{Hypothesis} \ref{deltacond} we show that we can define a number $S_\infty^0$ for $\epsilon=0$, which depends on the full jet of $g^0(\cdot,\cdot)$. If this number is nonzero then the center manifold is non-analytic and it will also be fundamental (as in \cite{MR4445442}) to assume that this quantity is nonzero.  

% Define 
% 
% It is absolutely convergent for all $\epsilon^{-1} \notin \mathbb N$. 

\begin{figure}[h!]
\begin{center}
\includegraphics[width=.4\textwidth]{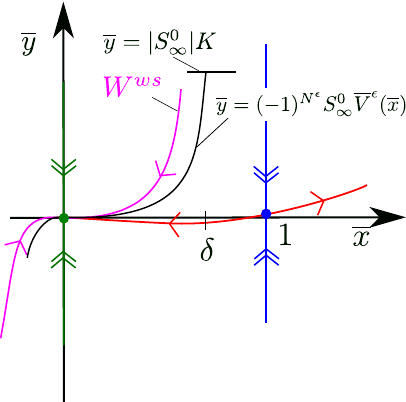}
\end{center}
\caption{Phaseportrait of \eqref{eq:sp} for $\epsilon>0$, $\epsilon^{-1}\notin \mathbb N$ (and $(-1)^{N^\epsilon}S_\infty^0>0$); please compare with Fig. \ref{node}. \thmref{main2} says that if $S_\infty^0\ne 0$ then we can track $W^{ws}$ (in magenta) by the graph $\overline  y = (-1)^{N^\epsilon}S_\infty^0 \overline V^\epsilon(\overline x)$, see also \corref{main2b} and further details in \thmref{main2}. }
\label{barnode}
% Remark: If $U$ smooth for $|u|<\delta$, the compact manifolds lie inside $|u|<\delta$ for large $n$
\end{figure}

% ($V^\epsilon(x)$ in \eqref{eq:mepsExpansion0} is $\epsilon \overline V^\epsilon (\epsilon^{-1} x)$, but we will perform the tracking of $y=m^\epsilon(x)$ using the scaled coordinates $(\overline x,\overline y)$.)
% In further details, 
Our main result on the analytic weak-stable manifold then takes the following form (see Fig. \ref{barnode}).
\begin{theorem}\thmlab{main2}% \label{thm:main0}
Fix $K>0$, $\delta_2>0$, $0<\upsilon\ll K$ and consider \eqref{eq:sp} with $g^\epsilon$ given by \eqref{eq:gexpansion}, satisfying \rsp{Hypotheses} \ref{assa} and \ref{deltacond}. Then the quantity $S_\infty^0$ from \thmref{main1} is well-defined. 
We suppose that
 \begin{align}
  S_\infty^0\ne 0, \label{assSinf}
 \end{align}
 so that the center manifold is nonanalytic.

%  \begin{enumerate}
%   \item The analytic weak-stable manifold and the unstable manifold of the saddle do not intersect.
%   \item Moreover, 
%  \end{enumerate}
% 
%  
 Now, consider the convergent series $\overline V^\epsilon$ defined in \eqref{eq:Veps0}. Then the following holds for all $0<\epsilon\ll 1$, $\epsilon^{-1}\notin \mathbb N$: Let $W^{ws}:\,\overline y=\overline m^\epsilon(\overline x)$, with $\overline m^\epsilon$ defined in a neighborhood of the origin, denote the analytic weak-stable manifold in the $(\overline x,\overline y)$-coordinates, see \eqref{scalingxy}, and 
% \begin{enumerate}[resume*]
%  \item 
let $I\subset \left[-\delta_2\epsilon,\frac34\right]$ be an interval so that
 \begin{align}
\vert \overline V^\epsilon(\overline x)\vert\le K\quad \forall\, \overline x\in I.\label{VcondK}
\end{align} 
Then $I\subset \operatorname{domain}(\overline m^\epsilon)$ and
\begin{align}
\vert \overline m^\epsilon (\overline x)-(-1)^{N^\epsilon}  S_\infty^0 \overline V^\epsilon(\overline x)\vert \le \upsilon \quad \forall\,\overline x\in I. \label{eq:Wwsapp}
\end{align}
% \end{enumerate}
 
\end{theorem}
  In other words, when \eqref{assSinf} holds true, then by taking $0<\epsilon\ll 1$, \textit{we can  track $W^{ws}:\overline y=\overline m^\epsilon(\overline x)$ through $\overline y=(-1)^{N^\epsilon}  S_\infty^0\overline V^\epsilon(\overline x)$}. 
  Moreover, we have the following result, which we illustrate for $S_\infty^0>0$ in \figref{node_cases}; the weak manifold has to be reflected about the $x$-axis for $S_\infty^0<0$.

  \begin{cor}\corlab{main2b}
    Fix $c>0$ small enough, suppose that \rsp{Hypotheses} \ref{assa} and \ref{deltacond} hold true and that $S_\infty^0\ne 0$. Put $s=\operatorname{sign}(S_\infty^0)$ and let $W^{ws}$ denote the analytic weak-stable manifold. Then the following holds true regarding the position of $W^{ws}$ for all $N^\epsilon=\lfloor \epsilon^{-1}\rfloor\gg 1$: 
    
%     $W^{ws}$ does not intersect $W^u$ and the following holds regarding the position of $W^{ws}$:
    \vspace{0.2cm}
    
  \noindent  \textbf{Intersections of $W^{ws}$ with $\{y=\pm c\}$ for $x>0$}:
  
    \begin{enumerate}
    \item \label{2new} $W^{ws}$ does not intersect $W^u$. More precisely, we have the following:
    \begin{enumerate}
    \item \label{1} Suppose that $N^\epsilon$ is even. Then $W^{ws}$ intersects $\{y=sc\}$ for $x>0$.  
%      \item For all $0<1-\alpha^\epsilon\ll 1$, $W^{ws}$ intersects $y=sc$ for $x>0$ and $y=-sc$ for $x<0$.
%     \end{enumerate}
\item \label{2}
   Suppose that $N^\epsilon$ is odd. Then $W^{ws}$ intersects $\{y=-sc\}$ for $x>0$.  
   \end{enumerate}
%      \item For all $0<1-\alpha^\epsilon\ll 1$, $W^{ws}$ intersects $y=-sc$ for both $x>0$ and $x<0$.
%     \end{enumerate}
    \end{enumerate}

 \noindent    \textbf{Intersections of $W^{ws}$ with $\{y=\pm c\}$ for $x<0$}:
   
   \vspace{0.2cm}
    
   Define
    \begin{align}\label{eq:underlinealpha}
     \underline \alpha(N^\epsilon) : = (N^\epsilon)^{a^0-N^\epsilon},\quad 1-\overline \alpha(N^\epsilon) :=(N^\epsilon)^{a^0-1-N^\epsilon}. 
    \end{align}
    \begin{enumerate}[resume*]
    \item \label{3} Suppose that $N^\epsilon$ is even.  Then the following holds:
    \begin{enumerate}
     \item \label{3a} $W^{ws}$ intersects $\{y=sc\}$ for $x<0$ for all $0<\alpha^\epsilon\le \underline \alpha(N^\epsilon) $.
     \item \label{3b} $W^{ws}$ intersects $\{y=-sc\}$ for $x<0$ for all $0<1-\alpha^\epsilon\le 1-\overline \alpha(N^\epsilon)$.
    \end{enumerate}
\item \label{4}
    Suppose that $N^\epsilon$ is odd. Then the following holds:
    \begin{enumerate}
     \item \label{4a} $W^{ws}$ intersects $\{y=sc\}$ for $x<0$ for all $0<\alpha^\epsilon\le \underline \alpha(N^\epsilon)$.
     \item \label{4b} $W^{ws}$ intersects $\{y=-sc\}$ for $x<0$ for all $0<1-\alpha^\epsilon\le  1-\overline \alpha(N^\epsilon)$.
    \end{enumerate}
    \end{enumerate}

%      
%     \end{itemize}

%    \begin{align*}
%     \alpha^\epsilon = 
%    \end{align*}
%    Then the analytic weak-stable manifold $y=m^\epsilon(x)$ with $\epsilon=N^\epsilon+\alpha^\epsilon$ intersects $y=\pm c$.
  \end{cor}
  \begin{proof}
  \rspp{We first consider the statements in item \ref{2new} regarding the intersections of $W^{ws}$ with $\{y=\pm c\}$ for $x>0$ (proving items \ref{1} and \ref{2}). } 
  We let $K>0$ be large enough and take $0<\upsilon\ll K$ small enough. %and first consider the intersections of $W^{ws}$ with $\{y=\pm c\}$ for $x>0$ (proving items \ref{1} and \ref{2}). 
   We let $0<\epsilon\ll 1$ be so that $\overline V^\epsilon(\frac34)>K$, see \eqref{eq:Vepsprop2}. Then since $\overline V^\epsilon(\overline x)$ is an increasing function of $\overline x$, see \eqref{eq:Vepsprop},  $\delta\in (0,\frac34)$ defined by the equation
  \begin{align*}
   \overline V^\epsilon(\delta)=K,
  \end{align*}
  is uniquely determined.
% Here we have items \ref{Vepsprop} and  \ref{Vepsprop2} of \lemmaref{lemma:Veps}. 
We then apply \thmref{main2} with $I=[0,\delta]$. In particular, from \eqref{eq:Wwsapp} we conclude that $W^{ws}$ intersects $\{\overline y=\pm \frac12 S_\infty^0K\}$ for $\overline x\in (0,\delta)$ when $N^\epsilon$ is even/odd, respectively. From $\{\overline y = \pm \frac12 S_\infty^0 K\}$, we undo the scaling \eqref{scalingxy} and return to \eqref{eq:normalform} and apply the backward flow, see Fig. \ref{barnode}. This completes the proof of items \ref{1} and \ref{2}. %It also shows that $W^{ws}$ and $W^s$ have an empty intersection.
\rsp{To complete the proof of item \ref{2new}, we recall that $W^u$ and $W^{ws}$ (away from the singularities) are orbits of planar systems and therefore if branches of these manifold intersect, then they coincide. We have that $\overline y =\mathcal O(\epsilon)$ along $W^u$ for $\overline x\in [0,1]$, recall \eqref{eq:Wugraph}. Therefore $W^u$ does not intersect $\{y=\pm c\}$ within $\overline x\in [0,1]$ for all $0<\epsilon\ll 1$ and consequently $W^u$ and $W^{ws}$ do not coincide. Hence $W^s\cap W^{ws}=\emptyset$ as desired.}

We then turn to the intersection of $W^{ws}$ with $\{y=\pm c\}$ for $x<0$.  
% Since the unstable manifold $W^{u}$ of the saddle is a slow manifold of \eqref{dbarydbarx0}, remaining in an $\epsilon$-neighborhood of the normally hyperbolic critical manifold $\bar y=0$, it follows that $W^{ws}$ and $W^u$ never intersect. 
     For this purpose, we again let $K>0$ be large enough, put $\delta_2=1$ (for concreteness), take $0<\upsilon\ll K$ small enough, $I=[-\delta,0]$ with $0<\delta\le \epsilon$
%      
%      and then use $(-1)^{N^\epsilon}S_\infty \overline V^\epsilon$ as an approximation of $\overline m^\epsilon$:
%    \begin{align}\label{eq:mepsV}
%     \vert \overline m^\epsilon(\overline x)-(-1)^{N^\epsilon}S_\infty \overline V^\epsilon\vert \le \upsilon,
%    \end{align}
% for $0\le \vert \overline x\vert\le \delta_2\epsilon$.
   and use the expansion  \eqref{eq:expansionV} for $-\delta_2\epsilon \le \overline x\le 0$ to obtain the following for all $N^\epsilon\gg 1$: Consider $\underline \alpha(N^\epsilon)$ and $\overline \alpha(N^\epsilon)$ defined in \eqref{eq:underlinealpha}. Then for any $0<\alpha^\epsilon\le \underline \alpha(N^\epsilon)\ll 1$, $\overline V^\epsilon(\epsilon \overline x_2)$ is given by
   \begin{align}
   \frac{1}{\alpha^\epsilon} (N^\epsilon)^{a^\epsilon+1-N^\epsilon} \overline x_2^{N^\epsilon}e^{\overline x_2},\label{eq:Veps1}
   \end{align}
   to leading order, whereas for any $0<1-\alpha^\epsilon\le 1-\overline \alpha(N^\epsilon)\ll 1$, $\overline V^\epsilon(\epsilon \overline x_2)$ is given by 
  \begin{align}
   \frac{1 }{1-\alpha^\epsilon} (N^\epsilon)^{a^\epsilon-N^\epsilon} \overline x_2^{N^\epsilon+1} e^{\overline x_2},\label{eq:Veps2}
   \end{align}
   to leading order. We have here used \eqref{eq:Gamma0}, $$
 \left[  1+\overline x_2 \int_0^1 e^{(1-v)\overline x_2}v dv\right] = \frac{e^{\overline x_2}-1}{\overline x_2}\quad \mbox{and}\quad 
  \left[ 1+\overline x_2\int_0^1 e^{(1-v)\overline x_2} dv\right] = e^{\overline x_2}.$$
%    \begin{align}\label{eq:condalphaeps}
%    (N^\epsilon)^{\alpha^\epsilon}\approx 1 \quad \mbox{and}\quad (N^{\epsilon})^{1-\alpha^\epsilon}\approx 1,
%    \end{align} respectively.  
   In both cases (\eqref{eq:Veps1} and \eqref{eq:Veps2}), there are remainder terms that we can assume are bounded by $\upsilon>0$, uniformly with respect to $\alpha^\epsilon$ (whenever \eqref{eq:Veps1} and \eqref{eq:Veps2} do not exceed $K$ in absolute value); this characterization will be adequate for our purposes.
   
   We now further claim that for  any $0<\alpha^\epsilon\le \underline \alpha(N^\epsilon)$, then \eqref{eq:Veps1} with $\overline x_2=-1$ exceeds $K>0$ in absolute value. To show this, we just estimate    
   \begin{align*}
   \left|\frac{1}{\alpha^\epsilon} (N^\epsilon)^{a^\epsilon+1-N^\epsilon} (-1)^{N^\epsilon}e^{-1}\right| \ge \frac{1}{\underline \alpha(N^\epsilon)} (N^\epsilon)^{a^\epsilon+1-N^\epsilon} e^{-1}= (N^\epsilon)^{1+a^\epsilon-a^0}e^{-1}\ge (N^\epsilon)^{\frac12}e^{-1}>K,
   \end{align*}
   using that $\vert a^\epsilon-a^0\vert\le \frac12$ for all $N^\epsilon\gg 1$. 
 A similar result holds for \eqref{eq:Veps2} for all $0<1-\alpha^\epsilon\le 1-\overline \alpha(N^\epsilon)$. We leave out the details in this case. 
%    and this leads to two lower bounds of \eqref{eq:Veps1} and \eqref{eq:Veps2} in absolute value that only depend upon $\alpha^\epsilon$ and $N^\epsilon$: 
%    \begin{align}
%    \frac{1}{\alpha^\epsilon} (N^\epsilon)^{a^0-\nu+1-N^\epsilon} \vert \overline x_2\vert^{N^\epsilon}\frac{e^{\overline x_2}-1}{\overline x_2},\label{eq:Veps11}
%    \end{align}
%    and
% %    to leading order, whereas for any $0<1-\alpha^\epsilon\ll (N^\epsilon)^{-1}$, $\overline V^\epsilon(\epsilon \overline x_2)$ is given by 
%   \begin{align}
%    \frac{1 }{1-\alpha^\epsilon} (N^\epsilon)^{a^0-\nu-N^\epsilon} \vert \overline x_2\vert^{N^\epsilon+1} e^{\overline x_2},\nonumber %\label{eq:Veps21}
%    \end{align}
%    respectively. 
     %The remainder terms, due to  can be assumed to be bounded by $\upsilon>0$, $\upsilon\ll K$, 
     
%       We then set $\overline x_2=-1$ and equate the resulting expressions to $K>0$. This gives two equations for $\alpha^\epsilon$ with solutions $\alpha^\epsilon = \underline \alpha(N^\epsilon)$ and $\alpha^\epsilon = \overline \alpha(N^\epsilon)$, respectively, given by $$\underline\alpha^\epsilon = \frac{1}{K}(N^\epsilon)^{a^0-\nu+1-N^\epsilon} (1-e^{-1}) \ll (N^\epsilon)^{-1} \quad \mbox{and}\quad 1-\overline\alpha^\epsilon=\frac{1}{K}(N^\epsilon)^{a^0-\nu-N^\epsilon} e^{-1}\ll (N^\epsilon)^{-1}.$$ %Notice that these values of $\alpha^\epsilon$ satisfy \eqref{eq:condalphaeps}. 

            Consider now items \ref{3a} and \ref{4a} regarding $\alpha^\epsilon\rightarrow 0$. We then have by \eqref{eq:Veps1} (which is continuous and monotone with respect to $\overline x_2\in [-1,0)$) and \eqref{eq:Wwsapp} that for any $0<\alpha^\epsilon\le \underline\alpha(N^\epsilon)$, the equation $\vert \overline m^\epsilon(\overline x)\vert =\frac12 \vert S_\infty^0\vert K$ has a solution $\overline x_-\in (-\epsilon,0)$. The sign of $\overline m^\epsilon(\overline x_-)$ is determined by $(-1)^{N^\epsilon} s \overline x_-^{N^\epsilon}$, cf. \eqref{eq:Wwsapp} and \eqref{eq:Veps1}. From $\{\overline y = \pm \frac12 \vert S_\infty^0\vert K\}$, we undo the scaling \eqref{scalingxy} and return to \eqref{eq:normalform}. Then the proof of items \ref{3a} and \ref{4a} is completed by using the backward flow. Indeed, $W^{ws}$ aligns itself with one side of $W^{ss}$ in this case and we can therefore just use $W^{ss}$ as a guide for the backward flow up until $W^{ss}$'s transverse intersection with $\{y=\pm c\}$, see Fig. \ref{barnode}. The case $\alpha^\epsilon\rightarrow 1$ (items \ref{3b} and \ref{4b}) is similar and we therefore leave out further details.

%    and all $0<\alpha^\epsilon\le \underline\alpha^\epsilon(N^\epsilon,a^\epsilon)$
%    Here by dominated we mean that 
  \end{proof}
  
%     \textbf{.}

\begin{figure}[h!]
\begin{center}
\includegraphics[width=.75\textwidth]{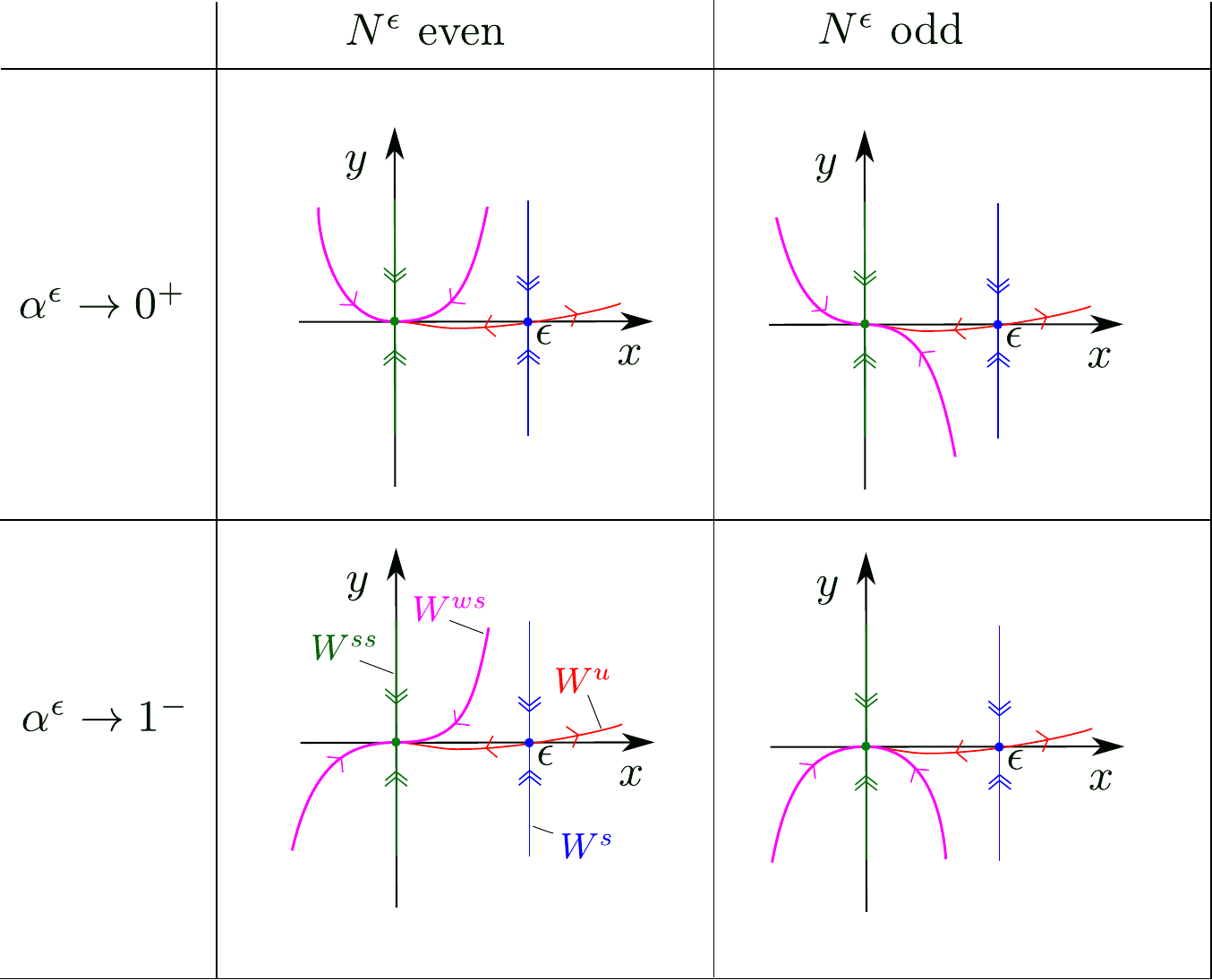}
\end{center}
\caption{Illustration of the results of \thmref{main2}, see also \corref{main2b}. The strong stable manifold $W^{ss}$ in green, the analytic weak-stable manifold \rsp{$W^{ws}$} in magenta, the stable manifold of the saddle $W^s$ in blue, and finally the unstable manifold of the saddle $W^u$ in red. The diagram assumes $S_\infty^0>0$; if $S_\infty^0<0$ then the diagram should be reflected about the $x$-axis. The analytic weak-stable manifold $W^{ws}$ ``flaps'' on the $x<0$-side of the node as $\alpha^\epsilon$ transverses $(0,1)$, aligning on $y>0$ or $y<0$ with $W^{ss}$ as either $\alpha^\epsilon\rightarrow 0^+$ or $\alpha^\epsilon\rightarrow 1^-$. On the $x>0$-side, $W^{ws}$ remains on one side of $W^u$ for all $\alpha^\epsilon \in (0,1)$ and only ``flaps'' (discontinously) when $N^\epsilon\gg 1$ changes parity. In particular, $W^{ws}$ and $W^u$ do not intersect.}
\figlab{node_cases}
% Remark: If $U$ smooth for $|u|<\delta$, the compact manifolds lie inside $|u|<\delta$ for large $n$
\end{figure}

\subsection{\rspp{Overview} }\label{sec:overview}
We prove \thmref{main1} in Section \ref{sec:cm}. \thmref{main2} is proven in Section \ref{sec:node}, see also Section \ref{sec:nodefinal} where \lemmaref{lemma:Veps} is proven. The strategy of the proof  of \thmref{main2} follows \cite{MR4445442} insofar that we write $\overline y=\overline m^\epsilon(\overline x)$ as a finite sum $\overline y=\sum_{k=2}^{N^\epsilon} \overline m_k^\epsilon \overline x^k$, up until ``before the resonance'', plus a remainder $\overline M^\epsilon(\overline x)=\mathcal O(\overline x^{N^\epsilon+1})$ that we solve by setting up a fixed-point equation using an integral operator $\mathcal T^\epsilon$, see Lemma \ref{lemma:TOp}. A main difficult lies in estimating the growth of coefficients in the series expansion of $\overline g^\epsilon$ when composed with the finite sum $\overline y=\sum_{k=2}^{N^\epsilon} \overline m_k^\epsilon \overline x^k$ (with the number of terms going unbounded as $\epsilon\rightarrow 0$). This is covered by the novel  Lemma \ref{lemma:yklNew} (which does not depend upon \rsp{Hypothesis} \ref{deltacond}). Our treatment of $\overline M^\epsilon$ is also novel (and also does not rely on \rsp{Hypothesis} \ref{deltacond}) insofar that we view the integral operator $\mathcal T^\epsilon$ as a bounded operator on a certain Banach space $\mathcal D^\epsilon_\delta$ of analytic functions $H=H(\overline x)$ with $H(\overline x)=\mathcal O(\overline x^{N^\epsilon+1})$, see \eqref{Cepsnorm}. \rspp{We believe that these novel aspects are crucial for making conclusions regarding $W^{ws}\cap W^s$. In particular, such conclusions cannot be made from the results of \cite{MR4445442} on their specific nonlinearity (to the best of our judgement). } 
% We conclude the paper in Section \ref{sec:discussion}. 

\section{The center manifold $W^c$: The proof of \thmref{main1}}\label{sec:cm}
%We first consider $\epsilon=0$ of \eqref{dydx}
In this section, we consider $\epsilon=0$ and \eqref{eq:normalformcm} in the equivalent form
\begin{align}
x^2 \frac{dy}{dx} + y(1+a^0 x) = g^0(x,y),\label{dydx0}
\end{align}
where 
\begin{align*}
 g^0(x,y) = f^0(x)+\mu h^0(x,y)= \sum_{k=2}^\infty f_k^0 x^k+\mu  \sum_{k=1}^\infty \sum_{l=2}^\infty h_{k,l}^0 x^k y^l,
\end{align*}
cf. \eqref{eq:fhexpansion} and \eqref{eq:gkcond0}.
% We know that 
% There exists a $C^\infty$ center manifold $y=m^0(x)$ of this system. It is known that it is in general not analytic \cite{}. %&For example, if $a^0=0$ and $g^0(x,y)=x$ the
Let 
\begin{align}
 \widehat m^0(x) = \sum_{k=2}^\infty m^0_k x^k,\label{m0}
\end{align}
denote the formal series expansion of the center manifold $y=m^0(x)$.
% \begin{lemma}\label{lemma:mk0}
% Suppose first that $$g^0(x,y) =\sum_{k=2}^\infty g_k^0 x^k,\quad g_k^0\in \mathbb R,$$ and define
%  \begin{align}\label{wk0}
%   w^0_k: = \Gamma(k+a^0),%(k-1)!\Pi_{j=1}^{k-1} \left\{1+\frac{a^0}{j}\right\},
%  \end{align}
%  and 
%  \begin{align*}
%   S^0_k: = \sum_{j=2}^k \frac{(-1)^j g_j^0}{w^0_j}.
%  \end{align*}
%  % \end{lemma}
% Then the 
% $m^0_k$'s in \eqref{m0} are given by
% \begin{align*}
%  m^0_k  = (-1)^k S^0_k  w^0_k.
% \end{align*}
%  \end{lemma}
% \begin{proof}
% By inserting \eqref{m0} in \eqref{dydx0}, it is elementary to obtain the following recursion relation for $m^0_k$:
%  \begin{align*}
%   m^0_k +(k-1+a^0) m^0_{k-1} = g^0_{k},\quad k\ge 2.
%  \end{align*}
%  Here we define $m^0_1=0$. Next, we put 
% 
% \end{proof}
% 
% \subsection{Growth properties of $m_k^0$}
% % \fbox{still working on this section}
% We now study the formal series \eqref{m0} and the growth properties of $m_k^0$. It is well-known, see e.g. \cite{}, that these are Gevrey-1, i.e.
% \begin{align*}
%  \vert m_k^0\vert  \le ER^k k!
% \end{align*}
% for some $E>0,R>0$. We will, however, present a more detailed description which (following \cite{MR4445442}) will be useful in our tracking of the analytic weak-stable manifold.
% 
% 
% For this purpose, we first replace $x$ by $\delta \widetilde x$. Then $m_k^0\mapsto \widetilde m_k^0:=\delta^k m_k^0$ and
% \begin{align*}
% m_k^0=  (-1)^k \delta^k S^0_k  w_k^0,
% \end{align*}
% \fbox{should not do this rescaling, just built $\delta$ into the norm}
% upon dropping the tildes again. Next, we define the norm
We define 
\begin{align}
 w_k^0:=\Gamma(k+a^0)\quad \forall\,k\ge 2,\label{eq:wk0}
\end{align}
and a norm
\begin{align}
 \Vert y\Vert = \sup_{k\ge 2} \frac{\vert y_k\vert }{w_{k}^0},\label{gevrey}
\end{align}
% \fbox{ add $\delta^k$ to norm instead?}
on the space \rsp{of formal series}
\begin{align*}
\rsp{\mathcal D^0=\left\{y=\sum_{k=2}^\infty y_k x^k\,:\,y_k \in \mathbb R\,\,\forall\,k\ge 2\right\}.}
\end{align*}
Notice that \eqref{eq:wk0} is well-defined by virtue of \rsp{Hypothesis} \ref{assa} and that $\mathcal D^0$ is a Banach space (due to the sequence space $l^\infty$ being Banach). For any $C>0$, we also define 
\begin{align}
\mathcal B^{C}: = \{y\in \mathcal D^0\,:\,\Vert y\Vert\le C\},\label{eq:ballC}
\end{align}
as the closed ball of radius $C$.
Moreover, for any $y(x)=\sum_{k=2}^\infty y_k x^k \in \mathcal D^0$, \rsp{the composition $g^0(x,y(x))$ of $y(x)$ with the analytic function $g^0$ is itself a formal series. The results below (see Proposition \ref{prop:boundgk0}) show that the associated operator
\begin{align*}
 \mathcal G^0:\mathcal D^0\rightarrow \mathcal D^0,\quad \mathcal G^0[y](x) = g^0(x,y(x)) = \sum_{k=2}^\infty \mathcal G^0[y]_k x^k,
\end{align*}
is well-defined. The expression also defines $ \mathcal G^0[y]_k$.} 
% Notice that 
% \begin{align}
%  \mathcal G^0[y]_k = \frac{1}{k!} \frac{\partial^k }{\partial x^k} (g^0(x,y(x)))\bigg\vert_{x=0},\eqlab{Gk0formula}
% \end{align}
% when $y$ and $\mathcal G[y]$ are convergent. In the general case, $\mathcal G^0[y]_k$ is determined by Cauchy's product formula (which is consistent with \eqref{Gk0formula} when differentiation of $y$ is termwise). 
% 
% composition with $g(\cdot,y(\cdot))$ =\sum_{k=2}^\infty \rsp{\mathcal G^0_k}[y]x^k\in \mathcal D^0$ with 
% 
% $\mathcal H^0[y]_k$ by 
% \begin{align*}
%  g^0(x,y(x)) = \sum_{k=2}^\infty \mathcal H^0[y]_k x^k;
% \end{align*}
\rsp{$\mathcal H^0[y]$ and $\mathcal H^0[y]_k$ are similarly defined through the composition $h^0(x,y(x))$ of $y(x)$ with $h^0$}:
\begin{align*}
 \rsp{\mathcal H^0[y](x) = h^0(x,y(x)) = \sum_{k=2}^\infty \mathcal H^0[y]_kx^k.}
\end{align*}
\rsp{By  \eqref{eq:gexpansion}, we have $\mathcal H^0[y]_k=0$ for $k=2,3$ and $4$ and therefore}
\begin{align}
\rsp{\begin{cases}
\mathcal G^0[y]_2&=f^0_{2}, \\
\mathcal G^0[y]_3&= f^0_{3},\\
\mathcal G^0[y]_4 &= f^0_{4},\\
 \mathcal G^0[y]_k &= f^0_{k} + \mu \mathcal H^0[y]_k,\quad k\ge 5.
 \end{cases}}\label{gk0}
\end{align} 
\rspp{Finally, for any $l\in \mathbb N$ and any $y\in \mathcal D^0$, we define $(y^l)_k$ as the coefficients of $y^l$:
\begin{align*}
 y(x)^l = :\sum_{k=2l}^\infty (y^l)_k x^k.
\end{align*}
One can express $(y^l)_k$ in terms of $y_{2},\ldots,y_{k-2}$ using the multinomial theorem, but we will not make use of this.
It will also follow from Proposition \ref{prop:boundgk0} that $y^l\in \mathcal D^0$.}
\begin{lemma}\label{lemma:h0k}
The following holds
\begin{align}\label{gk12}
\rsp{\mathcal H^0[y]_k} &=\sum_{l=2}^{\lfloor \frac{k-1}{2}\rfloor} \sum_{j=2l}^{k-1} h_{k-j,l}^0 (y^l)_j\quad \forall\,k\ge 5.
\end{align}
 
\end{lemma}
\begin{proof}
We use the expansion of $h^0$ in \eqref{eq:fhexpansion} and Cauchy's product rule:
\begin{align}\label{eq:cauchyproduct}
 \sum_{k=0}^\infty q_k \sum_{l=0}^\infty p_l = \sum_{k=0}^\infty\sum_{j=0}^k q_{k-j} p_{j}.
\end{align}
We have
\begin{align*}
 \rsp{\mathcal H^0[y](x)}=h^0(x,y(x))=\sum_{l=2}^\infty \sum_{k=1}^\infty  h_{k,l}^0 x^k y(x)^l &= \sum_{l=2}^\infty\left( \sum_{k=1}^\infty   h_{k,l}^0 x^k \right) \left(\sum_{j=2l}^\infty (y^l)_j x^j\right) \\
 &= \sum_{l=2}^\infty \sum_{k=2l+1}^\infty \left( \sum_{j=2l}^{k-1} h^0_{k-j,l} (y^l)_j\right) x^k\\
 &=\sum_{k=5}^\infty \sum_{l=2}^{\lfloor \frac{k-1}{2}\rfloor} \left( \sum_{j=2l}^{k-1} h^0_{k-j,l} (y^l)_j\right) x^k.
\end{align*}

% result follows by Cauchy's product rule. 
\end{proof}
% In the following, we work to prove the statement:
\begin{proposition}\label{prop:boundgk0}
Let $y\in\mathcal B^{C}$. Then $\rsp{\mathcal G^0[y]}\in \mathcal D^0$. In particular,  there is a constant $K=K(a^0,\rho,C)$ such that 
 \begin{align}
  \vert \rsp{\mathcal G^0[y]_k}\vert \le B \rho^{-k} +\mu K w_{k-2}^0\quad \forall\,k\ge 5.\label{eq:gk0est}
 \end{align}
Moreover, $y\mapsto\rsp{\mathcal H^0[y]}$ is $C^1$ (in the sense of Fr\'echet) and 
\begin{align}
  (D(\rsp{\mathcal H^0[y]})(z))(x) = \sum_{k=5}^\infty (D(\rsp{\mathcal H^0[y]})(z))_k x^k,\quad \vert (D(\rsp{\mathcal H^0[y]})(z))_k\vert \le  K w_{k-2}^0\Vert z\Vert\quad \forall\,z\in \mathcal D^0,\label{eq:gk0dest}
\end{align}
recall the definition of $\Vert \cdot\Vert$ in \eqref{gevrey}.

\end{proposition}
We prove this proposition in Section \ref{sec:boundgk0} below.
First we need some intermediate results. 
\begin{lemma}\label{lemma:wk0}
 Consider $w_k^0$ defined in \eqref{eq:wk0} for all $k\in \mathbb N\setminus\{1\}$ and suppose $a^0>-2$ (\rsp{Hypothesis} \ref{assa}). Then the following holds.
 \begin{enumerate}
  \item \label{C10} Convolution estimate: There exists a $C=C(a^0)>0$ such that 
  \begin{align*}
   \sum_{j=2}^{k-2} w_j^0 w_{k-j}^0 \le C w_{k-2}^0 \quad \forall k\,\ge 4.
  \end{align*}
  \item \label{C20} Let $\rho>0$. Then there exists a $C=C(a^0,\rho)>0$ such that 
 \begin{align*}
  \sum_{j=2}^{k-2} \rho^{j-k+2} w_j^0\le C w_{k-2}^0\quad \forall\,k\ge 4.
 \end{align*}
  \item \label{C30}
  Let $\xi>0$. Then there exists a $C=C(a^0,\xi)>0$ such that 
\begin{align*}
   \sum_{l=2}^{\lfloor \frac{k}{2}\rfloor } \xi^{l-2}  w_{k-2(l-1)}^0 \le C w_{k-2}^0 \quad \forall\, k\ge 4.
  \end{align*}
 \end{enumerate}

\end{lemma}
\begin{proof}
 We prove the items \ref{C10}--\ref{C30} successively in the following.

 \textit{Proof of item \ref{C10}}. We first notice that 
 \begin{align*}
 \sum_{j=2}^{k-2} w_{j}^0 w_{k-j}^0 \le 2\sum_{j=2}^{\lfloor \frac{k}{2}\rfloor} w_{j}^0 w_{k-j}^0 ,
 \end{align*}
 with $k\ge 4$. The result follows once we have shown that 
 \begin{align}
 2\sum_{j=2}^{\lfloor \frac{k}{2}\rfloor} w_{j}^0 w_{k-j}^0 \le Cw_{k-2}^0\quad \forall\,k\ge 4,\label{firstconv}
 \end{align}
 for some $C=C(a^0)$. We believe that this result, which is a result on gamma functions, is known, but for completeness we will present a simple proof that will form the basis for proofs of similar statements later on.
  
 The starting point for this approach is to define $\Phi_1^0(j)$ for $j\in [2,k-2]$ by  
 \begin{align}\label{eq:Phi00}
 w_j^0 w_{k-j}^0=\exp({\Phi_1^0(j)}).
 \end{align}
 We have
 \begin{align*}
 \frac{d}{dj} \Phi_1^0(j) = \phi(j+a^0)-\phi(k-j+a^0),\quad \frac{d^2}{dj^2}\Phi_1^0(j) = \phi'(j+a^0)+\phi'(k-j+a^0),
 \end{align*}
 using \eqref{eq:digamma}. Since the digamma function $\phi(z)$ is strictly increasing for $z>0$, see \eqref{eq:digammaprop}, and since $a^0>-2$ (recall \rsp{Hypothesis} \ref{assa}), we conclude that  $\Phi_1^0(j)$, $j\in [2,k-2]$, is convex, having a single minimum at $j=\frac{k}{2}$. We therefore have that
\begin{align}
 \Phi_1^0(j) \le Q_1^0(j-2)+P_1^0\quad \forall\,j\in [2,\textstyle{\frac{k}{2}}],\label{eq:Phi0ineq}
\end{align}
where 
\begin{align} \label{expP0}
\exp(P_1^0) = w_2^0w_{k-2}^0\quad \mbox{and}\quad 
 Q_1^0= \frac{1}{\frac{k}{2}-2}\log \frac{(w_{k/2}^0)^2}{w_{2}^0w_{k-2}^0}<0;
\end{align}
in particular equality holds in \eqref{eq:Phi0ineq} for $j=2$ and $j=\frac{k}{2}$ so that also $$\exp\left({Q_1^0\left(\frac{k}{2}-2\right)+P_1^0}\right) = w_{k/2}^0.$$ 
We illustrate the situation in Fig. \ref{Phi1}. Then by \eqref{stirling0}, a simple calculation shows that
\begin{align}\label{Q0asymp}
Q_1^0=-\log 4 + o(1) \quad \mbox{and}\quad (w_{k/2}^0)^2/w_{k-2}^0\rightarrow 0 \quad \mbox{for}\quad  k\rightarrow \infty.\end{align} 
Therefore 
 \begin{align*}
  \sum_{j=2}^{\lfloor \frac{k}{2}\rfloor} w_{j}^0 w_{k-j}^0  &\le w_2^0 w_{k-2}^0 + \int_2^{\infty} e^{Q_1^0(j-2)+P_1^0}dj\\
  &\le (1+\log^{-1} 4) w_2^0 w_{k-2}^0 \left(1 + o_{k_0\rightarrow \infty}(1)\right),
\end{align*}
using \eqref{expP0} and \eqref{Q0asymp},
for all $k\ge k_0$ large enough. 
% We conclude that 
% \begin{align*}
%  C:= \left(\sup_{j\in [2,k-2]} \frac{\vert f_j\vert}{w_j^0}\frac{\vert h_j\vert}{w_j^0}\right)^{-1} \sup_{k\ge 4} \left(\frac{1}{w_{k-2}^0} \vert (fh)_k\vert\right),
% \end{align*}
This finishes the proof of item \ref{C10}.

\begin{figure}[h!]
\begin{center}
\includegraphics[width=.4\textwidth]{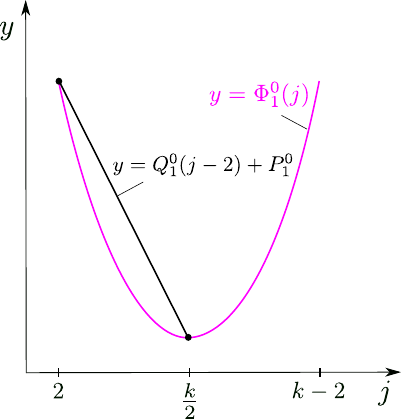}
\end{center}
\caption{Graph of the function $\Phi_1^0$ (magenta) and the secant $Q_1^0(j-2)+P_1^0$ (in black), see \eqref{eq:Phi00} and \eqref{expP0}. Since $\Phi_1^0$ is convex,  \eqref{eq:Phi0ineq} holds.}
\label{Phi1}
% Remark: If $U$ smooth for $|u|<\delta$, the compact manifolds lie inside $|u|<\delta$ for large $n$
\end{figure}

\textit{Proof of item \ref{C20}}. We proceed as in the proof of item \ref{C10}: Let $w_j^0=\exp({\Phi_2^0(j)})$ for $j\in [2,k-2]$. Then $\Phi_2^0$ is convex; in fact $\frac{d}{dj}\Phi_2^0(j) = \phi(j+a^0)$ (positive for $j\ge 4$ since $a^0>-2$), $\frac{d^2}{dj^2}\Phi_2^0(j) = \phi'(j+a^0)>0$, see \eqref{eq:digamma} and \eqref{eq:digammaprop}. We conclude that
 \begin{align}
  \Phi_2^0(j)\le Q_2^0(j-2)+P_2^0\quad \forall\,j\in [2,k-2],\label{eq:Phi2j}
 \end{align}
 where 
 \begin{align}
  \exp({P_2^0}) = w_2^0 \quad \mbox{and}\quad Q_2^0 = \frac{1}{k-4}\log \frac{w_{k-2}^0}{w_2^0}>0;\label{P0lemma34}
 \end{align}
 in particular equality holds in \eqref{eq:Phi2j} for $j=2$ and $j=k-2$.
By \eqref{stirling0}, we find that 
\begin{align}\label{eq:Q20est}
Q_2^0=\log k-1 + o(1)\quad \mbox{for}\quad k\rightarrow \infty.
\end{align} 
We can therefore estimate
\begin{align*}
 \sum_{j=2}^{k-2} \rho^{j-k+2} w_j^0 &\le \rho^{-k+2} e^{-2Q_2^0+P_2^0} \sum_{j=2}^{k-2} (\rho e^{Q_2^0})^j.
 \end{align*}
 By \eqref{eq:Q20est}, there is a $k_0\gg 1$ such that 
\begin{align*}
 \rho e^{Q_2^0} \ge 2\quad \forall\,k\ge k_0,
\end{align*}
and therefore by estimating the geometric sum and using \eqref{P0lemma34}, we find that 
\begin{align*}
 \sum_{j=2}^{k-2} \rho^{j-k+2} w_j^0 \le  2e^{Q_2^0(k-4)+P_2^0}= 2w_{k-2}^0\quad \forall\,k\ge k_0.
\end{align*}
It follows that
\begin{align*}
 C : = \sup_{k\ge 4}\left( \frac{1}{w_{k-2}^0} \sum_{j=2}^{k-2} \rho^{j-k+2} w_j^0\right)<\infty,
\end{align*}
is well-defined. 

\textit{Proof of item \ref{C30}}. We use 
\begin{align*}
 w_j^0 \le e^{Q_2^0 (j-2) +P_2^0}\quad \forall j\in [2,k],
\end{align*}
with $Q_2^0$ and $P_2^0$ defined in \eqref{P0lemma34}, 
to estimate
\begin{align*}
   \sum_{l=2}^{\lfloor \frac{k}{2}\rfloor } \xi^{l-2}  w_{k-2(l-1)}^0 \le e^{Q_2^0 k +P_2^0} \xi^{-2} \sum_{l=2}^{\lfloor \frac{k}{2}\rfloor } (\xi e^{-2Q_2^0})^l.
  \end{align*}
By \eqref{eq:Q20est}, there is a $k_0\gg 1$ such that 
\begin{align*}
 \xi e^{-2Q_2^0}\le \frac12 \quad \forall\,k\ge k_0,
\end{align*}
and therefore by estimating the geometric sum and using \eqref{P0lemma34}, we find that
\begin{align*}
 \sum_{l=2}^{\lfloor \frac{k}{2}\rfloor } \xi^{l-2}  w_{k-2(l-1)}^0 \le 2 e^{Q_2^0(k-4)+P_2^0} = 2 w_{k-2}^0\quad \forall\, k\ge k_0.
\end{align*}
It follows that
\begin{align*}
  C : = \sup_{k\ge 4}\left( \frac{1}{w_{k-2}^0} \sum_{l=2}^{\lfloor \frac{k}{2}\rfloor } \xi^{l-2}  w_{k-2(l-1)}^0\right)<\infty,
\end{align*}
is well-defined.
\end{proof}

 \begin{remark}
  The strategy used in the proof of Lemma \ref{lemma:wk0}, based on the convexity of the functions $\Phi_i^0(j)$,  see e.g. \eqref{eq:Phi00} and Fig. \ref{Phi1}, will also be used for $\epsilon>0$ below, see Lemma \ref{lemma:wkeps} . %We emphasize that this approach was not used in \cite{MR4445442}.
 \end{remark}
%  \begin{remark}\label{rem:rhocond}
% It is not clear how $C$ in Lemma \ref{lemma:wk0} item \ref{C20} depends on $\rho$. 
%  \end{remark}

\begin{lemma}\label{lemma:fh}
If $G\in \mathcal D^0$ and $H\in \mathcal D^0$ then $GH\rspp{=:\sum_{k=4}^\infty (GH)_k (\cdot)^k}\in \mathcal D^0$. In particular, there is a constant $C=C(a^0)$ such that
\begin{align}\label{eq:fh}
 \vert (GH)_k\vert &\le C\Vert G\Vert \Vert H\Vert w_{k-2}^0\quad \forall\,k\ge 4.
\end{align}

\end{lemma}
\begin{proof}
Notice that \eqref{eq:fh} implies the first statement since
\begin{align*}
 \frac{w_{k-2}^0}{w_k^0}= \frac{1}{(k-1+a^0)(k-2+a^0)} \quad \forall k\ge 4.
\end{align*}
using \eqref{eq:Gamma}. Next regarding \eqref{eq:fh}, we use \eqref{eq:cauchyproduct}: $(GH)_k = \sum_{j=2}^{k-2} G_j H_{k-j}\quad\Longrightarrow$
% For this purpose, we first consider $l=2$: 
 \begin{align*}
   \vert (GH)_k\vert \le  \Vert G\Vert \Vert H\Vert\sum_{j=2}^{k-2} w_{j}^0 w_{k-j}^0 \le C\Vert G\Vert \Vert H\Vert  w_{k-2}^0,
 \end{align*}
 by Lemma \ref{lemma:wk0} item \ref{C10}.
\end{proof}
A consequence of this result is that 
  \begin{align}
  \vert (y^l)_k\vert \le \Vert y\Vert^l C^{l-1} w_{k-2(l-1)}^0\quad \forall\,k\ge 2l,\label{eq:ylk}
 \end{align}
 for all $l\ge 2$. This follows by induction. Indeed, having already established the base case, $l=2$, in Lemma \ref{lemma:fh}, we can proceed analogously for any $l$ by writing
 \begin{align*}
 (y^l)_k = \sum_{j=2(l-1)}^{k-2} (y^{l-1})_{j} y_{k-j},
\end{align*}
and using
\begin{align*}
 \sum_{j=2}^{k-2(l-1)}w_j^0 w_{k-2(l-2)-j}^0\le C w_{k-2(l-1)}^0,
\end{align*}
cf. Lemma \ref{lemma:wk0} item \ref{C10}. 
% It is also easy to show the following:
We also emphasize the following:
\begin{align}\label{eq:ylkd}
(y^l)_k, \,k\ge 2l,\quad \mbox{only depends upon}\quad y_2,\cdots,y_{k-2(l-1)}\quad \forall\, l\in \mathbb N.
\end{align}
% Notice that this last inequality follows from \eqref{firstconv} with $k$ replaced by $k-2(l-2)$.
%  Further details are left for the reader. 

% \end{proof}

% \begin{lemma}\label{lemma:sum10}
%  
% \end{lemma}
% \begin{proof}
%  We proceed as above: 
%  
%  
% \end{proof}

\subsection{Proof of Proposition \ref{prop:boundgk0}}\label{sec:boundgk0}
We now turn to the proof of Proposition \ref{prop:boundgk0} (with $k\ge 5$). By \eqref{eq:gkcond0}, \eqref{gk0}, \eqref{gk12}  and Lemma \ref{lemma:fh}, we have 
\begin{align*}
 \vert\rsp{\mathcal G^0[y]_k}\vert &\le  B \rho^{-k} +\mu  \sum_{l=2}^{\lfloor \frac{k-1}{2}\rfloor} \Vert y\Vert^l  \rho^{-l} C^{l-1} \sum_{j=2l}^{k-1} \rho^{j-k} w_{j-2(l-1)}^0\\
&\le B \rho^{-k} +\mu  \sum_{l=2}^{\lfloor \frac{k}{2}\rfloor} \Vert y\Vert^l  \rho^{-l} C^{l-1} \sum_{j=2l}^{k} \rho^{j-k} w_{j-2(l-1)}^0;
\end{align*}
the last estimate, due to $$ \sum_{l=2}^{\lfloor \frac{k-1}{2}\rfloor} (\cdots)\sum_{j=2l}^{k-1}(\cdots)\le  \sum_{l=2}^{\lfloor \frac{k}{2}\rfloor }(\cdots) \sum_{j=2l}^k(\cdots),$$ is not important, but it streamlines some estimates for $\epsilon=0$ with similar ones for $\epsilon>0$ later on (see e.g. \eqref{thisest2}).
% Lemma \ref{lemma:sum10} directly shows that the second term has the form $\mu K w_{k-2}^0$, given that $\Vert y \Vert \le C$. 
We focus on the final term:
\begin{align*}
 \mu  \sum_{l=2}^{\lfloor \frac{k}{2}\rfloor} \Vert y\Vert^l  \rho^{-l} C^{l-1} \sum_{j=2l}^{k} \rho^{j-k} w_{j-2(l-1)}^0
\end{align*}
By Lemma \ref{lemma:wk0} item \ref{C20} with $k\rightarrow k-2(l-2)$, we can conclude that
\begin{align*}
 \sum_{j=2l}^{k} \rho^{j-k} w_{j-2(l-1)}^0 \le C w_{k-2(l-1)}^0,
\end{align*}
where $C>0$ is large enough but independent of $l$ and $k$. We are therefore left with 
\begin{align*}
  \mu  \sum_{l=2}^{\lfloor \frac{k}{2}\rfloor} \Vert y\Vert^l  \rho^{-l} C^{l} w_{k-2(l-1)}^0, 
\end{align*}
upon increasing $C>0$ if necessary. This sum is bounded by $\mu K w_{k-2}$, with $K=K(\Vert y\Vert)>0$, for all $k\ge 5$ by Lemma \ref{lemma:wk0} item \ref{C30}. This completes the proof of \eqref{eq:gk0est}.

The proof of \eqref{eq:gk0dest} proceeds completely analogously. \rsp{In particular, we find that 
\begin{align}
 (D(\mathcal H^0[y])(z))_k =  \sum_{l= 2}^{\lfloor \frac{k}{2}\rfloor}\sum_{j=2l}^{k-1} h^0_{k-j,l} l (y^{l-1}z)_j\quad \forall\,z(x)=\sum_{k=2}^\infty z_k x^k\in\mathcal D^0,\,k\ge 5,\label{eq:DH0k}
\end{align}
using the binomial theorem, 
which is well-defined by Lemma \ref{lemma:fh}.} We therefore leave out further details. 
\qed
% 
% 
% 
% We have
% \begin{align*}
%  \vert (g^0_0(x))_k \vert \le 
% \end{align*}

% 
% \begin{lemma}
% Consider any $y\in\mathcal B^{\mathcal C}$. Then there is a constant $K=K(a,C,\rho)$ such that 
% % Suppose that the following holds true
%  \begin{align*}
% \vert \mathcal H^0[y]_k   \vert \le \mu  K w_{k-2}^0,
%  \end{align*}
% for all $k\ge 4$. 
% \end{lemma}
% \begin{proof}
%  s
% \end{proof}

% \begin{remark}
%  Different versions of the norm \eqref{gevrey} have been used in \cite{} for the analysis of Gevrey analytic properties of formal series. It also relates to Borel transform \cite{}:
%  \begin{align*}
%   \mathcal B(h) = \sum_{k=2}^\infty \frac{h_k}{k!}x^k.
%  \end{align*}
%  The norm \eqref{gevrey} is therefore also typically defined as $\sum_{k=2}^\infty \frac{\vert h_k\vert}{k!}$, but (following \cite{}) it turns out to be more useful to divide by $w_{k}^0=\Gamma(k+a^0)$ instead, as this quantity relates more directly to the equations. 
% \end{remark}

\subsection{\rsp{The formal expansion
of the center manifold}}
\rsp{We are now ready to show that $\widehat m^0\in \mathcal D^0$ (i.e. that it has finite $\mathcal D^0$-norm; existence of $\widehat m^0$ is well-known).
For this purpose, we define 
the nonlinear operator $\mathcal P^{0}:\mathcal D^0\rightarrow \mathcal D^0$ by
\begin{align}
 \mathcal P^{0}\left(y \right)(x) = \sum_{k=2}^\infty (-1)^k    w_k^0 \left(\sum_{j=2}^k \frac{(-1)^j \mathcal G^0[y]_j}{w_j^0}\right) x^k.\label{eq:mapT0}
\end{align}
It follows from \lemmaref{lemma:mk0linearcase} that $\widehat m^0(x) = \sum_{k=2}^\infty m_k^0 x^k$ is given recursively by
% 
% 
% that the $m_k^0$'s of the formal series of the center manifold $\widehat m^0(x) = \sum_{k=2}^\infty m_k^0 x^k$ are given by 
 \begin{align}
  m_k^0 = (-1)^k w_k^0 S_k^0,\quad S_k^0:=\sum_{j=2}^k \frac{(-1)^j \mathcal G^0[\widehat m^0]_j}{w_j^0},\label{eq:recmk}
 \end{align}
 where the right hand side only depends upon $m_2^0,\ldots,m_{k-3}^0$ (which is a simple consequence of \eqref{gk12} and \eqref{eq:ylkd}). Consequently, at the level of formal series, $\widehat m^0$ is a fixed-point of $\mathcal \mathcal P^{0}$:}
% Suppose that $m^0\in \mathcal D^0$ is a fixed-point of $T^0$ 
\begin{align*}
 \rsp{\mathcal P^{0}(\widehat m^0)=\widehat m^0.}
\end{align*}

\begin{lemma}
% Suppose that the formal series expansion $m^0$ of the center manifold satisfies $m^0\in \mathcal D^0$. Then $m^0$ is a fixed-point of $T^0$:
\rsp{Let 
\begin{align*}
 F:=B\sum_{j=2}^\infty \frac{ \rho^{-j}}{w_j^0}<\infty,
\end{align*}
Then there is a $\mu_0>0$ small enough, such that $\mathcal P^{0}:\mathcal B^{2F}\rightarrow \mathcal B^{2F}$ is well-defined for all $0\le \mu<\mu_0$. Moreover, $\mathcal P^{0}(y;f_2^0,\mu)$ is $C^1$ with respect to $y,f_2^0$ and $\mu$, specifically 
\begin{align}\label{contraction}
D\mathcal P^{0}(y) = \mathcal O(\mu) \quad \forall\,y\in \mathcal B^{2F},
\end{align}
so that $\mathcal P^0$ is a contraction on $\mathcal B^{2F}$ for all $0\le \mu\ll 1$.}
% Moreover, let $\widehat m^0(x)=\sum_{k=2}^\infty m_k^0 x^k$ be the formal series expansion of the center manifold $y=m^0(x)$ of \eqref{eq:normalformcm}. 
% 
% the following holds true
% \begin{align}
%   \rsp{m_k^0 = (-1)^k w_k^0 S_k^0,\quad S_k^0:=\sum_{j=2}^k \frac{(-1)^j \mathcal G^0[\widehat m^0]_j}{w_j^0}\quad \forall k\ge 2},\label{eq:recmk}
%  \end{align}
%  the right hand side only depends upon $m_2^0,\ldots,m_{k-3}^0$.
% 
% Then $\widehat m^0\in \mathcal D^0$ if and only if 
\end{lemma}
 \begin{proof}
 We have 
\begin{equation}\label{eq:T0bound}
\begin{aligned}
 \Vert \mathcal P^{0} (y)\Vert &\rsp{\le}  \sum_{j=2}^{\infty} \frac{\vert \rsp{\mathcal G^0[y]_j}\vert }{w_j^0}\\
 &\le \sum_{j=2}^\infty \frac{\vert f^0_j\vert}{w_j^0}+\mu 
  \sum_{j=5}^\infty  \frac{\vert \rsp{\mathcal H^0[y]_j}\vert }{w_j^0}\\
  &\le F+\mathcal O(\mu)\le 2F,
\end{aligned}\end{equation}
for all $y\in\mathcal B^{2F}$, provided that $\mu>0$ is small enough. Here we have used that
% using that 
\begin{align}\label{eq:boundgk0}
 \sum_{j=2}^\infty \frac{\vert f^0_j\vert}{w_j^0}&\le F,\quad  \sum_{j=5}^\infty \frac{\vert \rsp{\mathcal H^0[y]_j}\vert}{w_j^0} \le \rsp{K \sum_{j=5}^\infty (j-4)^{-2}}<\infty,\quad K=K(F),
\end{align}
by Proposition \ref{prop:boundgk0}, see \eqref{eq:gk0est}, and $a^0>-2$, recall \rsp{Hypothesis} \ref{assa}. \rsp{The statement regarding $\mathcal P^{0}(y;f_2^0,\mu)$ being $C^1$ follows from \eqref{eq:gk0dest} and the linearity with respect to $f_2^0$ and $\mu$.}
% 
%  
%  It follows from \lemmaref{lemma:mk0linearcase} that the $m_k^0$'s of the formal series of the center manifold $\widehat m^0(x) = \sum_{k=2}^\infty m_k^0 x^k$ are given by 
%  \begin{align}
%   m_k^0 = (-1)^k w_k^0 S_k^0,\quad S_k^0:=\sum_{j=2}^k \frac{(-1)^j \mathcal G^0[\widehat m^0]_j}{w_j^0};\label{eq:recmk}
%  \end{align}
%  here the right hand side only depends upon $m_2^0,\ldots,m_{k-3}^0$ (which is a simple consequence of \eqref{gk12} and \eqref{eq:ylkd}). Consequently,
%  \begin{align*}
%   \widehat m^0 = \mathcal P^{f_2^0,\mu}(\widehat m^0),
%  \end{align*}
%  and the statement follows.
% %Moreover, since $m^0$ is uniquely determined by \eqref{eq:recmk0}, the existence of a fixed-point of $T^0$ also implies that $m^0\in \mathcal D^0$.  
 \end{proof}

% \rsp{It is important to emphasize that the existence of a unique  formal series $\widehat m^0$ is standard; in fact, it also follows directly from the recursion formula \eqref{eq:recmk}. In the following, we prove that $\widehat m^0$ has a finite $\mathcal D^0$-norm and that $S_\infty^0$ exists. }  

% \fbox{I think we should use Shauder's fixed-point theorem instead}
% 
% \fbox{Let $y_k=w_k^0\xi_k^0$ and write the mapping in terms of $\xi_k^0$ instead: $\xi\mapsto \widetilde T^0(\xi)$}
% 
% \fbox{Then $\widetilde T^0(\xi)$ is a  convergent sequence for any $\xi\in l^\infty$}

\begin{proposition}\label{prop:m0D0}
\rsp{
There exists a $\mu_0>0$ sufficiently small, so that $\widehat m^0\in \mathcal D^0$ for all $\mu\in [0,\mu_0)$ and 
 \begin{align}
 \Vert \widehat m^0 \Vert \le 2F\quad \forall\,\mu \in [0,\mu_0).\label{eq:m0estF}
\end{align}
Moreover, $\widehat m^0$ is $C^1$ with respect to $f_2^0\in I$ and $\mu\in [0,\mu_0(I))$, with $I\subset \mathbb R$ any fixed compact set.}
% \rsp{and define 
% \begin{align*}
%  \mu_0(F) =F \overline K(F)^{-1} \left(\sum_{j=5}^\infty (j-4)^{-2}\right)^{-1},
% \end{align*}
% where $\overline K(F)=K(a^0,\rho,2F)$ with $K$ being the constant from Proposition \ref{prop:boundgk0}.
% Then $\widehat m^0\in \mathcal D^0$ for all $\mu \in [0,\mu_0(F))$ and
% \begin{align}
%  \Vert \widehat m^0 \Vert \le 2F\quad \forall\,\mu \in [0,\mu_0(F)).\label{eq:m0estF}
% \end{align}}
% \end{proposition}
% such that 
% $T^0:\mathcal B^{2F}\rightarrow\mathcal B^{2F}$ (recall \eqref{eq:ballC}) is $C^1$ and a contraction for all $0\le \mu<\mu_0$ small enough. 
%well-defined. In particular, for any 
 %  
%  . In particular, we have the following bound on the operator norm:
%  \begin{align}
%   \Vert T^0\Vert \le \frac{\delta^2}{1-\delta}.\label{T0norm}
%  \end{align}
\end{proposition}
\begin{proof}
 \rsp{There are different ways to proceed (see also Remark \ref{rem:ref2} below), but we will simply use the fact that $\mathcal P^0(y;f_2^0,\mu)$ is $C^1$ and a contraction for all $0\le \mu\ll 1$. Indeed, define $\mathcal Q^0(y,f_2^0,\mu):= y-\mathcal P^{0}(y;f_2^0,\mu)$, $y\in \mathcal B^{2F}$. Fixed-points of $\mathcal P^0$ then correspond to roots of $\mathcal Q^0$.
 Define $
 \widehat m^0_*=\sum_{k=2}^\infty m_{*k}^0 x^k$ with $m_{*k}^0$ given by \lemmaref{lemma:mk0linearcase} (the $y$-linear case for $\mu=0$):
 \begin{align*}
 m_{*k}^0 = (-1)^k \Gamma(k+a^0) S_k^0,\quad S_k^0 = \sum_{j=2}^k \frac{(-1)^j f_j^0}{\Gamma(j+a^0)}. 
 \end{align*}
  Then by \eqref{contraction}, we have that:
 \begin{align*}
  \mathcal Q^0(\widehat m^0_*,f_2^0,0)=0,\quad D_y\mathcal Q^0(\widehat m^0_*,f_2^0,0)=\operatorname{Id}_{\mathcal D^0}\quad \forall\,f_2^0\in I.
 \end{align*}
The result therefore follows by the implicit function theorem.} 
\end{proof}

% \begin{proof}\end{proof}

\begin{remark}\label{rem:ref2}
\rsp{As \eqref{eq:recmk} defines $m_k^0$ recursively, it is clearly also possible to prove \eqref{eq:m0estF} by induction. By definition, \eqref{eq:m0estF} is equivalent with
\begin{align}
 \left\vert \frac{m_{k}^0}{w_{k}^0}\right\vert \le 2F\quad \forall\,k\ge 2.\label{eq:induction}
\end{align}
The statement is clearly true for all $k=2,\ldots,4$ (base case). For the induction step, suppose that the result holds true for any $k\ge 4$. 
Then analogously to \eqref{eq:T0bound}, we find that \eqref{eq:recmk}, Proposition \ref{prop:boundgk0} and $a^0>-2$ imply that
\begin{align*}
 \left\vert \frac{m_{k+1}^0}{w_{k+1}^0}\right\vert =\vert S_{k+1}^0\vert &\le  B\sum_{j=2}^{\infty} \frac{\rho^{-j}}{w_j^0}+\mu K\sum_{j=5}^{\infty} (j-4)^2 \le 2F,
\end{align*}
for all $0\le \mu<\mu_0(F)$, as desired. Here we have used that $S_{k+1}$ is a finite sum that only depends upon $m_2^0,\ldots,m_{k-2}^0$ (where \eqref{eq:induction} holds true by the induction hypothesis). }
\end{remark}
% 
% then
% \begin{align*}
%  
% \end{align*}
% 
% if $\mu\ge 0$ is small enough. }
% \vert \frac{m_k^0}{w_k^0}\vert \le 2F
% 
% We have 
% \begin{equation}\label{eq:T0bound}
% \begin{aligned}
%  \Vert T^0(y)\Vert &\rsp{\le}   \sum_{j=2}^{\infty} \frac{\vert \mathcal G^0[y]_j\vert }{w_j^0}\\
%  &\le \sum_{j=2}^\infty \frac{\vert f^0_j\vert}{w_j^0}+\mu 
%   \sum_{j=5}^\infty  \frac{\vert \mathcal H^0[y]_j\vert }{w_j^0}\le F+\mathcal O(\mu)\le 2F,
% \end{aligned}\end{equation}
% for all $y\in\mathcal B^{2F}$, provided that $\mu>0$ is small enough. Here we have used that
% % using that 
% \begin{align}\label{eq:boundgk0}
%  \sum_{j=2}^\infty \frac{\vert f^0_j\vert}{w_j^0}&\le F,\quad  \sum_{j=5}^\infty \frac{\vert \mathcal H^0[y]_j\vert}{w_j^0} \le K \sum_{j=5}^\infty \frac{1}{(j-1+a^0)(j-2+a^0)}<\infty,\quad K=K(F),
% \end{align}
% by Proposition \ref{prop:boundgk0}, see \eqref{eq:gk0est}, and $a^0>-2$, recall \rsp{Hypothesis} \ref{assa}. $T^0$ is also $C^1$ and $D(T^0(y)) = \mathcal O(\mu)$ cf. \eqref{eq:gk0dest}. 

For any $0\le \mu<\mu_0$, we have $\widehat m^0\in \mathcal D^0$.
We then define 
\begin{align}\label{eq:S0inf}
 S_\infty^0:=\lim_{k\rightarrow \infty} S_k^0 = \sum_{j=2}^\infty \frac{(-1)^j \rsp{\mathcal G^0[\widehat m^0]_j}}{w_j^0},
\end{align}
see \eqref{eq:recmk}.
\begin{lemma}\label{lemma:S0inf}
Consider the assumptions of Proposition \ref{prop:m0D0}. Then the series $S^0_\infty$ is absolutely convergent and $\vert S^0_\infty\vert\le 2F$.
\end{lemma}
\begin{proof}
We have $$\vert S_\infty^0\vert\le \sum_{j=2}^\infty \frac{\vert \rsp{\mathcal G^0[\widehat m^0]_j}\vert}{w_j^0}\le 2F,$$ by \eqref{eq:m0estF}.
% We have
% \begin{align*}
%  \sum_{k=2}^\infty \frac{\vert g^0_k\vert }{w^0_k} = \Vert g^0(x,y)\Vert\le \sum_{k,l}\frac{\vert g^0_{k,l} \vert}{w_{k}^0}v(\delta)^l<\infty,
% \end{align*}
% see \eqref{g0bound}.
\end{proof}
In turn, if $S^0_\infty\ne 0$ then
\begin{align}\label{eq:mk0final}
  m_k^0 = (-1)^k(1+o(1)) S_\infty^0 w_k^0\quad \mbox{for}\quad k\rightarrow \infty,
\end{align}
cf. \eqref{eq:recmk} and \eqref{eq:S0inf},
and there are constants $0<C_1<C_2$ such that 
\begin{align}
 C_1 (k-1)! k^{a^0} \le \vert m^0_k\vert \le C_2 (k-1)! k^{a^0} ,\label{m0k}
%  c_1 w_k^0 \delta^k\le \vert m^0_k\vert \le c_2 w_k^0 \delta^k,
\end{align}
for all $k$ large enough. Here we have used \eqref{stirling}:
\begin{align*}
 \Gamma(k+a^0) = \Gamma(k)(1+o(1)) k^{a^0} = (k-1)! (1+o(1)) k^{a^0}.
\end{align*}
In this way, we obtain our first result.
\begin{lemma}\label{lemma:nonanalytic}
 If $S^0_\infty \ne 0$ then $\widehat m^0\in \mathcal D^0$ is not convergent for any $x\ne 0$ and the center manifold of $(x,y)=(0,0)$ for \eqref{dydx0} is therefore not analytic.
\end{lemma}
% \begin{proof}
% From \eqref{m0k} we have that $ m^0_kx^k\not \rightarrow 0$ as $k\rightarrow 0$ for any $x\ne 0$.
% % \begin{align*}
% %  \liminf_{k\rightarrow \infty} \left|\frac{m^0_{k+1}}{m^0_{k}}\right| \ge \lim_{k\rightarrow \infty} \left|c_2 c_1^{-1} k(1+o(1)) \delta \right|= \infty.
% % \end{align*}
% \end{proof}

\begin{remark}\label{rem:converse}
 We expect that the converse: 
 \begin{align*}
 \mbox{``if $S_\infty^0=0$ holds, then the center manifold is analytic''},
 \end{align*}
 is true in general (recall Lemma \ref{lemma:linear}), \rspp{but leave this for future work}. 
\end{remark}

% \begin{remark}
% For $\mu=0$, then $g^0=f^0$ (i.e. $g^0$ is independent of $y$) and in this case, it is easy to show that the center manifold is analytic if and only if $S^0_\infty =0$. Indeed, if 
 
%  However, if $\mu>0$, so that \eqref{dydx0} is truly nonlinear, then $\mathcal H^0[y]_k$ with $y=m^0(x)=\sum_{k=2}^\infty m_k^0 x^k$ does not grow like $B\rho^{-k}$; in particular, $h^0(\cdot ,m^0(\cdot))$ will not be analytic in general. 
%   At this moment, we are not able to say something about whether $$(S_\infty^0=0)\Longrightarrow \mbox{center manifold is analytic}.$$ We will leave this to future work. 
%   
%   Nevertheless, we do have the following.
 \begin{lemma}\label{lemma:Sinf2}
  \rsp{$S_\infty^0=S_\infty^0(f_2,\mu)$ is $C^1$ with respect to $f_2^0$ for all $0\le \mu \ll 1$.
 In particular,
\begin{align*}
  \frac{\partial S^0_\infty}{\partial f_{2}^0} =\frac{1}{w_2^0}+\mathcal O(\mu)\ne 0.
\end{align*}}
 \end{lemma}

\begin{proof} 
% We first prove that $\widehat m^0=\widehat m^0(f_2,\mu)\in \mathcal D^0$ is $C^1$. For this purpose, define $T^0:\mathcal B^{2F} \times I\times [0,\mu_0)\rightarrow \mathcal B^{2F}$ by 
% \begin{equation}\label{eq:mapT0}
% \begin{aligned}
%   T^0\left(y,f_2^0,\mu \right)(x): &= \sum_{k=2}^\infty (-1)^k    w_k^0 \left(\sum_{j=2}^k \frac{(-1)^j \mathcal G^0[y]_j}{w_j^0}\right) x^k\\
%   &=\sum_{k=2}^\infty (-1)^k    w_k^0 \left(\sum_{j=2}^k \frac{(-1)^j f_j^0}{w_j^0}\right) x^k + \mu \sum_{k=5}^\infty (-1)^k    w_k^0 \left(\sum_{j=2}^k \frac{(-1)^j (h^0(\cdot,y(\cdot))_j}{w_j^0}\right) x^k.
% \end{aligned}
% \end{equation}
% $T^0$ is $C^1$, in particular $D_y(T^0(y,f_2^0,\mu))=\mathcal O(\mu)$ cf. \eqref{eq:gk0dest}. Moreover, by  \eqref{eq:recmk}, $T^0(\widehat m^0(f_2^0,\mu),f_2^0,\mu)=\widehat m^0(f_2^0,\mu)$ and the smoothness of $\widehat m^0\in \mathcal D^0$ with respect to $f_2^0$ and $\mu$ then follows from the unicity of $\widehat m^0$ and the implicit function theorem. Indeed, for the latter we consider the $C^1$-function $$(y,f_2^0,\mu)\mapsto y-\mathcal T^0(y,f_2^0,\mu),$$ with the root
% $$ (y,f_2^0,\mu) = \left(\sum_{k=2}^\infty m_k^0 x^k,f_2^0,0)\right),$$  with $m_k^0$ given by \lemmaref{lemma:mk0linearcase} for $\mu=0$ being regular: $D(y-\mathcal T^0(y,f_2^0,0))=I$. 
\rsp{Having already established the $C^1$-smoothness of $\widehat m^0=\widehat m^0(f_2^0,\mu)$ the result follows from differentiation of \eqref{eq:S0inf}
 with respect to $f_2^0$ (using \eqref{eq:DH0k}).
}

\end{proof}

% \end{remark}
% \subsection{Completing the proof of \thmref{main1}}
\thmref{main1} item 1 follows from Lemma \ref{lemma:nonanalytic}, see also \eqref{eq:mk0final} with $m_k^0=\frac{1}{k!}\frac{d^k }{dx^k}m^0(0)$, $w_k^0=\Gamma(k+a^0)$. Finally, Lemma \ref{lemma:Sinf2} is precisely the statement in \thmref{main1} item 2. 
% Henceforth we assume that $S^0_\infty\ne 0$. 
\section{The analytic weak-stable manifold $W^{ws}$: The proof of \thmref{main2}}\label{sec:node}
% \fbox{need to update this}

% \fbox{need to use sup-norm instead}
To study \eqref{dydx} and the analytic weak-stable manifold for all $0<\epsilon\ll 1$, $\epsilon^{-1}\notin \mathbb N$, we use the scalings \eqref{scalingxy}, repeated here for convenience: $$x=\epsilon \overline x,\quad y=\epsilon \overline y.$$ 
\rsp{In the  $(\overline x,\overline y)$-coordinates, \eqref{eq:normalform} becomes the following singularly perturbed system:}
\begin{equation}\label{eq:barxyt}
\begin{aligned}
 \dot{\overline x} &=\epsilon \overline x(\overline x-1),\\
 \dot{\overline y} &=-\overline y(1+\epsilon a^\epsilon \overline x)+\epsilon \overline g^\epsilon(\overline x,\overline y).
\end{aligned}
 \end{equation}
\rsp{or alternatively in the form}
\begin{align}
 \epsilon \overline x (\overline x-1) \frac{d\overline y}{d \overline x} + \overline y (1+\epsilon a^\epsilon\overline x) =\epsilon \overline g^\epsilon(\overline x,\overline y),\label{ybarEqn}
\end{align}
\rsp{relevant for invariant manifold solutions},
where
\begin{equation}\label{overlineg}
\rsp{\begin{aligned}
 \overline g^\epsilon(\overline x,\overline y) :&=\epsilon^{-2} g^\epsilon(\epsilon \overline x,\epsilon \overline y)= \overline f^\epsilon (\overline x)+\mu \overline h^\epsilon(\overline x,\overline y),\\
 \overline f^\epsilon(\overline x)&:=\epsilon^{-2} f^\epsilon(\epsilon \overline x),\\
 \overline h^\epsilon(\overline x,\overline y)&:=\epsilon^{-2} h^\epsilon(\epsilon \overline x,\epsilon\overline y).
 %\underbrace{\sum_{k=2}^\infty f_k^\epsilon \epsilon^{k-2} \overline x^k}_{:=\overline f^\epsilon(\overline x)}+\underbrace{\sum_{k=2}^\infty h_{k,1}^\epsilon \epsilon^{k-1} \overline x^k \overline y+\sum_{k=0}^\infty \sum_{l=2}^\infty h_{k,l}^\epsilon \epsilon^{k+l-2} \overline x^k \overline y^l}_{:=\overline h^\epsilon(\overline x,\overline y)}.
\end{aligned}}
\end{equation}
% using \eqref{eq:gexpansion}.
Here we have also defined $\overline f^\epsilon$ and $\overline h^\epsilon$. By \eqref{eq:fhexpansion}, we obtain the absolutely convergent power series expansion of $\overline f^\epsilon$ and $\overline h^\epsilon$:
\begin{align}
 \overline f^\epsilon(\overline x) = \sum_{k=2}^\infty f_k^\epsilon \epsilon^{k-2} \overline x^k,\quad \overline h^\epsilon(\overline x,\overline y) = \sum_{k=2}^\infty h_{k,1}^\epsilon \epsilon^{k-1} \overline x^k \overline y+\sum_{k=1}^\infty \sum_{l=2}^\infty h_{k,l}^\epsilon \epsilon^{k+l-2} \overline x^k \overline y^l.\label{eq:overlinefheps}
\end{align}

% \fbox{Replace $\overline g$ by $\epsilon\overline g$}

For all $\epsilon^{-1}\notin \mathbb N$, $0<\epsilon\ll 1$, $(x,y)=(0,0)$ is a nonresonant hyperbolic node of \eqref{eq:barxyt} (the eigenvalues being $-\epsilon$ and $-1$). Consequently, there is an analytic weak-stable manifold:
\begin{align}
 W^{ws}:\quad \overline y=\overline m^\epsilon( \overline x),\quad 
 \overline m^\epsilon (\overline x) = \sum_{k=2}^\infty \overline m^\epsilon_k \overline x^k,\quad \overline x\in (-\delta,\delta),\label{barmeps}
\end{align}
with 
$\delta=\delta(\epsilon)>0$,
see e.g. \cite[Theorem 2.14]{dumortier2006a}, \rsp{which solves \eqref{ybarEqn}}. Now, for any series $\overline y(\overline x)=\sum_{k=2}^\infty \overline y_k \overline x^k$, we define \rsp{${\overline{\mathcal G}^\epsilon[y]}$ and ${\overline{\mathcal G}^\epsilon[y]}_k$ as above by composition with the analytic function $\overline g^\epsilon$:
 \begin{align*}
 {\overline{\mathcal G}^\epsilon[y]}(x):=\overline g^\epsilon(\overline x,\overline y(\overline x))=\sum_{k=2}^\infty {\overline{\mathcal G}^\epsilon[y]}_k \overline x^k.
 \end{align*}
  Again, $\overline{\mathcal H}^\epsilon[y]$ and $\overline{\mathcal H}^\epsilon[y]_k$ are defined in the same way by composition with the analytic function $\overline h^\epsilon$, recall \eqref{overlineg}. We have ${\overline{\mathcal H}^\epsilon[y]}_k=0$ for $k=2$ and $3$ and therefore
\begin{align}
\begin{cases}
{\overline{\mathcal G}^\epsilon[y]}_2&=f^\epsilon_{2}, \\
{\overline{\mathcal G}^\epsilon[y]}_3 &= f^\epsilon_{3}\epsilon,\\
 {\overline{\mathcal G}^\epsilon[y]}_k &= f^\epsilon_{k}\epsilon^{k-2} + \mu \overline{\mathcal H}^\epsilon[y]_k,\quad k\ge 4.
 \end{cases}\label{gk0eps}
\end{align}
 Finally, we define $(\overline y^l)_k$, $k\ge 2l$, by
\begin{align*}
 \overline y(\overline x)^l =: \sum_{k=2l}^\infty (\overline y^l)_k \overline x^k,
\end{align*}
 for all $\overline y =\sum_{k=2}^\infty \overline y_k \overline x^k$, $l\in\mathbb N$.
}
\begin{lemma}
 The following holds: 
 \begin{align}\label{gk12eps}
 \rsp{\overline{\mathcal H}^\epsilon[y]_k} &= \sum_{j=2}^{k-2} h_{k-j,1}^\epsilon \epsilon^{k-j-1} \overline y_j+\sum_{l=2}^{\lfloor \frac{k-1}{2}\rfloor} \sum_{j=2l}^{k-1} h_{k-j,l}^\epsilon \epsilon^{k-j+l-2}(\overline y^l)_j,\quad k\ge 4.
\end{align}
(The last sum is zero for $k=4$.)
\end{lemma}
\begin{proof}
 The proof is identical to the proof of Lemma \ref{lemma:h0k} and further details are therefore left out. 
\end{proof}

% Here 

% 
% As for $\epsilon=0$, we first suppose that 
% \begin{align*}
% \overline g^\epsilon(\overline x,\overline y) = \sum_{k=2}^\infty  \epsilon^{k-2} g^\epsilon_k  \overline x^k.
% \end{align*}
% with $\overline h^0_1=0$. 
% with 
% \begin{align*}
% \overline h^\epsilon_k = \frac{1}{(k-2)!} \frac{d^{k-2}}{dx^{k-2}}f(0,\epsilon),
% \end{align*}
% for $k\ge 2$,

\begin{lemma}\lemmalab{lemma:barmkepsrec}
Suppose that $\epsilon^{-1}\notin \mathbb N$, $0<\epsilon\ll 1$ and let \eqref{barmeps} denote the analytic weak-stable manifold. 
Then the $\overline m^\epsilon_k$'s satisfy the recursion relation:
\begin{align}
 (1-\epsilon k) \overline m^\epsilon_k  +\epsilon ( k-1+ a^\epsilon) \overline m^\epsilon_{k-1} = \epsilon \rsp{\overline{\mathcal G}^\epsilon[\overline m^\epsilon]_k} \quad \forall\,k\ge 2;\label{barmkepsrec}
\end{align}
 here we define $\overline m^\epsilon_{1}=0$. In particular, the right side of \eqref{barmkepsrec} only depends upon $\overline m_2^\epsilon,\ldots,\overline m_{k-2}^\epsilon$. 
\end{lemma}
\begin{proof}
 Simple calculation.
\end{proof}

% If $g^\epsilon_k$ given, we can solve this recursion for $m^\epsilon_k$:
% 
% 
% 
\begin{lemma}\label{lemma:barSkeps}
% With $g^\epsilon_k$ given, we can solve this recursion for $m^\epsilon_k$:
For $\epsilon^{-1}\notin \mathbb N$, $0<\epsilon\ll 1$, define
 \begin{align}
  \overline w^\epsilon_k: =\frac{\Gamma \left(\epsilon^{-1}-k\right)\Gamma\left(k+a^\epsilon\right)}{\epsilon \Gamma\left(\epsilon^{-1}\right)}\quad \forall\,k\ge 2,\label{barwkeps}
 \end{align}
%  for $k\ge 2$,
%  \fbox{We need to scale $\overline w^\epsilon_k$ in a different way; here for $k=2$ we have $\overline w^\epsilon_k\approx \epsilon^3$}
%  
%  
%  \fbox{should be $\Gamma\left(\epsilon^{-1} -1\right)$ in the denominator}
  and 
 \begin{align*}
  \overline S^\epsilon_k: = \sum_{j=2}^k \frac{(-1)^j \epsilon \rsp{\overline{\mathcal G}^\epsilon[\overline m^\epsilon]_j}}{\overline w^\epsilon_j (1-\epsilon j)}\quad \forall\, k\ge 2.
 \end{align*}
%   for all $k\ge 2$. 
%  In particular, $\overline S_1= -f(0,\epsilon)
Then  $\overline S^\epsilon_k$ depends upon $\overline m^\epsilon_2,\ldots,\overline m^\epsilon_{k-2}$ for each $k\ge 4$ and $\overline m^\epsilon_k$ satifies
\begin{align}
 \overline m^\epsilon_k = (-1)^k  \overline w^\epsilon_k\overline S^\epsilon_k\quad \forall\,k\ge 2.\label{barmkeps}
\end{align}
% for all $k\ge 2$.
\end{lemma}
\begin{proof}
% We insert $\overline m_k^\epsilon = \epsilon^{-2} \overline w_k^\epsilon \overline \xi_k^\epsilon$ into \eqref{barmkepsrec}. 
%  Then using the fact that the $\overline w_k^\epsilon$'s, defined by \eqref{barwkeps}, 
The result follows from induction on $k$, with the base case being $k=2$, upon using \eqref{barmkepsrec} and the recursion relation 
 \begin{align*}
  (1-\epsilon k)\overline w^\epsilon_k = \epsilon ( k-1+a^\epsilon) \overline w^\epsilon_{k-1},
 \end{align*}
 for the $\overline w^\epsilon_k$'s in the induction step.
% we obtain
% \begin{align*}
% \overline \xi_k^\epsilon +\overline \xi_{k-1}^\epsilon = \frac{\epsilon^{k-1} g_k^\epsilon}{(\epsilon^{-2}\overline w_k^\epsilon)(1-\epsilon k)}.
% \end{align*}
% To complete the proof we proceed as in the proof of Lemma \ref{lemma:mk0}. 
\end{proof}

\begin{lemma}\label{convmbarkeps}
 Write
 \begin{align}
 \overline m_k^\epsilon =: \epsilon^{k-1}m_k^\epsilon,\label{eq:mkeps}
 \end{align} 
 and let $\widehat m^0(x) = \sum_{k=2}^\infty m_k^0 x^k$ denote the formal series expansion of the center manifold for $\epsilon=0$, recall \eqref{eq:recmk}.
 Then for any fixed $k$,
 \begin{align*}
  m_k^\epsilon \rightarrow m_k^0,
 \end{align*}
 as $\epsilon\rightarrow 0$. 
\end{lemma}
\begin{proof}
 Inserting \eqref{eq:mkeps} into \eqref{barmkepsrec}, it is straightforward to obtain
 \begin{align*}
  m_k^\epsilon (1-\epsilon k ) + (k-1+a^\epsilon) m_{k-1}^\epsilon = \rsp{\mathcal G^\epsilon[y]_k}\rightarrow m_k^0  + (k-1+a^0) m_{k-1}^0 = \rsp{\mathcal G^0[\widehat m^0]_k},
 \end{align*}
as $\epsilon\rightarrow 0$. \rsp{(Here $\mathcal G^\epsilon[y]$ is the power series defined by composition $g^\epsilon(\cdot,y)$ of a series $y$ with the analytic function $g^\epsilon$ (without bars).)} The result then follows from induction on $k$.
\end{proof}

% The radius of convergence of the series \eqref{barmeps} may not be bounded uniformly away from zero as $\epsilon\rightarrow 0$; the analytic weak-stable manifold could diverge. However, 
% Following \cite{MR4445442} we represent $\overline m^\epsilon$ as a truncated series plus a ``remainder'' $r_{\overline m^\epsilon}$. Here we use citation marks to emphasize that the remainder is not necessarily small; in particular, the insight of \cite{MR4445442} is that the leading order term of $r_{\overline m^\epsilon}$ can be determined. This leads to a separate remainder term, which is truly small. 

%  \begin{remark}
%   For fixed $k$
%  \end{remark}
\subsection{Growth properties of $\overline m_k^\epsilon$}
% \fbox{this section is not updated}
 We now study the formal series \eqref{barmeps} and the growth properties of $\overline m_k^\epsilon$, $k=2,\ldots N^\epsilon$. 
 For this purpose, the following lemma, on the properties of the $\overline w_k^\epsilon$'s, defined in \eqref{barwkeps}, will be crucial. 

\begin{lemma}\label{lemma:wkeps}
Suppose that $a^0>-2$, that $\epsilon^{-1}\notin \mathbb N$ and write 
\begin{align}
 \epsilon^{-1} =: N^\epsilon+\alpha^\epsilon,\quad N^\epsilon:=\lfloor \epsilon^{-1}\rfloor,\quad \alpha^\epsilon \in (0,1),\label{Nepsdef}
\end{align}
% Suppose that $a^0>0$. 
% \fbox{Replace $N^\epsilon$ by $N^\epsilon-1$?y then estimates on the finite sum become uniform in $\epsilon$.}
% Let $\rho>0$ small enough. 
Then the following can be said about $\overline w_k^\epsilon$, defined in \eqref{barwkeps}, for all $2\le k \le N^\epsilon$:
\begin{enumerate}
\item \label{C0} For fixed $k\in \mathbb N\setminus \{1\}$:
\begin{align*}
 \overline w_k^\epsilon =\epsilon^{k-1}(1+o(1)) \Gamma(k+a^\epsilon),
\end{align*}
as $\epsilon\rightarrow 0$. 
\item \label{C2} Lower bound of $\overline w_{k}^\epsilon(1-\epsilon k)$: \begin{align*}
       \overline w_k^\epsilon (1-\epsilon k)\ge  \Gamma(k+a^\epsilon)\epsilon^{k-1}\quad \forall\, 2\le k \le N^\epsilon+1.
      \end{align*}
      \item \label{C5} Convolution estimate: There is a $C=C(a^0)$ such that 
\begin{align}
\sum_{j=2}^{k-2} \overline w_j^\epsilon \overline w_{k-j}^\epsilon &\le C \overline w_2^\epsilon \overline w_{k-2}^\epsilon \quad \forall\,4\le k\le N^\epsilon+1,\label{eq:C51}
\end{align}
%        \frac{\overline w_j^\epsilon (1-\epsilon j) \overline w_{k-j}^\epsilon (1-\epsilon(k-j))}{\overline w_k^\epsilon(1-\epsilon k)} \le  \epsilon \Gamma(a^\epsilon),
%       \end{align*}
      and
      \begin{align}
      \sum_{j=k-(N^\epsilon-1)}^{N^\epsilon-1}  \overline w_j^\epsilon   \overline w_{k-j}^\epsilon &\le C  \overline w_{N^\epsilon-1}^\epsilon  \overline w_{k-(N^\epsilon-1)}^\epsilon\quad \forall\,N^\epsilon+1\le k\le 2(N^\epsilon-1).\label{eq:C52}
      \end{align}          
\item \label{C3} Define 
\begin{equation}\label{AepsBeps}
\begin{aligned}
 Q_4^\epsilon &:= \frac{1}{N^\epsilon-3} \log \frac{\overline w_{N^\epsilon-1}^\epsilon}{\overline w_{2}^\epsilon},\\
 P_4^\epsilon &: = \log (\overline w_2^\epsilon) 
\end{aligned}
\end{equation}
Then 
\begin{equation}\label{AepsBepsApp}
\begin{aligned}
 Q_4^\epsilon &= \frac{1}{N^\epsilon-3}\left((a^\epsilon+1-\alpha^\epsilon) \log N^\epsilon + \log \frac{\Gamma(1+\alpha^\epsilon)}{\Gamma(2+a^\epsilon)}+o(1)\right),\\
 P_4^\epsilon &=\log \epsilon + \log \Gamma(2+a^\epsilon) +  o(1)
\end{aligned}
\end{equation}
and 
\begin{equation}\label{wkepsmajorant}
\begin{aligned}
 \overline w_{k}^\epsilon \le e^{Q_4^\epsilon (k-2) + P_4^\epsilon}\quad \forall\, 2\le k\le N^\epsilon-1,
\end{aligned}
\end{equation}
for all $0<\epsilon\ll 1$. 
In particular,
\begin{align*}
       \sum_{k=2}^{N^\epsilon-1}  \overline w_k^\epsilon \delta^k\le \sum_{k=2}^{N^\epsilon-1} e^{Q_4^\epsilon (k-2) + P_4^\epsilon}\delta^k \le \delta^2 C\epsilon\quad \forall\,0<\delta\le \frac34.
      \end{align*}
      for some $C>0$ and all $0< \epsilon\ll 1$.% $0\le \delta\ll 1$. 
%       \note{s}
%       \ffs{We could also just do 
%       \begin{align*}
%         \sum_{k=2}^{N^\epsilon-1} \epsilon^{-2} \overline w_k^\epsilon \delta^k\le \epsilon^{-2} \overline w_2^\epsilon \delta^2+ \int_2^{N^\epsilon-1} e^{\Phi(k)} \delta^k dk...
%       \end{align*}
%       Gives slightly worse estimate $\mathcal O(\epsilon \delta^2 \log \delta^{-1})$}

%       and for all $k\in [N^\epsilon-k,N^\epsilon]$
%       \begin{align*}
%        \overline w_k^\epsilon = \mathcal O(1) \Gamma(\epsilon^{-1}-k)\epsilon^{-(\epsilon^{-1}-k)-1+a^\epsilon} 
%       \end{align*}
%       \fbox{fix this}
%       \fbox{mistake here...}
% \item 
% %$\epsilon^{-1}=N^\epsilon + \alpha^\epsilon$, $\alpha^\epsilon\in (0,1)$ 
% \item \label{C3new} There is a constant $C=C(a^0)$ and an integer $p=p(a^0)$ such that 
% \begin{align*}
% %  \sum_{k=2}^{N^\epsilon-1}  \overline w_k^\epsilon \delta^k \le C\frac{\epsilon \delta^2}{1-\delta}+\delta^{N^\epsilon-p}\sum_{k=0}^{p-1}  \overline w_{N^\epsilon-p+k}^\epsilon \delta^k\quad \forall \,0\le \delta<1,\,0<\epsilon\ll 1.
%  \sum_{k=2}^{N^\epsilon-1}  \overline w_k^\epsilon \delta^k \le C\left(\frac{\epsilon \delta^2}{1-\delta}+(N^{\epsilon})^{a^\epsilon-\alpha^\epsilon+1} \delta^{N^\epsilon-p}  \sum_{k=0}^{p-1} (N^{\epsilon}\delta)^k\right)\quad \forall \,0\le \delta<1,\,0<\epsilon\ll 1.CHECK
% \end{align*}
% % for all $0<\epsilon\ll 1$.
\item \label{kinf} For fixed $\epsilon^{-1}\notin \mathbb N$, 
\begin{equation}\label{eq:kinf}
\begin{aligned}
       \overline w_k^\epsilon &= \frac{(-1)^{N^\epsilon-k}\Gamma(\alpha^\epsilon)\Gamma(1-\alpha^\epsilon) }{\epsilon \Gamma(\epsilon^{-1})} \frac{\Gamma(k+a^\epsilon)}{\Gamma(k+1-\epsilon^{-1} )}\\
       &=\mathcal O(1)k^{\epsilon^{-1}+a^\epsilon-1},
      \end{aligned}
      \end{equation}
      with respect to $k\rightarrow \infty$.
      \item \label{est1} Let $\xi>0$. Then there is a constant $C=C(a^\epsilon,\xi)$ such that 
      \begin{equation}\label{eq:est1}
      \begin{aligned}
       \sum_{j=2}^{k-2} (\xi^{-1}\epsilon)^{k-2-j} \overline w_j^\epsilon &\le C \overline w_{k-2}^\epsilon\quad  \forall\, 4\le k\le N^\epsilon+1,
      \end{aligned}
      \end{equation}
for all $0<\epsilon\ll 1$.
 \item \label{est2} Let $\xi>0$. Then there is a constant $C=C(a^\epsilon,\xi)$ such that 
      \begin{equation}\label{eq:est2}
      \begin{aligned}
       \sum_{l=2}^{\lfloor \frac{k}{2}\rfloor } (\xi^{-1}\epsilon)^{l-2} (\overline w_2^\epsilon)^{l-1} \overline w_{k-2(l-1)}^\epsilon &\le C\overline w_{2}^\epsilon \overline w_{k-2}^\epsilon\quad  \forall\, 4\le k\le N^\epsilon+1,
      \end{aligned}
      \end{equation}
for all $0<\epsilon\ll 1$.
%  \item $\overline w_k^\epsilon = \frac{(-1)^{N^\epsilon-k} \pi}{\sin (\pi \alpha^\epsilon)}$
% \item \label{convolutionwkeps} For any $4\le k\le N^\epsilon+1$:
% \begin{align}
%  \sum_{j=2}^{k-2} (\epsilon^{-2}\overline w_j^\epsilon) (\epsilon^{-2} \overline w_{k-j}^\epsilon) \le C\epsilon(\epsilon^{-2} \overline w_{k-2}^\epsilon),\label{eq:convolutionwkeps}
% \end{align}
% where $C=C(a^0)$ is independent of $\epsilon$. More generally, for any $l\ge 2$ and any $k\ge 2l$:
% \begin{align}\label{eq:convolutionwkleps}
%  \sum_{i_1+\cdots+ i_l=k} (\epsilon^{-2}\overline w_{i_1}^\epsilon) \cdots (\epsilon^{-2} \overline w_{i_l}^\epsilon)\le C^{l-1} \epsilon^{l-1} (\epsilon^{-2} \overline w_{k-2(l-1)}^\epsilon).
% \end{align}

\end{enumerate}

\end{lemma}
\begin{proof}

We prove the items \ref{C0}-\ref{est2} successively in the following. 

\textit{Proof of item \ref{C0}}. For fixed $k\in \mathbb N$, we have
\begin{align*}
 \overline w_k^\epsilon = \frac{\Gamma(\epsilon^{-1}-k)}{\epsilon\Gamma(\epsilon^{-1})} \Gamma(k+a^\epsilon)=\frac{\Gamma(\epsilon^{-1})\epsilon^{k-1}}{\Gamma(\epsilon^{-1})} (1+o(1))\Gamma(k+a^\epsilon)=\mathcal O(\epsilon^{k-1}),
\end{align*}
using \eqref{stirling}, \eqref{barwkeps} and the definiton of the gamma function.

\textit{Proof of item \ref{C2}}. 
We calculate
\begin{align*}
\frac{\overline w_k^\epsilon(1-\epsilon k)}{\Gamma(k+a^\epsilon)}
 &= \frac{\Gamma(\epsilon^{-1}-k+1)}{ \Gamma(\epsilon^{-1})} = \frac{1}{\Pi_{j=1}^{k-1} (\epsilon^{-1}-j)}= \epsilon^{k-1} \Pi_{j=1}^{k-1}  \frac{1}{1-j\epsilon}\ge \epsilon^{k-1},
\end{align*}
using \eqref{eq:Gamma} and $1-(k-1)\epsilon\ge \alpha^\epsilon>0$ for $2\le k\le N^\epsilon+1$.
% with $C>\frac{\sqrt{2\pi}}{8(\log 2)^{\frac32}}$ and $\epsilon$ small enough. 
% We find 

      \textit{Proof of item \ref{C5}}.
      We first focus on \eqref{eq:C51} and notice from item \ref{C0} that the claim holds true for all $4\le k\le k_0$ with $k_0>0$ fixed and all $0<\epsilon\ll 1$. We therefore consider $k_0<k\le N^\epsilon+1$ with $k_0>0$ fixed large. We write
      \begin{align*}
       \sum_{j=2}^{k-2} \overline w_{j}^\epsilon \overline w_{k-j}^\epsilon &\le  2\sum_{j=2}^{\lfloor \frac{k}{2}\rfloor} \overline w_{j}^\epsilon \overline w_{k-j}^\epsilon=:2\sum_{j=2}^{\lfloor \frac{k}{2}\rfloor} e^{\Phi_{31}^\epsilon(j)}.
      \end{align*}
      By proceeding as in the proof of Lemma \ref{lemma:fh}, a simple computation, using \eqref{eq:digamma} and \eqref{eq:digammaprop}, shows that $\Phi_{31}^\epsilon(j)$, $j\in [2,k-2]$, is convex, having a unique minimum at $j=\frac{k}{2}$. Therefore
\begin{align}
 \Phi_{31}^\epsilon(j) \le Q_{31}^\epsilon (j-2) +P_{31}^\epsilon,\label{eq:PhijC51}
\end{align}
where $Q_{31}^\epsilon$ and $P_{31}^\epsilon$ are chosen such that 
\begin{align*}
 Q_{31}^\epsilon =\frac{\Phi_{31}^\epsilon\left(\frac{k}{2}\right)-\Phi_{31}^\epsilon(2)}{\frac{k}{2}-2} = 
 \frac{1}{\frac{k}{2}-2}\log \frac{\left(\overline w_{\frac{k}{2}}^\epsilon\right)^2}{ \overline w_{2}^\epsilon \overline w_{k-2}^\epsilon},\quad P_{31}^\epsilon = \Phi_{31}^\epsilon(2).
\end{align*}
% 
% $(\epsilon^{-2} \overline w_{2}^\epsilon) (\epsilon^{-2} \overline w_{k-2}^\epsilon)=\exp(P^\epsilon)$ and 
% \begin{align*}
%  Q^\epsilon = \frac{1}{\frac{k}{2}-2}\log \frac{\left(\epsilon^{-2} \overline w_{\frac{k}{2}}^\epsilon\right)^2}{(\epsilon^{-2} \overline w_{2}^\epsilon) (\epsilon^{-2} \overline w_{k-2}^\epsilon)}.
% \end{align*}
In particular, equality holds for $j=2$ and $j=\frac{k}{2}$ in \eqref{eq:PhijC51} and consequently
% In particular, we have 
\begin{align*}
\left( \overline w_{\frac{k}{2}}^\epsilon\right)^2 = e^{Q_{31}^{\epsilon}(\frac{k}{2}-2)+P_{31}^\epsilon},\quad  \overline w_{2}^\epsilon\overline w_{k-2}^\epsilon = e^{P_{31}^\epsilon}.
\end{align*}
Using \eqref{barwkeps} and \eqref{stirling0}, a simple calculation shows that
\begin{align*}
 Q_{31}^\epsilon <\frac{1}{\frac{k}{2}-2}\log \frac{\Gamma\left(\frac{k}{2} +a^\epsilon\right)^2}{\Gamma(2+a^\epsilon)\Gamma(k-2+a^\epsilon)} = - \log 4 (1+o_{k_0\rightarrow \infty}(1)),
\end{align*}
for all $k_0<k\le N^\epsilon+1$, uniformly in $0<\epsilon\ll 1$. Then proceeding as in the proof of Lemma \ref{lemma:fh}, we have
\begin{align*}
 \sum_{j=2}^{k-2} \overline w_{j}^\epsilon \overline w_{k-j}^\epsilon&\le 2 \overline w_{2}^\epsilon \overline w_{k-2}^\epsilon +\int_2^{\infty} e^{Q_{31}^\epsilon (j-2)+P_{31}^\epsilon} dj\\
 &\le 2(1+\log 4) \overline w_{2}^\epsilon  \overline w_{k-2}^\epsilon (1+o(1)),
\end{align*}
which completes the proof of \eqref{eq:C51}.

The inequality \eqref{eq:C52} is proven in a similar way. First, we put $k=2(N^\epsilon-1)-p$ and use \eqref{barwkeps} and \eqref{stirling} to obtain
\begin{align*}
 \overline w_{N^\epsilon-1-p+j}^\epsilon \overline w_{N^\epsilon-1-j}^\epsilon = \Gamma(\alpha^\epsilon+1+p-j)\Gamma(\alpha^\epsilon+1+j) (N^\epsilon)^{2(a^\epsilon-\alpha^\epsilon)-p} (1+o(1)).
\end{align*}
Therefore
\begin{align*}
 \sum_{j=k-(N^\epsilon-1)}^{N^\epsilon-1} \overline w_{j}^\epsilon \overline w_{k-j}^\epsilon &= \sum_{j=0}^p \overline w_{N^\epsilon-1-p+j}^\epsilon\overline w_{N^\epsilon-1-j}^\epsilon\\
 &= (N^\epsilon)^{2(a^\epsilon-\alpha^\epsilon)-p}\sum_{j=0}^p  \Gamma(\alpha^\epsilon+1+p-j)\Gamma(\alpha^\epsilon+1+j)(1+o(1)).
\end{align*}
Here $$\sum_{j=0}^p \Gamma(\alpha^\epsilon+1+p-j)\Gamma(\alpha^\epsilon+1+j)\le C\Gamma(\alpha^\epsilon+1)\Gamma(\alpha^\epsilon+1+p),$$ cf. \eqref{firstconv} and therefore \eqref{eq:C52} holds true for all $2(N^\epsilon-1)-p\le k\le 2(N^\epsilon-1)$ and any $p>0$ provided that $0<\epsilon\ll 1$. 

We therefore proceed to consider $N^\epsilon+1\le k\le 2(N^\epsilon-1)-p$ with $p>0$ fixed large. We write
\begin{align*}
       \sum_{j=k-(N^\epsilon-1)}^{N^\epsilon-1}  \overline w_{j}^\epsilon\overline w_{k-j}^\epsilon &\le 2\sum_{j=k-(N^\epsilon-1)}^{\lfloor \frac{k}{2}\rfloor}  \overline w_{j}^\epsilon \overline w_{k-j}^\epsilon:=2\sum_{j=k-(N^\epsilon-1)}^{\lfloor \frac{k}{2}\rfloor} e^{\Phi_{32}^\epsilon(j)}.
      \end{align*}
      As above, $\Phi_{32}^\epsilon(j)$, $j\in [2,k-2]$, is convex, having a unique minimum at $j=\frac{k}{2}$, so that
\begin{align}
 \Phi_{32}^\epsilon(j) \le Q_{32}^\epsilon (j-2) +P_{32}^\epsilon,\label{eq:PhijC52}
\end{align}
where $Q_{32}^\epsilon$ and $P_{32}^\epsilon$ are now chosen such that 
\begin{align*}
 Q_{32}^\epsilon &=\frac{\Phi_{32}\left(\frac{k}{2}\right)-\Phi_{32}(k-(N^\epsilon-1))}{N^\epsilon-1-\frac{k}{2}} = 
 \frac{1}{N^\epsilon-1-\frac{k}{2}}\log \frac{\left(\overline w_{\frac{k}{2}}^\epsilon\right)^2}{ \overline w_{k-(N^\epsilon-1)}^\epsilon\overline w_{N^\epsilon-1}^\epsilon},\\ P_{32}^\epsilon &= \Phi^\epsilon(k-(N^\epsilon-1)).
\end{align*}
Equality holds in \eqref{eq:PhijC52} for $j=k-(N^\epsilon-1)$ and $j=\frac{k}{2}$. 
      Now, using \eqref{barwkeps} and \eqref{stirling0} a simple calculation shows that
      \begin{align*}
       Q_{32}^\epsilon < \frac{1}{N^\epsilon-1-\frac{k}{2}}\log \frac{\Gamma\left(\epsilon^{-1}-\frac{k}{2}\right)^2}{\Gamma(1+\alpha^\epsilon)\Gamma(\epsilon^{-1} -(k-(N^\epsilon-1)))} \le -\log 4(1+o_{p\rightarrow \infty}(1)),
      \end{align*}
      for all $N^\epsilon+1\le k\le 2(N^\epsilon-1)-p$, uniformly in $0<\epsilon\ll 1$. 
      We can now complete the proof by proceeding in the exact same way that we did in the proof of \eqref{eq:C51}.

\textit{Proof of item \ref{C3}}.
      First, we write
      \begin{align*}
       \overline w_k^\epsilon =  e^{\Phi_4^\epsilon(k)},
      \end{align*}
where
\begin{align*}
 \Phi_4^\epsilon(k) = \log \frac{\Gamma(\epsilon^{-1}-k)\Gamma(k+a^\epsilon)}{\epsilon \Gamma(\epsilon^{-1})}.
\end{align*}
% Here we have used \eqref{barwkeps} and the definition of the gamma function to write:
% \begin{align*}
%  \epsilon^2 \Gamma(\epsilon^{-1}+1) = (1-\epsilon)\Gamma(\epsilon^{-1}-1).
% \end{align*}
% Then, by proceeding as we did with $F(j)$ above, using the monotonicity of the digamma function $\phi(z)$, we show that 
Again, $\Psi_4^\epsilon(k)$ is convex on $k\in [2,N^\epsilon-1]$ (having a minimum at $k=k_m(\epsilon):=\frac{1}{2\epsilon}-\frac{a^\epsilon}{2}$). 
% Consequently,  
% % Let
Next, $Q_4^\epsilon$ and $P_4^\epsilon$, defined by \eqref{AepsBeps}, are chosen such that 
\begin{align*}
 Q_4^\epsilon=\frac{\Psi_4^\epsilon(N^\epsilon-1)-\Psi_4^\epsilon(2)}{N^\epsilon-3},\quad P_4^\epsilon=\Psi_4^\epsilon(2),
\end{align*}
% It then follows, from the convexity of $\Psi^\epsilon$, that
specifically
\begin{align*}
 \Psi_4^\epsilon(k)\le Q_4^\epsilon(k-2)+P_4^\epsilon,
\end{align*}
for all $k\in [2,N^\epsilon-1]$ with equality for $k=2$ and $k=N^\epsilon-1$:
\begin{align}
 \overline w_2^\epsilon = e^{P_4^\epsilon},\quad  \overline w_{N^\epsilon-1}^\epsilon = e^{Q_4^\epsilon(N^\epsilon-3)+P_4^\epsilon}.\label{eq:AepsBepsProp}
\end{align}

% We also notice that 
% \begin{align*}
%  a^\epsilon = \log 4(1+o(1)),\quad B(\epsilon) = \log \epsilon.
% \end{align*}
Moreover, $Q_4^\epsilon=o(1)$, see \eqref{AepsBepsApp} which we prove below. 
Consequently, for all $0<\delta\le \frac34$, we have
\begin{align*}
 \sum_{k=2}^{N^\epsilon-1}  \overline w_k^\epsilon \delta^k 
 &\le \sum_{k=2}^{N^\epsilon -1}  e^{Q_4^\epsilon(k-2)+P_4^\epsilon}\delta^k\\
 &\le e^{P_4^\epsilon}\delta^2+\int_2^{\infty} e^{Q_4^\epsilon(k-2)+P_4^\epsilon} \delta^k dk \\
 &\le C\delta^{2} e^{P_4^\epsilon}.
%  &= \left(1+\frac{1}{\log \delta^{-1} -Q_4^\epsilon}\right)\delta^{2} e^{P_4^\epsilon}.
\end{align*}
% and all $0<\epsilon\ll 1$.
 Here we have used that 
\begin{align*}
 \int_{2}^{\infty} e^{a(k-2)} \delta^k dk= \frac{1}{\log \delta^{-1} -a}\delta^{2}\quad \forall\,0<\delta<e^{-a}.
\end{align*}
% for all $\delta>0$ small enough.
% We estimate $\sum_{k=\lfloor k_m(\epsilon)\rfloor+1}^{N^\epsilon-1} \epsilon^{-2} \overline w_k^\epsilon \delta^k$ in a similar way: Let
% \begin{align*}
%  C(\epsilon) = \frac{\Psi^\epsilon(N^\epsilon-1)-\Psi^\epsilon(k_m(\epsilon))}{N^\epsilon-1-k_m(\epsilon)},\quad D(\epsilon) = \Psi^\epsilon(N^\epsilon-1),
% \end{align*}
% then 
% \begin{align*}
%  \Phi_\epsilon(k)\le C(\epsilon) (k-(N^\epsilon-1))+D(\epsilon),
% \end{align*}
% again using the convexity of $\Phi_\epsilon$. 
% Here 
% \begin{align*}
% %  C(\epsilon) = (1+a^\epsilon) \log \epsilon^2 (1+o(1)),\quad 
%  C(\epsilon) = \log 4(1+o(1)),\quad 
% D(\epsilon) = (1-\alpha+a^\epsilon) \log \epsilon^{-1}(1+o(1)).
% \end{align*}
% Here we have used that $\vert \alpha^\epsilon \log \alpha^\epsilon \vert \le e^{-1}$.
% % for $K\ge 1$ given that $a^0>0$.
% We have
% \begin{align*}
% \sum_{k=\lfloor k_m(\epsilon)\rfloor+1}^{N^\epsilon-1} \epsilon^{-2} \overline w_k^\epsilon \delta^k &\le\frac{1}{1-\epsilon} \int_{k_m(\epsilon)}^{N^\epsilon-1} e^{C(\epsilon)(k-(N^\epsilon-1))+D(\epsilon)+k\log \delta}dk\\
% &\le\frac{2}{C(\epsilon)+\log \delta}\left(e^{D(\epsilon)}\delta^{N^\epsilon-1}-e^{\Phi_\epsilon(k_m(\epsilon))}\delta^{k_m(\epsilon)} \right).
% % &\le \frac{2}{C(\epsilon)+\log \delta} e^{D(\epsilon)}\delta^{N^\epsilon-1}.
% \end{align*}
% The result follows.

To complete the proof of item \ref{C3}, we just have to prove the asymptotics in \eqref{AepsBepsApp}. 
The asymptotics of $P_4^\epsilon$ follows from item \ref{C0}, so we focus on $Q_4^\epsilon$. For this, we use Stirling's approximation in the form \eqref{stirling} for $N^\epsilon\gg 1$:
\begin{equation}\label{eq:Qepsapp}
\begin{aligned}
 Q_4^\epsilon &= \frac{1}{N^\epsilon-3}\log \frac{\overline w_{N^\epsilon-1}^\epsilon}{\overline w_2^\epsilon}\\
 &=\frac{1}{N^\epsilon-3}\log\left( \frac{\Gamma(1+\alpha^\epsilon)}{\Gamma(2+a^\epsilon)}\frac{\Gamma(N^\epsilon-1+a^\epsilon)}{\Gamma(N^\epsilon+\alpha^\epsilon-2)}\right)\\
 &=\frac{1}{N^\epsilon-3}\log \left(\frac{\Gamma(1+\alpha^\epsilon)}{\Gamma(2+a^\epsilon)}(1+o(1))(N^{\epsilon})^{a^\epsilon+1-\alpha^\epsilon}\right),
\end{aligned}
\end{equation}
using \eqref{barwkeps}, $\epsilon^{-1}= N^\epsilon+\alpha$ and
\begin{align*}
 \frac{\Gamma(N^\epsilon-1+a^\epsilon)}{\Gamma(N^\epsilon+\alpha^\epsilon-2)} = (1+o(1))\frac{(N^\epsilon)^{a^\epsilon-1}}{(N^\epsilon)^{\alpha^\epsilon-2}},
\end{align*}
in the last equality of \eqref{eq:Qepsapp}.

\textit{Proof of item \ref{kinf}}.
For \eqref{kinf}, we use the reflection formula \eqref{reflection} and Sterling's approximation in the form \eqref{stirling} for $k\rightarrow \infty$.

\textit{Proof of item \ref{est1}}. It is easy to verify the claim for all $4\le k\le  k_0$ for any $k_0>0$ fixed and $0<\epsilon\ll 1$ by using item \ref{C0}. We therefore consider $k_0< k\le N^\epsilon+1$ with $k_0>0$ fixed large and write
\begin{align*}
  \overline w_j^\epsilon =:e^{\Phi_6^\epsilon(j)}.
\end{align*}
Again, $\Phi_6^\epsilon$ is convex for any $j\in [2,N^\epsilon-1]$ and therefore
\begin{align*}
 \Phi_6^\epsilon(j) \le Q^\epsilon_6(j-2)+P_6^\epsilon,
\end{align*}
where $Q^\epsilon_6$ and $P_6^\epsilon$ are chosen such that equality holds for $j=2$ and $j=k-2$:
\begin{align}\label{eq:Phi6cond}
 Q^\epsilon_6 = \frac{1}{k-4} \log \frac{ \overline w_{k-2}^\epsilon}{ \overline w_2^\epsilon},\quad e^{P_6^\epsilon} =  \overline w_2^\epsilon.
\end{align}
By the convexity of $\Phi_6^\epsilon$ it follows that $Q^\epsilon_6$ is increasing. Therefore by item \ref{C0} and \eqref{stirling}
\begin{align*}
 Q^\epsilon_6 &\ge \log \epsilon +\mathcal O(\epsilon) + \log \left(\frac{\Gamma(k_0-2+a^\epsilon)}{\Gamma(2+a^\epsilon)}\right)^\frac{1}{k_0-4}\\
 &= \log \epsilon +\mathcal O (\epsilon) +\log k_0 (1+o_{k_0\rightarrow \infty}(1)),
\end{align*}
for all $k_0\le k\le N^\epsilon-1$. In turn, we can assume that
\begin{align*}
 \xi\epsilon^{-1}e^{Q_6^\epsilon}\ge 2\quad \forall\, k\in [k_0, N^\epsilon-1].
\end{align*}
This allow us to estimate the sum as a geometric sum:
 \begin{align*}
       \sum_{j=2}^{k-2} (\xi^{-1}\epsilon)^{k-2-j} \overline w_j^\epsilon &\le 
       (\xi^{-1}\epsilon)^{k-2} e^{P^\epsilon_6} \sum_{j=2}^{k-2} (\xi\epsilon^{-1}e^{Q_6^\epsilon})^{j}\\
       &\le 2 (\xi^{-1}\epsilon)^{k-2} e^{P^\epsilon_6} (\xi\epsilon^{-1}e^{Q_6^\epsilon})^{k-2}\\
       &\le 2 e^{Q_6^\epsilon(k-2)+P^\epsilon_6}\\
       &=2\overline w_{k-2}^\epsilon.
      \end{align*}

      \textit{Proof of item \ref{est2}}. It is easy to verify the claim for all $4\le k< k_0$ for any $k_0>0$ and $0<\epsilon\ll 1$ by using item \ref{C0}. We therefore consider $k_0\le k\le N^\epsilon+1$ with $k_0>0$ fixed large and write
\begin{align*}
 \overline w_j^\epsilon =:e^{\Phi_6^\epsilon(j)},
\end{align*}
as in the proof of item \ref{est1},
      with 
      \begin{align*}
       \Phi_6^\epsilon(j)\le Q_6^\epsilon (j-2)+P_6^\epsilon,
      \end{align*}
for all $j\in [2,k-2]$ with equality for $j=2$ and $j=k-2$. %Proceeding as in the proof of item \ref{est1}, 
We may assume that $k_0>0$ is such that
\begin{align*}
 (\xi^{-1} \epsilon e^{-2Q_6^\epsilon+P_6^\epsilon})\le \frac12 \quad \forall\, k\in [k_0,N^\epsilon+1],
\end{align*}
for all $0<\epsilon\ll 1$. In this way, we estimate can estimate the sum as a geometric sum
\begin{align*}
 \sum_{l=2}^{\lfloor \frac{k}{2}\rfloor } (\xi^{-1}\epsilon)^{l-2} (\overline w_2^\epsilon)^{l-1} \overline w_{k-2(l-1)}^\epsilon& \le  (\xi\epsilon^{-1})^{2} \sum_{l=2}^{\lfloor \frac{k}{2}\rfloor } (\xi^{-1}\epsilon)^{l} e^{P_6^\epsilon (l-1)}e^{Q_6^\epsilon(k-2l) + P_6^\epsilon}\\
 &\le (\xi\epsilon^{-1})^{2} e^{Q_6^\epsilon k} \sum_{l=2}^{\lfloor \frac{k}{2}\rfloor } \left(\xi^{-1}\epsilon e^{-2Q_6^\epsilon+P_6^\epsilon}\right)^l\\
 &\le 2(\xi\epsilon^{-1})^{2}e^{Q_6^\epsilon k} \left(\xi^{-1}\epsilon e^{-2Q_6^\epsilon+P_6^\epsilon}\right)^2\\
 &=2  \overline w_{2}^\epsilon \overline w_{k-2}^\epsilon,
\end{align*}
for all $k_0<k\le N^\epsilon+1$. Here we have used \eqref{eq:Phi6cond} in the last equality.

\end{proof}

  In contrast to the analysis of the center manifold for $\epsilon=0$, we are in the present case of $\epsilon>0$  only interested in estimating the partial sum of \eqref{barmeps}:
 \begin{align*}
  \overline y = \sum_{k=2}^{N^\epsilon-1} \overline m_k^\epsilon\overline x^k, \quad m_k^\epsilon = (-1)^k \overline w_k^\epsilon \overline S_k^\epsilon,\quad \overline S_k^\epsilon = \sum_{j=2}^k \frac{(-1)^j \epsilon \rsp{\overline{\mathcal G}^\epsilon[\overline m^\epsilon]_j} }{\overline w_j^\epsilon(1-\epsilon j)},
 \end{align*}
 where 
 \begin{align*}
  N^\epsilon = \lfloor \epsilon^{-1}\rfloor,
 \end{align*}
recall \eqref{Nepsdef}. (We will deal with the remainder later, see Section \ref{sec:Top}).
 We therefore define the semi-norm
 \begin{align}\label{eq:seminorm}
  \Vert \sum_{k=2}^{\infty} \overline y_k \overline x^k\Vert :=\sup_{k\in [2,{N^\epsilon-1}]} \frac{\vert \overline y_k\vert}{\overline w_k^\epsilon},
 \end{align}
 on the set of formal series $\overline y=\sum_{k=2}^{\infty} \overline y_k \overline x^k$.

\begin{lemma}
 Consider $\overline y(\overline x) = \sum_{k=2}^{N^\epsilon-1} \overline y_k \overline x^k$ and define $(\overline y^l)_k, k=2l,\ldots,l(N^\epsilon-1)$ by
 \begin{align}
  \overline y(\overline x)^l =: \sum_{k=2l}^{l(N^\epsilon-1)} (\overline y^l)_k \overline x^k.\label{eq:ylkz}
 \end{align}
 Then there exists a $C=C(a^0)>0$ such that for any $l\in \mathbb N\setminus \{1\}$ and all $1\le p\le l$ the following holds true:
 \begin{equation}\label{eq:ylkeps}
 \begin{aligned}
 \vert (\overline y^l )_k\vert &\le \Vert y\Vert^{l} \begin{pmatrix} l-1 \\ p-1\end{pmatrix} C^{l-1} (\overline w_2^\epsilon)^{l-p} ( \overline w_{N^\epsilon-1}^\epsilon)^{p-1} ( \overline w_{k-(p-1)(N^\epsilon-1)-2(l-p)}^\epsilon),\\
 \forall\,k&\in [(p-1)(N^\epsilon-1)+2(l-p+1),p(N^{\epsilon}-1)+2(l-p)],
 \end{aligned}
 \end{equation}
 for all $0<\epsilon\ll 1$.
 Here $$\begin{pmatrix}                                                                            l-1\\p-1\end{pmatrix}$$ denotes the binomial coefficient for any $1\le p\le l$.
 
 In particular, for $p=1$:
  \begin{equation}\label{eq:ylkpEq1}
 \begin{aligned}
 \vert (\overline y^l )_k\vert &\le \Vert y\Vert^{l}  C^{l-1} (\overline w_2^\epsilon)^{l-1}  \overline w_{k-2(l-1)}^\epsilon,\\
 \forall\,k&\in [2l,N^{\epsilon}-1+2(l-1)],
 \end{aligned}
 \end{equation}
  for all $0<\epsilon\ll 1$.

\end{lemma}
\begin{proof}
The claim is proven by induction, with the base case being $l=2$, $p=1$ and $p=2$.
 
 \textit{The base case: $(l,p)=(2,1),\,(2,2)$ }.
 For $l=2$, we have by Cauchy's product formula:
 \begin{align*}
  \vert (\overline y^2 )_k\vert \le \Vert y\Vert^2 \sum_{j=\max(2,k-(N^\epsilon-1))}^{\min(k-2,N^\epsilon-1)}  \overline w_j^\epsilon \overline w_{k-j}^\epsilon.
 \end{align*}
We first consider $p=1$: $4\le k\le  (N^\epsilon-1)+2=N^\epsilon+1$. Then by item \ref{C5} of Lemma \ref{lemma:wkeps}, see \eqref{eq:C51}, we conclude that
\begin{align*}
\vert (\overline y^2 )_k\vert \le \Vert y\Vert^2 C  \overline w_2^\epsilon \overline w_{k-2}^\epsilon.
\end{align*}
Next, for $p=2$:
\begin{align*}
 \vert (\overline y^2 )_k\vert \le \Vert y\Vert^2 \sum_{j=k-(N^\epsilon-1)}^{N^\epsilon-1}  \overline w_j^\epsilon\overline w_{k-j}^\epsilon\le \Vert y\Vert^2 C \overline w_{N^\epsilon-1}^\epsilon \overline w_{k-(N^\epsilon-1)}^\epsilon,
 \end{align*}
 using \eqref{eq:C52}. 
 
 \textit{Induction step}. The induction proceeds in two steps: We assume that the claim is true for all $l\in \mathbb N\setminus \{1\}$ and all $1\le p\le l$. We then first proof that it is true for $l+1$, $1\le p\le l$. Subsequently, we consider $p=l+1$. 
 
 We assume that \eqref{eq:ylkeps} holds true. Then by using Cauchy's product formula we find that
 \begin{align*}
  (\overline y^{l+1} )_k = \sum_{j=\max(2l,k-(N^\epsilon-1))}^{\min (k-2,l(N^\epsilon-1))} (\overline y^l )_j \overline y_{k-j}.
 \end{align*}
For $p=1$ and $k\in [2(l+1),N^\epsilon-1+2l]$, we find 
\begin{align*}
   \vert (\overline y^{l+1} )_k\vert &\le\Vert y\Vert^{l+1} C^{l-1} ( \overline w_2^\epsilon)^{l-1} \sum_{j=2l}^{k-2}  \overline w_{j-2(l-1)}^\epsilon \overline w_{k-j}^\epsilon\\
   &\le \Vert y\Vert^{l+1} C^{l-1} ( \overline w_2^\epsilon)^{l-1} \sum_{j=2}^{k-2l} \overline w_{j}^\epsilon  \overline w_{(k-2(l-1))-j}^\epsilon\\
   &\le \Vert y\Vert^{l+1} C^{l} ( \overline w_2^\epsilon)^{l} \overline w_{k-2l}^\epsilon,
\end{align*}
using \eqref{eq:C51}, 
which proves \eqref{eq:ylkeps} with $l\rightarrow l+1$ and $p=1$. Next, for $2\le p\le l$, we find completely analogously that
 \begin{align*}
 \vert (\overline y^{l+1} )_k\vert &\le  \sum_{j=k-(N^\epsilon-1)}^{k-2}\vert (\overline y^l )_j\vert \vert \overline y_{k-j}\vert \\
 &\le \sum_{j=k-(N^\epsilon-1)}^{(p-1)(N^\epsilon-1)+2(l-p+1)} \vert (\overline y^l )_j\vert \vert \overline y_{k-j}\vert +\sum_{j=(p-1)(N^\epsilon-1)+2(l-p+1)}^{k-2} \vert (\overline y^l )_j\vert \vert \overline y_{k-j}\vert,
 \end{align*}
 for $$k\in [(p-1)(N^\epsilon-1)+2(l-p+1),p(N^\epsilon-1)+2(l-p+1)].$$ Therefore by \eqref{eq:ylkeps} (for $(l,p)$ and $(l,p)\rightarrow (l,p-1)$):
 \begin{align*}
  \vert (\overline y^{l+1} )_k\vert &\le \Vert \overline y\Vert^{l+1} \begin{pmatrix} 
  l-1 \\
        p-2                                                         \end{pmatrix}C^{l-1}( \overline w_2^\epsilon)^{l-p+1} (\overline w_{N^\epsilon-1}^\epsilon)^{p-2} \\
        &\times \sum_{j=k-(N^\epsilon-1)}^{(p-1)(N^\epsilon-1)+2(l-p+1)}\overline w_{j-(p-2)(N^\epsilon-1)-2(l-p+1)}^\epsilon \overline w_{k-j}^\epsilon\\
        &+\Vert y\Vert^{l+1} \begin{pmatrix} 
  l-1 \\
        p -1                                                        \end{pmatrix}C^{l-1}( \overline w_2^\epsilon)^{l-p} (\overline w_{N^\epsilon-1}^\epsilon)^{p-1} \\
        &\times \sum_{j=(p-1)(N^\epsilon-1)+2(l-p+1)}^{k-2} \overline w^\epsilon_{j-(p-1)(N^\epsilon-1)-2(l-p)} \overline w^\epsilon_{k-j}\\
       &\le \Vert y\Vert^{l+1}  C^{l} ( \overline w_2^\epsilon)^{l-p+1} (\overline w_{N^\epsilon-1}^\epsilon)^{p-1} \left(\begin{pmatrix} 
  l-1 \\
        p-2                                                         \end{pmatrix}+\begin{pmatrix} 
  l-1 \\
        p -1                                                        \end{pmatrix}\right) \\
        &\times \overline w_{k-(p-1)(N^\epsilon-1)-2(l-p+1)}^\epsilon,
\end{align*}
using \eqref{eq:C51} and \eqref{eq:C52} to estimate the two sums. Then as
\begin{align}
\begin{pmatrix} 
  l-1 \\
        p-2                                                         \end{pmatrix}+\begin{pmatrix} 
  l-1 \\
        p  -1                                                       \end{pmatrix}=\begin{pmatrix} 
  l \\
        p    -1                                                   \end{pmatrix}\end{align}
        the claim follows.

        We are left with proving that the claim holds true for $p=l+1$ and 
        $$k\in [l(N^\epsilon-1)+2, (l+1)(N^\epsilon-1)],$$
        where
        \begin{align*}
        \vert (\overline y^{l+1} )_k\vert &\le  \sum_{j=k-(N^\epsilon-1)}^{l(N^\epsilon-1)}\vert (\overline y^l )_j\vert \vert \overline y_{k-j}\vert.
\end{align*}
By the induction assumption, we have
\begin{align*}
\vert (\overline y^l )_k\vert \le \Vert \overline y\Vert^l C^{l-1} (\overline w_{N^\epsilon-1}^\epsilon)^{l-1} \overline w_{k-(l-1)(N^\epsilon-1)}^\epsilon,
\end{align*}
for all $$k\in [(l-1)(N^\epsilon-1)+2, l(N^\epsilon-1)],$$
see \eqref{eq:ylkeps} with $p= l$. Therefore
\begin{align*}
  \vert (\overline y^{l+1} )_k\vert &\le \Vert \overline y\Vert^{l+1} C^{l-1} ( \overline w_{N^\epsilon-1}^\epsilon)^{l-1}   \sum_{j=k-(N^\epsilon-1)}^{l(N^\epsilon-1)} \overline w_{j-(l-1)(N^\epsilon-1)}^\epsilon \overline w_{k-j}^\epsilon \\
  &\le \Vert \overline y\Vert^{l+1} C^{l} (\overline w_{N^\epsilon-1}^\epsilon)^{l}\overline w_{k-l(N^\epsilon-1)}^\epsilon,
\end{align*}
using \eqref{eq:C52}. This proves \eqref{eq:ylkeps} with $l\rightarrow l+1$ and $p=l+1$ and completes the proof. 

\end{proof}

By using Lemma \ref{lemma:wkeps} item \ref{C3}, we obtain the following bound on $(\overline y^l)_k$
\begin{lemma}\label{lemma:yklNew}
Consider $\overline y(\overline x) = \sum_{k=2}^{N^\epsilon-1} \overline y_k \overline x^k$ and recall the definition of $(\overline y^l)_k$ in \eqref{eq:ylkz}. Then there is a new $C>0$ such that 
\begin{align}
\vert (\overline y^l)_k\vert\le \Vert \overline y\Vert^l C^{l-1} e^{(-2Q_4^\epsilon+P_4^\epsilon)l} e^{Q_4^\epsilon k}\quad \forall\,2l\le k\le l(N^\epsilon-1), \label{eq:ylkepsfinal}
\end{align}
for all $0< \epsilon\ll 1$.
Here $Q_4^\epsilon$ and $P_4^\epsilon$ are defined in \eqref{AepsBeps}.
\end{lemma}
\begin{proof}
We will use \eqref{eq:ylkeps}, repeated here for convenience:
 \begin{equation}\label{eq:ylkeps2}
 \begin{aligned}
 \vert (\overline y^l )_k\vert &\le \Vert y\Vert^{l} \begin{pmatrix} l-1 \\ p-1\end{pmatrix} C^{l-1} \underline{(\overline w_2^\epsilon)^{l-p} ( \overline w_{N^\epsilon-1}^\epsilon)^{p-1}  \overline w_{k-(p-1)(N^\epsilon-1)-2(l-p)}^\epsilon},\\
 \forall\,k&\in [(p-1)(N^\epsilon-1)+2(l-p+1),p(N^{\epsilon}-1)+2(l-p)],
 \end{aligned}
 \end{equation}
where $1\le p\le l$. Using \eqref{AepsBeps} we have 
\begin{align*}
 \overline w_2^\epsilon = e^{P_4^\epsilon}, \quad \overline w_{N^\epsilon-1}^\epsilon = e^{Q_4^\epsilon (N^\epsilon-3)+P_4^\epsilon} \quad \mbox{and}\quad \overline w_{k}^\epsilon\le e^{Q_4^\epsilon(k-2)+P_4^\epsilon}\quad \forall\, 2\le k \le N^\epsilon+1,
\end{align*}
and we can therefore estimate the underlined factor in \eqref{eq:ylkeps2} as follows:
\begin{align*}
 &\underline{( \overline w_2^\epsilon)^{l-p} ( \overline w_{N^\epsilon-1}^\epsilon)^{p-1}  \overline w_{k-(p-1)(N^\epsilon-1)-2(l-p)}^\epsilon}\le  e^{P_4^\epsilon (l-p)}e^{(Q_4^\epsilon (N^\epsilon-3)+P_4^\epsilon)(p-1)} e^{(Q_4^\epsilon (k-\{\cdots\}-2)+P_4^\epsilon)},
\end{align*}
where $\{\cdots\}=(p-1)(N^\epsilon-1)+2(l-p)$. By simplifying, we obtain
\begin{align}
 \underline{( \overline w_2^\epsilon)^{l-p} ( \overline w_{N^\epsilon-1}^\epsilon)^{p-1}  \overline w_{k-(p-1)(N^\epsilon-1)-2(l-p)}^\epsilon} \le e^{(-2Q_4^\epsilon+P_4^\epsilon)l} e^{Q_4^\epsilon k}.\label{eq:underlineest}
\end{align}
Subsequently, we use  
\begin{align} \label{eq:binomest}
 \begin{pmatrix}
  l-1\\
  p-1
 \end{pmatrix}\le \sum_{q=0}^{l-1}  \begin{pmatrix}
  l-1\\
  q
 \end{pmatrix} = 2^{l-1},
%  \begin{pmatrix}
%   l-1\\
%   \lfloor \frac{l+1}{2}\rfloor
%  \end{pmatrix} =(\sqrt{2}+o_{l\rightarrow \infty}(1)) \frac{2^{l-1}}{\sqrt{(l-1)\pi}},
\end{align}
for all $1\le p \le l$. Therefore \eqref{eq:ylkepsfinal} follows from \eqref{eq:ylkeps2}, \eqref{eq:underlineest} and \eqref{eq:binomest}. 
% with equality for $p=\lfloor \frac{l+1}{2}\rfloor$ (where the left hand side is maximized). 
\end{proof}

 % %For this purpose, it is (as for the center manifold above) useful to replace $x$ by $\delta \widetilde x$. 
 %Then $\overline m_k^\epsilon\mapsto \widetilde m_k^\epsilon:=\delta^k \overline m_k^\epsilon$. We drop the tildes
% \begin{align*}
% m_k^0=  (-1)^k \delta^k S^0_k  w_k^0,
% \end{align*}
%  
 %and consider the fixed-point equation:
 
%  This time, however, we are only interested in estimating the partial sum of \eqref{barmeps}:
%  \begin{align*}
%   \overline y = \sum_{k=2}^{N^\epsilon-1} \overline m_k^\epsilon\overline x^k, \quad m_k^\epsilon = (-1)^k \epsilon^{-2}\overline w_k^\epsilon \overline S_k^\epsilon\delta^k,
%  \end{align*}
%  where 
%  \begin{align*}
%   N^\epsilon = \lfloor \epsilon^{-1}\rfloor,
%  \end{align*}
% recall \eqref{Nepsdef}. 
%  We now estimate $T^\epsilon(\overline y)$ in terms of the semi-norm
 
%  Notice that $1-\epsilon k>0$ for all $k\le N^\epsilon-1$ and all $\epsilon>0$; specifically,
%  \begin{align*}
%   1-\epsilon (N^\epsilon-1) = \epsilon (1+\alpha^\epsilon), 
%  \end{align*}
% using that $N^\epsilon = \epsilon^{-1}-\alpha^{\epsilon}$. 
%  We have
 \begin{lemma}\label{lemma:gkestp1}
   Recall the definition of the semi-norm $\Vert \cdot \Vert$ in \eqref{eq:seminorm} and suppose that $\Vert \overline y\Vert\le C$ with $C>0$. Then there is a $\overline K=\overline K(C)>0$, independent of $\mu$ and $\epsilon$, such that 
   \begin{align*}
   \begin{cases} \vert  \epsilon \rsp{\overline{\mathcal G}^\epsilon[\overline y]_2}\vert &\le B \rho^{-2} \epsilon,\\
    \vert \epsilon \rsp{\overline{\mathcal G}^\epsilon[\overline y]_3}\vert &\le B \rho^{-3} \epsilon^2,\\
%    \end{cases}
\vert \epsilon \rsp{\overline{\mathcal G}^\epsilon[\overline y]_k}\vert &\le B \rho^{-k} \epsilon^{k-1}+\mu  \epsilon^2 \overline K  \overline w_{k-2}^\epsilon\quad \forall\,4\le k\le N^\epsilon+1,
\end{cases}
   \end{align*}
   for all $0<\epsilon\ll 1$.
%  \begin{align*}
%   \Vert T^\epsilon \left( \overline y\right) \Vert &\le \delta^2 \log \delta^{-2}\Vert \sum_{k=2}^\infty \overline h_k x^k\Vert.
% %   &=\frac{\delta^2}{1-\delta}\Vert \sum_{k=2}^\infty \overline h_k x^k\Vert.
%  \end{align*}
 \end{lemma}
 \begin{proof}
 We use \eqref{gk0eps}:
 $$\vert \rsp{\overline{\mathcal G}^\epsilon[y]_k}\vert \le \vert \overline f_k^\epsilon\vert  +\mu \vert \rsp{\overline{\mathcal H}^\epsilon[y]_k}\vert \quad \forall\,k\ge 2. $$
 The first term on the right hand side is directly estimated by \eqref{eq:gkcond0}:
 \begin{align*}
  \vert \overline f^\epsilon_k\vert\le B \rho^{-k}\epsilon^{k-2},
 \end{align*}
for all $k\in \mathbb N\setminus\{1\}$. We therefore focus on the second term, which vanishes for $k=2$ and $k=3$. 
By using \eqref{eq:gkcond0}, \eqref{gk12eps}, $$\vert \overline y_j\vert \le \Vert \overline y\Vert \overline w_j^\epsilon \quad \forall\,j\in [2,N^\epsilon-1],$$ and \eqref{eq:ylkpEq1}, we obtain
 \begin{align*}
 \vert \rsp{ \overline{\mathcal H}^\epsilon[y]_k}\vert &\le   \sum_{j=2}^{k-2} \rho^{-k+j-1}\epsilon^{k-j-1}  \Vert \overline y\Vert \overline w_j^\epsilon\\
 &+\sum_{l=2}^{\lfloor \frac{k}{2}\rfloor } \sum_{j={2l}}^k \rho^{-k+j-l} \epsilon^{k-j+l-2}  \Vert \overline y\Vert^l C^{l-1} ( \overline w_2^\epsilon)^{l-1}  \overline w_{j-2(l-1)}^\epsilon\\
 &=  \Vert \overline y\Vert \rho^{-1} \epsilon \sum_{j=2}^{k-2} (\rho^{-1}\epsilon)^{k-2-j}  \overline w_j^\epsilon\\
 &+\sum_{l=2}^{\lfloor \frac{k}{2}\rfloor } \rho^{-l} \epsilon^{l-2} \Vert \overline y\Vert^l C^{l-1} (\overline w_2^\epsilon)^{l-1} \sum_{j={2}}^{k-2(l-1)} (\rho^{-1} \epsilon)^{k-2(l-1)-j}   \overline w_{j}^\epsilon,
%  &\le 
 \end{align*}
 for all $4\le k\le N^\epsilon+1$. We now use \eqref{eq:est1} and \eqref{eq:est2}, respectively:
 \begin{align*}
 \vert \rsp{\overline{\mathcal H}^\epsilon[y]_k}\vert & \le \Vert \overline y\Vert \rho^{-1} \epsilon C \overline w_{k-2}^\epsilon\\
 &+ \sum_{l=2}^{\lfloor \frac{k}{2}\rfloor } \rho^{-l} \epsilon^{l-2} \Vert \overline y\Vert^l C^{l} ( \overline w_2^\epsilon)^{l-1} %\sum_{j={2}}^{k-2(l-1)} (\rho^{-1} \epsilon)^{k-2(l-1)-j}  
 \overline w_{k-2(l-1)}^\epsilon\\
 &\le  \Vert \overline y\Vert \rho^{-1} \epsilon C  \overline w_{k-2}^\epsilon\\
 &+\rho^{-2} \Vert \overline y\Vert^{2} C^{2} \sum_{l=2}^{\lfloor \frac{k}{2}\rfloor }   \left(\rho^{-1} \Vert \overline y\Vert C \epsilon\right)^{l-2}( \overline w_2^\epsilon)^{l-1} %\sum_{j={2}}^{k-2(l-1)} (\rho^{-1} \epsilon)^{k-2(l-1)-j}  
 \overline w_{k-2(l-1)}^\epsilon\\
 &\le   \overline K \epsilon \overline w_{k-2}^\epsilon,
 \end{align*}
with $\overline K=\overline K(\Vert y\Vert,a^0,\rho)>0$ large enough. Here we have used that $ \overline w_2^\epsilon=\mathcal O(\epsilon)$ cf. Lemma \ref{lemma:wkeps} item \ref{C0}.
 %We now use $\vert \overline y_j\vert \le \Vert \overline y\Vert (\epsilon^{-2} \overline w_j^\epsilon)$ for all $j\in [2,N^\epsilon-1]$ and \eqref{eq:ylkpEq1}: 

 \end{proof}

 This leads to the following important estimate:

\begin{lemma} \label{lemma:sumquantity}
Fix $C>0$ and define
$$F = B\sum_{j=2}^{\infty} \frac{ \rho^{-j}}{\Gamma(j+a^0)}.$$
Then the following holds for all $0\le \mu<\mu_0$ with $\mu_0>0$ small enough: 
% Consider the analytic weak-stable manifold $\overline m^\epsilon=\sum_{k=2}^\infty\overline m_k^\epsilon \overline x^k$ and let 
% The quantity
\begin{align*}
 \left| \sum_{j=2}^k \frac{(-1)^j \epsilon \rsp{\overline{\mathcal G}^\epsilon[\overline y]_j} }{\overline w^\epsilon_j (1-\epsilon j)}\right| \le 2F\quad \forall\,2\le k \le N^\epsilon,\, \Vert \overline y\Vert\le C,
 \end{align*}
%  satifies
% \begin{align*}
% \vert \overline S^\epsilon_k\vert \le  2F\quad \forall\,2\le k \le N^\epsilon,\,\forall\,\Vert y\Vert\le C,
% \end{align*}
for all $0<\epsilon\ll 1$.

% Moreover, define
%  \begin{align*}
%   \Vert  \overline m^\epsilon\Vert\le v(\delta)=\mathcal O(\delta^2 \log \delta^{-1}).
%  \end{align*}
% %  for all $N^\epsilon\gg 1$.
%  Moreover, write 
%  \begin{align}\label{eq:bargk}
%  \overline g(\overline x,\overline m_k^\epsilon(\overline x),\epsilon)=:\sum_{k=2}^\infty \epsilon^{k-2} g^\epsilon_k \overline x^k,
%  \end{align}
%  as a formal series, 
%  and in turn define
   \end{lemma}
%  for all $k\ge 2$. Then
%  \begin{align}
%  \vert \overline S_{k}^\epsilon\vert  \le K_2,
% \quad \overline m_k^\epsilon \le K_2 \epsilon^{-2} \overline w_k^\epsilon \delta^k,\label{eq:mkepsest}
% \end{align}
% for all $2\le k\le N^\epsilon$ and all $0<\epsilon\ll 1$. %Here $C_2$ is independent of $\alpha^\epsilon$. 
%  \end{lemma}
%  \begin{proof}
\begin{proof}
  Let $\overline K=\overline K(C)>0$ be the constant in Lemma \ref{lemma:gkestp1}.
   We then estimate
   \begin{align*}
   \left| \sum_{j=2}^k \frac{(-1)^j \epsilon \rsp{\overline{\mathcal G}^\epsilon[\overline y]_j} }{\overline w^\epsilon_j (1-\epsilon j)}\right|&\le  B\sum_{j=2}^{N^\epsilon} \frac{ \rho^{-j}\epsilon^{j-1}}{\overline w_j^\epsilon (1-\epsilon j))}+\mu \overline K \sum_{j=4}^{N^\epsilon} \frac{\epsilon^2 \overline w_{j-2}^\epsilon}{\overline w_j^\epsilon(1-\epsilon j)},
   \end{align*}
   for all $2\le k\le N^\epsilon$
   using Lemma \ref{lemma:gkestp1}. Then by the definition of $\overline w_k^\epsilon$ \eqref{barwkeps} and \eqref{eq:Gamma}, we have
   \begin{align*}
    \frac{\overline w_{j-2}^\epsilon}{\overline w_j^\epsilon}&=\frac{\Gamma(\epsilon^{-1} -j+2)\Gamma(j-2+a^\epsilon)}{\Gamma(\epsilon^{-1}-j) \Gamma(j+a^\epsilon)}\\
    &=\frac{(\epsilon^{-1} -j+1)(\epsilon^{-1}-j)}{(j-1+a^\epsilon)(j-2+a^\epsilon)}.
   \end{align*}
Therefore by Lemma \ref{lemma:wkeps} item \ref{C2} we find that  
\begin{equation}\label{eq:gkestp1}
\begin{aligned}
 \left| \sum_{j=2}^k \frac{(-1)^j \epsilon \rsp{\overline{\mathcal G}^\epsilon[\overline y]_j}}{\overline w^\epsilon_j (1-\epsilon j)}\right|&\le 
B\sum_{j=2}^{N^\epsilon} \frac{ \rho^{-j}}{\Gamma(j+a^\epsilon)}+\mu \overline K \sum_{j=4}^{N^\epsilon} \frac{(1-\epsilon (j-1))(1-\epsilon j)}{(j-1+a^\epsilon)(j-2+a^\epsilon)(1-\epsilon j)}\\
&\le 
B\sum_{j=2}^{\infty} \frac{ \rho^{-j}}{\Gamma(j+a^\epsilon)}+\mu \overline K \sum_{j=4}^{\infty} \frac{1}{(j-1+a^\epsilon)(j-2+a^\epsilon)}.
\end{aligned}
\end{equation}
The result now follows.
   \end{proof}
%  By proceeding as in the proof of Lemma \ref{lemma:Tepsest}, we conclude that
%  \begin{align}\label{Skbound}
%   \vert \overline S_{k}^\epsilon \vert & \le \frac{\rho}{\rho-\Gamma(a^\epsilon) v(\delta)}\sum_{k=0}^\infty \frac{B\rho^{-k}}{\Gamma(k+a^\epsilon)}<\infty
%    \end{align}
%    for all $k\le N^{\epsilon}$. Here we have also used that $g^\epsilon_j$, $j=2,\ldots,N^\epsilon$ only depend upon $y=\sum_{k=2}^{N^\epsilon-1} \overline m_k^\epsilon\overline x^k$ which is bounded by $v(\delta)$  in the semi-norm $\Vert \cdot\Vert$. 
% %   From the majorant equation abovem we obtain directly that 
% %   \begin{align*}
% %    \overline m_k^\epsilon \le C_2 \overline w
% %   \end{align*}
% This in turn leads to the desired bound on $\overline m_k^\epsilon$, recall \eqref{barmkeps}.
%  \end{proof}

Following Lemma \ref{lemma:barSkeps}, we have that $\overline y = \overline m^\epsilon(\overline x)$ (as a power series) is a fixed-point of the nonlinear operator $\rsp{\mathcal P^\epsilon}$ defined by 
%  \begin{align*}
% %   \overline y = T^\epsilon \left(\epsilon \overline g(\overline x,\overline y,\epsilon)\right),
%  \end{align*}
% where $T^\epsilon$ is the linear operator defined on the set of formal series $ \sum_{k=2}^\infty \overline h_k x^k$ as follows:
\begin{align}\label{eq:Tepsdefn}
 \rsp{\mathcal P^\epsilon} \left( \overline y\right) = \sum_{k=2}^\infty (-1)^k  \overline w_k^\epsilon \sum_{j=2}^k \frac{(-1)^j \epsilon \rsp{ \overline{\mathcal G}^\epsilon[\overline y]_j} }{\overline w_j^\epsilon (1-\epsilon j))} x^k.
\end{align} 
By Lemma \ref{lemma:gkestp1} and Lemma \ref{lemma:sumquantity}, we have that there is a $\mu_0>0$ such that for all $0\le \mu<\mu_0$ the following estimate holds:
\begin{align*}
 \Vert \rsp{\mathcal P^\epsilon}(\overline y)\Vert \le 2F\quad \forall \,\Vert \overline y\Vert\le 2F,\,0<\epsilon\ll 1,
\end{align*}
 with respect to the semi-norm $\Vert \cdot \Vert$ (that only involves the finite sum) defined in \eqref{eq:seminorm}. \KUK{Then by proceeding as in Remark \ref{rem:ref2} (using induction on $k$), we directly obtain the following:}
\begin{proposition}\label{prop:seminormmeps}
There is a $\mu_0>0$, such that for all $0\le \mu<\mu_0$ the following holds true:
\begin{enumerate}
\item The analytic weak-stable manifold satisfies the following estimate
\begin{align*}
 \Vert \overline m^\epsilon \Vert \le 2F,
\end{align*}
for all $0<\epsilon\ll 1$.
\item The numbers
\begin{align*}
 \overline S_k^\epsilon:=\sum_{j=2}^k \frac{(-1)^j  \epsilon\rsp{\overline{\mathcal G}^\epsilon[\overline m^\epsilon]_j} }{ \overline w^\epsilon_j (1-\epsilon j)},\quad 2\le k\le N^\epsilon,
\end{align*}
are uniformly bounded with respect to $0<\epsilon\ll 1$, $\epsilon^{-1}\notin \mathbb N$.
\end{enumerate}
\end{proposition}
 
% The following result:

\begin{lemma}\label{lemma:SNepslimit}
Let $0\le \mu<\mu_0$ with $\mu_0>0$ small enough so that Proposition \ref{prop:seminormmeps} applies and so that the series $S_\infty^0$ from Lemma \ref{lemma:S0inf} is well-defined and absolutely convergent. Then $$\overline S^\epsilon_{N^\epsilon}\rightarrow S^0_\infty \quad \mbox{for}\quad N^{\epsilon}\rightarrow  \infty.$$
\end{lemma}
\begin{proof}
The proof is elementary, but since this result is crucial to the whole construction, we provide the full details:

For simplicity, we write $$\rsp{\overline{\mathcal G}^\epsilon[\overline m^\epsilon]_j=:\overline{\mathcal G}_j^\epsilon,\quad \mathcal G^0[\widehat m^0]_j= :\mathcal G_j^0},$$
in the following. 
% \fbox{There is a mistake here with the $c\epsilon$-term }
 By Lemma \ref{lemma:sumquantity}, $\vert \overline S^\epsilon_{N^\epsilon}\vert\le 2F$.
%  \begin{align*}
% \vert \overline S^\epsilon_{N^\epsilon}\vert \le\frac{1}{\alpha^\epsilon} \sum_{k,l}(c\epsilon)^{l} \frac{\epsilon^{l-1} \vert \overline g_{k,l}\vert }{\Gamma(k+a^\epsilon)}\Vert y\Vert^l
%  \end{align*}
Moreover, 
 $S^0_\infty=\sum_{j=2}^\infty \frac{(-1)^j \rsp{\mathcal G^0_j}}{w_j^0}$ is absolutely convergent, recall Lemma \ref{lemma:S0inf}.

 For fixed $j$ we have (recall item \ref{C0} of Lemma \ref{lemma:wkeps})
 \begin{align*}
  \overline w^\epsilon_j & =  \Gamma\left(j+a^0\right) \epsilon^{j-1}(1+o(1))=w_j^0 \epsilon^{j-1}(1+o(1)).
 \end{align*}
 Moreover, by Lemma \ref{convmbarkeps} we have $$\epsilon^{1-j} \epsilon \rsp{\overline{\mathcal G}^\epsilon_j} \rightarrow \rsp{\mathcal G^0_j},$$ and therefore
%   Consequently, upon using Lemma \ref{convmbarkeps}:
 \begin{align}\label{eq:SjepsSj0}
  \frac{(-1)^j  \epsilon \rsp{\overline{\mathcal G^\epsilon}_j} }{\overline w^\epsilon_j(1-\epsilon j)}\rightarrow \frac{(-1)^j \rsp{\mathcal G^0_j}}{\Gamma(j+a^0)},
 \end{align}
 as $\epsilon\rightarrow 0$ (fixed $j$).

Next, we estimate
\begin{equation}\label{thisest2}
 \begin{aligned}
  \left|\sum_{j=2}^{N^\epsilon}  \frac{(-1)^j \epsilon \overline{\mathcal G^\epsilon}_j }{ \overline w^\epsilon_j(1-\epsilon j)}-\sum_{j=2}^{\infty} \frac{(-1)^j\rsp{\mathcal G^0_j}}{w_j^0} \right|&\le  \left| \sum_{j=2}^{J} \left( \frac{(-1)^j  \epsilon \overline{\rsp{\mathcal G^\epsilon_j}} }{\overline w^\epsilon_j(1-\epsilon j)}-\frac{(-1)^j\rsp{\mathcal G^0_j}}{w_j^0}\right)\right|\\
   &+ \left| \sum_{j=J+1}^{N^\epsilon} \frac{(-1)^j  \epsilon \rsp{\overline{\mathcal G^\epsilon}_j } }{\overline w^\epsilon_j(1-\epsilon j)}\right|
   +  \left| \sum_{j=J+1}^\infty \frac{(-1)^j\rsp{\mathcal G^0_j}}{w_j^0}\right|\\
   &\le   \sum_{j=2}^{J} \left|\frac{(-1)^j  \epsilon \rsp{\overline{\mathcal G^\epsilon}_j} }{\overline w^\epsilon_j(1-\epsilon j)}-\frac{(-1)^j\rsp{\mathcal G^0_j}}{w_j^0}\right|\\
   &+ \sum_{j=J+1}^{\infty}\left(\frac{B\rho^{-j}}{\Gamma(j+a^\epsilon)}+ \frac{\mu \overline K}{(j-1+a^\epsilon)(j-2+a^\epsilon)}\right)\\
   & +   \sum_{j=J+1}^\infty\left(\frac{B \rho^{-j}}{\Gamma(j+a^0)}+\frac{\mu K}{(j-1+a^0)(j-2+a^0)}\right),
   \end{aligned}
   \end{equation}
   for any $2\le J\le N^\epsilon$,
    using $w_j^0=\Gamma(j+a^0)$, Lemma \ref{lemma:gkestp1} (see also \eqref{eq:gkestp1}) and Proposition \ref{prop:boundgk0} (see also \eqref{eq:boundgk0}). 
   Consequently, we have
   \begin{equation}\label{eq:SNepsS0est}
   \begin{aligned}
%    \left|\sum_{j=2}^{N^\epsilon}  \frac{(-1)^j   \epsilon \overline g^\epsilon(\cdot,\overline m^\epsilon(\cdot))_k) }{(\epsilon^{-2}\overline w^\epsilon_j)(1-\epsilon j)}-\sum_{j=2}^{\infty} \frac{(-1)^j\mathcal G^0_j}{w_j^0} \right| 
\left|\sum_{j=2}^{N^\epsilon}  \frac{(-1)^j \epsilon \rsp{\overline{\mathcal G^\epsilon}_j } }{ \overline w^\epsilon_j(1-\epsilon j)}-\sum_{j=2}^{\infty} \frac{(-1)^j\rsp{\mathcal G^0_j}}{w_j^0} \right|&\le \sum_{j=2}^{J} \left|\frac{(-1)^j  \epsilon \rsp{ \overline{\mathcal G^\epsilon}_j} }{\overline w^\epsilon_j(1-\epsilon j)}-\frac{(-1)^j\rsp{\mathcal G^0_j}}{w_j^0}\right|\\
   &+ 2\mu (\overline K+K)\sum_{j=J+1}^{\infty}\frac{1}{(j-2+a^0)(j-1+a^0)}
   +  3B\sum_{j=J+1}^\infty\frac{ \rho^{-j}}{\Gamma(j+a^0)},
 \end{aligned}
 \end{equation}
 for all $0<\epsilon\ll 1$, $\epsilon^{-1}\notin \mathbb N$.
 Now, for any $\upsilon>0$, we take $J\gg 1$ (independent of $\epsilon>0$) so that each of the last two convergent series on the right hand side of \eqref{eq:SNepsS0est} are less than $\upsilon/3$. Subsequently, we then take $\epsilon>0$ small enough so that the first term on the right hand side of \eqref{eq:SNepsS0est} (using \eqref{eq:SjepsSj0}) is less than $\upsilon/3$. In total, we have
\begin{align*}
\left|\sum_{j=2}^{N^\epsilon}  \frac{(-1)^j \epsilon \rsp{ \overline{\mathcal G^\epsilon}_j} }{ \overline w^\epsilon_j(1-\epsilon j)}-\sum_{j=2}^{\infty} \frac{(-1)^j\rsp{\mathcal G^0_j}}{w_j^0} \right|\le \upsilon,
\end{align*}
and the result follows.
 \end{proof}
% \begin{lemma}
%  Consider $\epsilon^{k-2}\overline g_k$ in \eqref{eq:bargk}. Then for $k=N^\epsilon+1$:
%  \begin{align*}
%   \vert \epsilon^{N^\epsilon-2}\overline g_{
%  \end{align*}
% 
% \end{lemma}

%  Henceforth, we undo the scaling by $\delta$. Then the estimate of $\overline m_k^\epsilon$ in \eqref{eq:mkepsest} becomes
%  \begin{align}\label{eq:mkepsest2}
%   \vert \overline m_k^\epsilon \vert \le C_2 (\epsilon^{-2}\overline w_k^\epsilon)\le C_2 e^{Q^\epsilon (k-2)+P^\epsilon},
%  \end{align}
% using \eqref{wkepsmajorant} in the final inequality. 

\subsection{Estimating the finite sum}
% \fbox{not updated yet}
Let $j^n[H]$ denote the $n$th-order Taylor jet/partial sum of $H(\overline x)=\sum_{k=2}^\infty H_k \overline x^k$:
\begin{align}
 j^n[H]:=\sum_{k=2}^n H_k(\cdot)^k\quad \forall\,n\in \mathbb N.\label{eq:jn}
\end{align}
Moreover, we define the $n$th-order remainder by
\begin{align}
r^{n}[H] = (I-j^{n})[H]: = \sum_{k=n+1}^\infty H_k (\cdot)^k \quad \forall\,n\in \mathbb N.\label{eq:rn}
\end{align}
% In the following, we will now use $\Vert\cdot\Vert_\delta$ to denote the sup-norm on $\overline x\in [-\delta,\delta]$:
% \begin{align*}
%  \Vert \overline h\Vert_\delta: = \sup_{\overline x\in [-\delta,\delta]} \vert \overline h(\overline x)\vert.
% \end{align*}

\begin{lemma}\label{lemmaJN_1}
 Consider the partial sum
 $$j^{N^\epsilon-1}[\overline m^\epsilon](\overline x)=\sum_{k=2}^{N^\epsilon-1} \overline m_k^\epsilon \overline x^k,$$ of the series $\overline m^\epsilon(\overline x)=\sum_{k=2}^\infty\overline m_k^\epsilon \overline x^k$.
 Then there is a constant $C>0$ such that 
 \begin{align}
  \vert j^{N^\epsilon-1}[\overline m^\epsilon](\overline x)\vert \le C \epsilon \quad \forall\,\overline x\in \left[-\frac34,\frac34\right], \label{estN_1}
 \end{align}
for all $0<\epsilon\ll 1$. 
\end{lemma}
\begin{proof}
  The estimate \eqref{estN_1}  follows from item \ref{C3} of Lemma \ref{lemma:wkeps} with $\delta=\frac34$. 
\end{proof}

\begin{lemma}\label{lemma:tildeg}
 For any $\overline D>0$, we \rspp{consider}
 \begin{align}\label{tildegdef}
 \overline g^\epsilon(\overline x,j^{N^\epsilon-1}[\overline m^\epsilon](\overline x)+q)\quad \forall\, \overline x\in \left[-\frac34,\frac34\right],\, q\in (-\overline D,\overline D).
 \end{align} 
  It is well-defined for all $0<\epsilon\ll 1$ and has the following absolutely convergent power series expansion
%  Then
%  we have
\begin{align}\label{tildeg}
\rspp{\overline g^\epsilon}(\overline x,j^{N^\epsilon-1}[\overline m^\epsilon](\overline x)+q) &=  \rspp{\overline g_0^\epsilon}(\overline x) +\overline x^2\rspp{\overline g_1^\epsilon}(\overline x) q+\overline x \sum_{l=2}^\infty \rspp{\overline g_{l}^\epsilon}(\overline x) q^l,
% \epsilon \sum_{k\ge 0} (\widetilde g_0(\epsilon))_{k} \overline x^k,\\
% \widetilde g_1(\overline x,r,\epsilon) &= \epsilon \sum_{k\ge 0,l\ge 0}  (\widetilde g_1(\epsilon))_{k,l} \overline x^k r^l,
\end{align}
% satisfying
with
\begin{align*}
\rspp{\overline g_0^\epsilon}(\overline x)&= \sum_{k=2}^\infty \rsp{\overline{\mathcal G}^\epsilon[j^{N^\epsilon-1}[\overline m^\epsilon]]_k}  \overline x^k.
% \widetilde g_0^\epsilon(\overline x)&= \sum_{k=2}^\infty \widetilde g_{k,0}^\epsilon  \overline x^k,\quad  \widetilde g_{k,0}^\epsilon=\overline{\mathcal G}^\epsilon[j^{N^\epsilon-1}[\overline m^\epsilon]]_k.
\end{align*}
Moreover, we have the following estimates ($Q_4^\epsilon$ is defined in \eqref{AepsBeps}):
\begin{align}\label{eq:gkest}
\vert \rsp{\overline{\mathcal G}^\epsilon[j^{N^\epsilon-1}[\overline m^\epsilon]]_k} \vert& \le C (\overline w_k^\epsilon)^2 e^{Q_4^\epsilon (k-4)} \quad \forall\,k\ge N^\epsilon+1;
\end{align}
specifically, for $k=N^\epsilon+1$:
\begin{align}\label{eq:gN1}
 \vert \rsp{\overline{\mathcal G}^\epsilon[j^{N^\epsilon-1}[\overline m^\epsilon]]_{N^\epsilon+1}} \vert\le C \overline w_2^\epsilon\overline w_{N^\epsilon-1}^\epsilon,
\end{align}
and
% where $\widetilde g_l$, $l\ge 1$, satisfies
% \begin{align*}
%  \vert \widetilde g_{k,0}^\epsilon \vert \le A \epsilon   (2\delta^0)^{-k},
% \end{align*}
% for any $k\ge 2$ and
\begin{align}
 \vert \rsp{\overline g_{l}^\epsilon}(\overline x)\vert \le \mu C \overline D^{-l+1} \quad \forall\,l\ge 1,\label{eq:radius}
\end{align}
for all $0<\epsilon\ll 1$, $\overline x\in [-\frac34,\frac34]$.  Here $C>0$ is
some constant that is independent of $\overline D$ and $\epsilon$.
% In particular,
% \begin{align}
%  \vert \widetilde g_{N^\epsilon+1,0}(\epsilon)\vert \le C\epsilon (\epsilon^{-2} \overline w_{N^\epsilon-1}^\epsilon),\label{eq:gN1b}
% \end{align}
% for $C>0$ large enough and all $0<\epsilon\ll 1$. 
% sufficiently small. 
% where 
\end{lemma}
\begin{proof}
 The expansion of \rspp{\eqref{tildegdef}} follows from composition of analytic functions.
  For the property of the convergence radius in \eqref{eq:radius}, we use the binomial theorem to obtain
  \begin{align}\label{eq:radiusproof}
 \rspp{\overline g^\epsilon_l}(\overline x) = \mu \sum_{n=l}^\infty  \left(\sum_{m=1}^\infty h^\epsilon_{m,n} \epsilon^{m-1} \overline  x^m \right)\epsilon^{n-1}  \begin{pmatrix}
                                                                                                                                                   n\\
                                                                                                                                                   l
                                                                                                                                                  \end{pmatrix} \left(j^{N^\epsilon-1}[\overline m^\epsilon](\overline x)\right)^{n-l},\quad l\ge 2,
  \end{align}
    cf. \eqref{overlineg} and \eqref{eq:overlinefheps}, and use \eqref{estN_1}, \eqref{eq:binomest} and \eqref{eq:gkcond0}. This gives
    \begin{align*}
     \vert \rspp{\overline g^\epsilon_l}(\overline x)\vert \le \frac32\mu\rho^{-1} \sum_{n=l}^\infty \epsilon^{n-1} \rho^{-n} 2^n (C\epsilon)^{n-l} &\le 3\mu\rho^{-1} (2\rho^{-1})^l \epsilon^{l-1}\le 6\mu \rho^{-2} \overline D^{-l+1}, \end{align*}
     for all $\overline x\in [-\frac34,\frac34]$, $\overline D <(2\rho^{-1}\epsilon)^{-1}$ and $0<\epsilon\ll 1$, upon estimating the geometric sums. 
% and all $0<\epsilon\ll 1$. %The claim follows since $\overline D^{-1} \ge 2\rho^{-1} \epsilon$.
 
 Next, we notice that \eqref{eq:gN1} follows from \eqref{eq:gkest} upon using \eqref{eq:AepsBepsProp}. 
 We therefore turn to proving \eqref{eq:gkest}. For this purpose, we use \eqref{gk0eps} and focus on estimating $$\rsp{\overline{\mathcal H}^\epsilon [ j^{N^\epsilon-1}[\overline m^\epsilon]]_k}.$$ By \eqref{gk12eps}, \eqref{eq:gkcond0}, $\Vert \overline m^\epsilon\Vert \le 2F$ in the seminorm \eqref{eq:seminorm} (cf. Proposition \ref{prop:seminormmeps}) and Lemma \ref{lemma:yklNew}, we obtain that 
%  \KUK{CHECK THIS}
 \begin{align*}
  \vert \rsp{\overline{\mathcal H}^\epsilon [ j^{N^\epsilon-1}[\overline m^\epsilon]]_k}\vert &\le 2 F\sum_{j=2}^{\min(k-2,N^\epsilon-1)}\rho^{-k+j-1} \epsilon^{k-j-1} \overline w_j^\epsilon+\\
  &+ \sum_{l=2}^{\lfloor \frac{k}{2}\rfloor} \sum_{j=2l}^{\min(k,l(N^\epsilon-1))} \rho^{-k+j-l}\epsilon^{k-j+l-2} (2F)^l C^{l-1} e^{(-2Q_4^\epsilon +P_4^\epsilon)l}e^{Q_4^\epsilon j}\\
  &\le 2F \rho^{-3} \epsilon  (\rho^{-1} \epsilon)^{k-(N^\epsilon+1)} \sum_{j=2}^{N^\epsilon-1} (\rho^{-1} \epsilon)^{N^\epsilon-1-j}\overline w_j^\epsilon\\
  &+ e^{Q_4^\epsilon k} \epsilon^{-2} C^{-1} \sum_{l=2}^{\lfloor \frac{k}{2}\rfloor} \left(2\rho^{-1}\epsilon F Ce^{-2Q_4^\epsilon+P_4^\epsilon}\right)^l\sum_{j=2l}^{\min(k,l(N^\epsilon-1))}\left(\rho^{-1}\epsilon e^{-Q_4^\epsilon}\right)^{k-j}.
 \end{align*}
Here 
\begin{align*}
 0<\rho^{-1}\epsilon e^{-Q_4^\epsilon}\ll 1,\quad 0<2\rho^{-1}\epsilon F Ce^{-2Q_4^\epsilon+P_4^\epsilon}\ll 1,
\end{align*}
for all $0<\epsilon\ll 1$, recall \eqref{AepsBepsApp}. But then, by estimating the geometric series and using $\exp(P_4^\epsilon) =  \overline w_2^\epsilon$ (see \eqref{eq:AepsBepsProp}), we conclude that
\begin{align*}
  \vert\rsp{\overline{\mathcal H}^\epsilon [ j^{N^\epsilon-1}[\overline m^\epsilon]]_k}\vert \le \overline C e^{Q_4^\epsilon k} e^{-4Q_4^\epsilon+2P_4^\epsilon} = \overline C (\overline w_2^\epsilon)^2 e^{Q_4^\epsilon(k-4)},
\end{align*}
 for some $\overline C>0$ large enough. This gives the desired estimates (upon $\overline C\rightarrow C$).
%  
%  In particular, for $k=N^\epsilon+1$ we have 
%  \begin{align*}
%   \vert\overline{\mathcal H}^\epsilon [ j^{N^\epsilon-1}[\overline m^\epsilon]]_{N^\epsilon+1} \le \overline C\overline w_2^\epsilon e^{P_4^\epsilon} e^{Q_4^\epsilon(N^\epsilon-3)}=\overline C\overline w_2^\epsilon \overline w_{N^\epsilon-1}^\epsilon,
%  \end{align*}
% using $e^{Q_4^\epsilon(N^\epsilon-3)+P_4^\epsilon}= \overline w_{N^\epsilon-1}^\epsilon$ in the last equality, see \eqref{eq:AepsBepsProp}.
\end{proof}

 We now turn to estimating $j^{N^\epsilon}[\overline m^\epsilon]$; in contrast to $j^{N^\epsilon-1}[\overline m^\epsilon]$, it is not uniformly bounded with respect to $\alpha^\epsilon\in (0,1)$.
    \begin{lemma}
    Suppose that $S_\infty^0\ne 0$. Then 
    \begin{align}
     \vert \overline m_{N^\epsilon}^\epsilon \overline x^{N^\epsilon} \vert\le  (1+o(1))\vert S_\infty^0 \vert \overline w_{N^\epsilon}^\epsilon\delta^{N^\epsilon}\quad \forall \,\overline x\in [-\delta,\delta],\label{mNeps}
    \end{align}
    for all $0<\epsilon\ll 1$.
    Moreover,  fix any $K>0$ and suppose for $N^\epsilon\gg 1$ and $\alpha^\epsilon\in (0,1)$ that 
    \begin{align}\label{delta1}
       \delta \le   \min\left(\frac34,\left(\frac{K}{2\vert  S_{\infty}^0\vert \vert \Gamma(\alpha^\epsilon) (N^\epsilon)^{a^\epsilon+1-\alpha^\epsilon}}\right)^{\frac{1}{N^\epsilon}}\right).
              %\left(\frac{1}{\Gamma(\alpha^\epsilon)}\right)^{\frac{1}{N^\epsilon}},
  \end{align}
   (For fixed $\alpha^\epsilon$, the expression on the right hand side of \eqref{delta1} converges to $\frac34$ for $N^\epsilon\rightarrow \infty$).
    Then 
  \begin{align}
   \vert j^{N^\epsilon}[ \overline m^\epsilon](\overline x)\vert\le K\quad \forall\, \overline x\in [-\delta,\delta].\label{eq:JNbound}
  \end{align}
%   for all $\overline x\in [-\delta,\delta]$.
%   In particular, 
%   then $\vert j^{N^\epsilon}(\overline y)(\overline x)\vert\le 2$. 
 \end{lemma}
 \begin{proof}
%   For any $-\delta\le \overline x\le \delta$, we have
We estimate
  \begin{equation}\label{eq:mkNeps}
  \begin{aligned}
  \vert  \overline m_{N^\epsilon}^\epsilon \overline x^{N^\epsilon}\vert &\le \vert \overline S_{N^\epsilon}^\epsilon\vert   \overline w_{N^\epsilon}^\epsilon \delta^{N^\epsilon}\\
  &=(1+o(1)) \vert  S_{\infty}^0\vert \frac{\Gamma(\alpha^\epsilon)\Gamma( N^\epsilon+a^\epsilon)}{\Gamma(\epsilon^{-1}-1)}\delta^{N^\epsilon}\\
  &=(1+o(1))\vert  S_{\infty}^0\vert\Gamma(\alpha^\epsilon) (N^\epsilon)^{a^\epsilon+1-\alpha^\epsilon} \delta^{N^\epsilon},
  \end{aligned}
  \end{equation}
  using \eqref{barmkeps}, \eqref{eq:Gamma}, $S_\infty^0\ne 0$ and
   Stirling's approximation (in the form \eqref{stirling}) on the factor 
  \begin{align}\label{this2}
  \frac{\Gamma( N^\epsilon+a^\epsilon)}{\epsilon\Gamma(\epsilon^{-1})} = \frac{\Gamma( N^\epsilon+a^\epsilon)}{(1-\epsilon)\Gamma(\epsilon^{-1}-1)}= (1+o(1))\frac{\Gamma(N^\epsilon) (N^{\epsilon})^{a^\epsilon}}{\Gamma(N^\epsilon) (N^\epsilon)^{\alpha^\epsilon-1}}=(1+o(1)) (N^{\epsilon})^{a^\epsilon+1-\alpha^\epsilon},
  \end{align}
  for $N^\epsilon\rightarrow \infty$; in particular  the $o(1)$-terms in \eqref{eq:mkNeps} are uniform with respect to $\alpha^\epsilon$. Using \eqref{delta1},
  we have
  \begin{align*}
   (1+o(1))\vert  S_{\infty}^0\vert\Gamma(\alpha^\epsilon) (N^\epsilon)^{a^\epsilon+1-\alpha^\epsilon} \delta^{N^\epsilon} \le \frac12 K(1+o(1))
  \end{align*}
  The result then follows from $
   j^{N^\epsilon}[\overline m^\epsilon](\overline x)= j^{N^\epsilon-1} [\overline m^\epsilon](\overline x)+\overline m_{N^\epsilon}^\epsilon \overline x^{N^\epsilon}$. 
%  Notice that the right hand side is independent of $\alpha^\epsilon\in (0,1)$. 
 \end{proof}

 If we take $\overline D>C>0$ and $0<\epsilon\ll 1$, then it follows from Lemma \ref{lemmaJN_1} \rspp{(upon setting $q=\overline m_{N^\epsilon}^\epsilon \overline x^{N^\epsilon}$)} that 
$$\rspp{\overline g^\epsilon(x,j^{N^\epsilon}[\overline m^\epsilon](\overline x))=\overline g^\epsilon(x,j^{N^\epsilon-1}[\overline m^\epsilon](\overline x)+\overline m_{N^\epsilon}^\epsilon \overline x^{N^\epsilon})},$$ is well-defined for all $\overline x\in [\delta,\delta]$ with $\delta>0$ satisfying \eqref{delta1}.

\subsection{The operator $\mathcal T^\epsilon$}\label{sec:Top}
Define $H\mapsto \mathcal T^\epsilon [H]$ by 
\begin{equation}\label{eq:Topdefn}
\begin{aligned}
 \mathcal T^\epsilon [H]({\overline x}) &:= \frac{{\overline x}^{\epsilon^{-1}}}{(1-{\overline x})^{\epsilon^{-1}+ a^\epsilon}}\int_0^{\overline x} \frac{(1-v)^{\epsilon^{-1}+a^\epsilon-1}}{v^{\epsilon^{-1}+1}} H(v) dv\\
 &:=\frac{\vert \overline x\vert^{\alpha^\epsilon} \overline x^{N^\epsilon}}{(1-{\overline x})^{\epsilon^{-1}+ a^\epsilon}}\int_0^{\overline x} \frac{(1-v)^{\epsilon^{-1}+a^\epsilon-1}}{\vert v\vert^{\alpha^\epsilon} v^{N^\epsilon+1}} H(v) dv\quad \forall -1<\overline x<1.
 \end{aligned}
 \end{equation}
It is well-defined on analytic functions $H$ with $j^{N^\epsilon}[H] =0$,  see also \cite[Section 7]{MR4445442}. 
% In particular, we have the following:
\begin{lemma}\label{lemma:TOp0}
Suppose that $\epsilon^{-1}\notin \mathbb N$.
Then the following statements hold true:
\begin{enumerate}
\item \label{Tdiff} For any analytic $H$ with $j^{N^\epsilon}[H]=0$, $G=\mathcal T^\epsilon [H]$ is the unique solution of
\begin{align}\label{Theqn}
\epsilon {\overline x}(1-{\overline x}) \frac{dG}{d {\overline x}}-(1+\epsilon a^\epsilon {\overline x})G = \epsilon H \quad \mbox{and}\quad j^{N^\epsilon}[G] = 0.
\end{align}
\item \label{TuN1} $\mathcal T^\epsilon\left[(\cdot)^{N^\epsilon+1}\right]$ has an absolutely convergent power series representation for $-1<{\overline x}<1$:
\begin{align}
%  \mathcal T^\epsilon\left[(\cdot)^{N^\epsilon+1}\right](u)&= \frac{u^{\epsilon^{-1}}}{(1-u)^{\epsilon^{-1}+ a^\epsilon}}\int_0^u {(1-v)^{\epsilon^{-1}+a^\epsilon-1}} v^{-\alpha^\epsilon} dv\\
%  &=\frac{1}{(i-\alpha^\epsilon) (-1)^{N^\epsilon+1} \overline w_{N^\epsilon+i}^\epsilon } \sum_{k=N^\epsilon+i}^\infty (-1)^k \overline w_k^\epsilon u^k \\
  \mathcal T^\epsilon\left[(\cdot)^{N^\epsilon+1}\right]({\overline x})&=\frac{\Gamma(1-\alpha^\epsilon)}{\Gamma(N^\epsilon+1+a^\epsilon)} \sum_{k=N^{\epsilon}+1}^\infty \frac{\Gamma(k+a^\epsilon)}{\Gamma(k+1-\epsilon^{-1})}{\overline x}^k.\label{eq:TuN1}
\end{align}
% which 
% is absolutely convergent for all $-1<u<1$.
\end{enumerate}
\end{lemma}
% \end{lemma}
\begin{proof}
To prove item \ref{Tdiff}, we define $\mathcal J({\overline x}):= \frac{{\overline x}^{\epsilon^{-1}}}{(1-{\overline x})^{\epsilon^{-1}+a^\epsilon}}$ and subsequently $\mathcal I({\overline x}) := \int_0^{\overline x}  \frac{1}{\mathcal J(v) v(1-v)} H(v) dv$. Then $$\mathcal T^\epsilon [H]=\mathcal J \mathcal I\quad \mbox{and}\quad \mathcal J(\overline x) I'(\overline x) = \frac{1}{{\overline x}(1-{\overline x})} H({\overline x}).$$ Moreover,
\begin{align*}
 \mathcal J'({\overline x}) &= \mathcal J({\overline x}) \left(\epsilon^{-1} {\overline x}^{-1}+(\epsilon^{-1}+a^\epsilon)(1-{\overline x})^{-1}\right) \\
%  &=\mathcal J(u) \frac{1}{u(1-u)} \left(\epsilon^{-1}(1-u)+(\epsilon^{-1}+a^\epsilon)u\right)\\
 &=\mathcal J({\overline x}) \frac{1+\epsilon a^\epsilon {\overline x}}{\epsilon {\overline x}(1-{\overline x})},
\end{align*}
and therefore
\begin{align*}
 \epsilon {\overline x}(1-{\overline x}) \mathcal T^\epsilon [H]'({\overline x}) =  \left(1+\epsilon a^\epsilon {\overline x}\right)\mathcal J({\overline x})\mathcal I({\overline x})+ \epsilon H({\overline x}). 
\end{align*}
Consequently,
\begin{align*}
 \epsilon {\overline x}(1-{\overline x}) \mathcal T^\epsilon [H]'({\overline x}) -(1+\epsilon a^\epsilon {\overline x}) \mathcal T^\epsilon [H]({\overline x}) = \epsilon H({\overline x}),
 \end{align*}
as desired.

Next, to prove item \ref{TuN1},
we use item \ref{Tdiff} and the fact that the solution is unique. Then Lemma \ref{lemma:barSkeps} with 
 \begin{align*}
 \epsilon \rsp{\overline{\mathcal G}^\epsilon[\overline m^\epsilon]_k}=\begin{cases} -\epsilon & \text{for $k=N^\epsilon+1$},\\
                                   0 & \text{else},
                                  \end{cases}
                                  \end{align*}
                                  allow us to write $\mathcal T^\epsilon\left[(\cdot)^{N^\epsilon+1}\right]({\overline x})=\sum_{k=N^\epsilon+1}^\infty \overline m_k^\epsilon {\overline x}^k$ as an absolutely convergent power series; notice the change of sign when comparing \eqref{Theqn} and \eqref{ybarEqn}. In particular, we find that $$\overline S_k^\epsilon=\frac{(-1)^{N^\epsilon+1}}{\overline w_{N^\epsilon+1}^\epsilon (1-\alpha^\epsilon)} \quad \forall\,k\ge N^\epsilon+1 \,\mbox{(zero otherwise)},$$ and therefore
                                  \begin{align*}
                                   \overline m_{k}^\epsilon =  \frac{1}{1-\alpha^\epsilon}\frac{(-1)^k \overline w_k^\epsilon}{(-1)^{N^\epsilon+1} \overline w_{N^\epsilon+1}^\epsilon} \quad \forall\,k\ge N^\epsilon+1\,\mbox{(zero otherwise)},
                                  \end{align*}
                                  by \eqref{barmkeps}.
Subsequently, we then use item \ref{kinf} of Lemma \ref{lemma:wkeps} to write
                                  \begin{align*}
                                   \frac{(-1)^k \overline w_k^\epsilon}{(-1)^{N^\epsilon+1} \overline w_{N^\epsilon+1}^\epsilon} &= \frac{\Gamma(2-\alpha^\epsilon)}{\Gamma(N^\epsilon+1+a^\epsilon)}\frac{\Gamma(k+a^\epsilon)}{\Gamma(k+1-\epsilon^{-1})}\\
                                   &=\frac{(1-\alpha^\epsilon)\Gamma(1-\alpha^\epsilon)}{\Gamma(N^\epsilon+1+a^\epsilon)}\frac{\Gamma(k+a^\epsilon)}{\Gamma(k+1-\epsilon^{-1})}.
                                  \end{align*}
This gives the desired expression  for $\mathcal T^\epsilon \left[(\cdot)^{N+1}\right]$ in item \ref{TuN1}.

\end{proof}

% We shall see that this operator is well-defined on analytic functions $h$ with $j^{N^\epsilon} h =0$, 
We will view $\mathcal T^\epsilon$ on the Banach space $\mathcal D^\epsilon_\delta$ of analytic functions $H:[0,\delta]\rightarrow \mathbb R$ with $\vert H({\overline x}){\overline x}^{-N^\epsilon-1} \vert$ bounded at ${\overline x}=0$. 
More specifically, we define
\begin{align*}
\mathcal D^\epsilon_\delta:= \{H:[0,\delta] \rightarrow \mathbb R\,\,\text{analytic}\,:\,\Wert H\Wert_{\delta} <\infty\},
\end{align*}
with the Banach norm
\begin{align}
 \Wert H\Wert_{\delta}:=\sup_{{\overline x}\in (0,\delta]}\vert H({\overline x})\vert \frac{\mathcal T\left[(\cdot)^{N^\epsilon+1}\right](\delta)}{\mathcal T\left[(\cdot)^{N^\epsilon+1}\right]({\overline x})}; \label{Cepsnorm}
\end{align}
here we have used that $\mathcal T\left[(\cdot)^{N^\epsilon+1}\right]({\overline x})=\mathcal O({\overline x}^{N^\epsilon+1})$ as ${\overline x}\rightarrow 0$ and that $\mathcal T\left[(\cdot)^{N^\epsilon+1}\right]({\overline x})>0$ for all ${\overline x}\in (0,1)$, cf.  Lemma \ref{lemma:TOp0} item \ref{TuN1}.

The case ${\overline x}<0$ has to be treated slightly different (we will have to take $0\le -{\overline x}\le \delta_2\epsilon$); we will consider this case at the end of Section \ref{sec:completing} below.

A nice property of the Banach norm \eqref{Cepsnorm} is highlighted in the following Lemma.
\begin{lemma}
Define
\begin{align}
\Vert H\Vert_\delta:=\sup_{{\overline x}\in [0,\delta]} \vert H({\overline x})\vert.\label{eq:Vertdefn}
\end{align}
Then the following estimate holds: 
\begin{align*}
 \Vert H\Vert_\delta \le \Wert H\Wert_{\delta}\quad \forall\,H\in \mathcal D_\delta^\epsilon.
\end{align*}
\end{lemma}
\begin{proof}
The proof is elementary. Indeed, for any ${\overline x}\in (0,\delta]$ we have 
\begin{equation}\label{eq:huest}
\begin{aligned}
\vert H({\overline x})\vert &\le  \left(\left| H({\overline x})\right|\frac{\mathcal T\left[(\cdot)^{N^\epsilon+1}\right](\delta)}{\mathcal T\left[(\cdot)^{N^\epsilon+1}\right]({\overline x})}\right) \times \frac{\mathcal T\left[(\cdot)^{N^\epsilon+1}\right]({\overline x})}{\mathcal T\left[(\cdot)^{N^\epsilon+1}\right](\delta)}\le \Wert H\Wert_{\delta}\times \frac{\mathcal T\left[(\cdot)^{N^\epsilon+1}\right]({\overline x})}{\mathcal T\left[(\cdot)^{N^\epsilon+1}\right](\delta)},
\end{aligned}
\end{equation}
and therefore 
\begin{align}
\vert H({\overline x})\vert&\le \Wert H\Wert_{\delta} \times 1,\label{eq:huest2}
\end{align}
since $\mathcal T\left[(\cdot)^{N^\epsilon+1}\right]({\overline x})$ is an increasing function of ${\overline x}\in [0,1)$.
As $H(0)=0$ the inequality \eqref{eq:huest2} holds for all ${\overline x}\in [0,\delta]$, completing the proof.
\end{proof}
Notice also the (obvious) fact that
\begin{align*}
 \Wert H \Wert_{\delta'}\le \Wert H \Wert_{\delta},
\end{align*}
for any $0<\delta'<\delta$. This also holds with $\Wert$ replaced by  $\Vert$, recall \eqref{eq:Vertdefn}. We will use these properties without further mention in the following. 
% for all $u\in [-\delta,\delta]$.
The following set of equalities
\begin{align}
 \Vert \mathcal T^\epsilon\left[(\cdot)^{N^\epsilon+1}\right] \Vert_{\delta} = \Wert \mathcal T^\epsilon\left[(\cdot)^{N^\epsilon+1}\right]\Wert_{\delta}  = \mathcal T^\epsilon\left[(\cdot)^{N^\epsilon+1}\right](\delta)\quad \forall\,0<\delta<1,\label{VertWertT}
\end{align}
are consequences of $\mathcal T^\epsilon\left[(\cdot)^{N^\epsilon+1}\right]({\overline x})$ being an increasing function of ${\overline x}\in [0,1)$, and they will also be important.
% We let $\Vert g(y

It turns out that 
\begin{align}\label{eq:deltacond0}
 0<\delta\le \frac34,
\end{align}
will be adequate for our purposes.
% \fbox{This operator should be known. I will check.}
% \begin{align*}
%  (1-v)^{\epsilon^{-1}
% \end{align*}

\rsp{In the following (see e.g. item \ref{composition}),  we will use $(\cdot,\overline Y)$ to denote the composition $\overline x\mapsto (\overline x,\overline Y(\overline x))$ for given analytic functions $\overline Y:\overline x\mapsto \overline Y(\overline x)$.}

\begin{lemma} \label{lemma:TOp}
% \fbox{mistake here.} 
% 
% \fbox{can be fixed in different ways}
% 
% \fbox{simplest fix: let $\delta\le c\epsilon$}
% 
% \fbox{better express $\Wert \mathcal T^\epsilon [H]\Wert \le \delta^{-N^\epsilon-1}\Wert \mathcal T^\epsilon(()^{N^\epsilon+1})\Wert \Wert H\Wert$}
% \fbox{But result is still true (up to multiplication by const) if $\delta\le c\epsilon$}
Fix any $\delta_2>0$, suppose that $\epsilon^{-1}\notin \mathbb N$ and that \eqref{eq:deltacond0} holds. Then there exists a $K_2=K_2(\delta_2,a^0)$ such that the following holds true.
% the following holds for some constant $K_2>0$.
\begin{enumerate}
% \item \label{Tdiff}  For any $h\in \mathcal D^\epsilon_\delta$, $F=\mathcal T^\epsilon [H]$ is the unique solution in $\mathcal D^\epsilon_\delta$ of 
% \begin{align}\label{Theqn}
% \epsilon u(1-u) \frac{dF}{d u}-(1+\epsilon a^\epsilon u)F = \epsilon h.\end{align}
\item \label{phiest}
 Define $\sigma_\epsilon:[-\delta_2\epsilon,\frac34]\rightarrow \mathbb R_+$ by 
 \begin{align}\label{eq:phidefn}
  \sigma_\epsilon({\overline x}): = \begin{cases}
                                 1 & \quad \forall\, 0 \le \vert {\overline x}\vert\le\delta_2 \epsilon\\
                                 ({\overline x}^{-1}\delta_2\epsilon)^{1-\alpha^\epsilon} & \quad \forall\, \delta_2 \epsilon<{\overline x}\le \frac34,
                                \end{cases}
 \end{align}
 so that 
 \begin{align}\label{eq:phibound}
 \left(\frac43 \delta_2 \epsilon\right)^{1-\alpha^\epsilon} \le \sigma_\epsilon({\overline x})\le 1\quad \forall\, {\overline x}\in \left[-\delta_2\epsilon,\frac34\right].
 \end{align}
Then there are constants $0<C_1<C_2$, $C_i=C_i(\delta_2,a^0)$, such that the following holds for all $0<\epsilon\ll 1$, $\epsilon^{-1}\notin \mathbb N$:
\begin{align}
 C_1 \left(\frac{\vert {\overline x}\vert}{1- {\overline x}}\right)^{N^\epsilon+1} \sigma_\epsilon({\overline x}) \le (1-\alpha^\epsilon)\left|\mathcal T^\epsilon\left[(\cdot)^{N^\epsilon+1}\right]({\overline x})\right|\le C_2 \left(\frac{\vert {\overline x}\vert }{1-{\overline x}}\right)^{N^\epsilon+1} \sigma_\epsilon({\overline x}),\label{eq:phiest}
\end{align}
for all $ -\delta_2\epsilon\le {\overline x}\le \frac34$.
% (u^{-1}\delta_2\epsilon)^{1-\alpha^\epsilon}
% \end{lemma}
\item \label{Tu2} Asymptotics for $\overline x=\mathcal O(\epsilon)$: For any $\overline x =\epsilon\overline x_2$, $\overline x_2\in [-\delta_2,\delta_2]$, the following asymptotics hold true
\begin{align*}
 \mathcal T^\epsilon\left[(\cdot)^{N^\epsilon+1}\right](\epsilon {\overline x}_2) = \frac{1}{1-\alpha^\epsilon} (\epsilon \overline x_2)^{N^\epsilon+1} \left[1+\overline x_2\int_0^1 e^{(1-v)\overline x_2}  v^{1-\alpha^\epsilon} dv+\mathcal O(\epsilon)\right]
 \quad \forall\,\overline x_2\in [-\delta_2,\delta_2], 
\end{align*}
with $\mathcal O(\epsilon)$ being uniform with respect to $\alpha^\epsilon \in (0,1)$. 
\item \label{Tbounded} 
 $\mathcal T^\epsilon:\mathcal D^\epsilon_\delta\rightarrow \mathcal D^\epsilon_\delta$ is a bounded operator. In particular, let 
 \begin{align*}
  \Wert \mathcal T^\epsilon\Wert_\delta: = \sup_{\Wert H\Wert_\delta=1} \Wert \mathcal T^\epsilon [H]\Wert_\delta,
 \end{align*}
denote the operator norm. Then 
\begin{align*}
\Wert \mathcal T^\epsilon\Wert_\delta \le \frac{K_2}{1-\alpha^\epsilon}(1+ \log \sigma_\epsilon(\delta)^{-1}).
\end{align*}
% In particular, the operator norm is uniformly bounded with respect to $0<\epsilon\ll 1$, $\epsilon^{-1}\notin \mathbb N$ if $0<\delta\le \delta_2\epsilon$ (cf. \eqref{eq:phidefn}).
\item \label{boundT2i}
% \item \label{boundT2i}
The following holds for any $i\in \mathbb N$:
\begin{align}
  \Wert \mathcal T^\epsilon\left[(\cdot)^i H)\right]\Wert_\delta &\le \frac{K_2\delta^i}{i}\Wert H\Wert_\delta \quad \forall\,\,H\in \mathcal D^\epsilon_\delta.\label{Tuih}
 \end{align}
 \item \label{boundT2i2}
The following holds for any $l\in \mathbb N$:
\begin{align}
  \Wert \mathcal T^\epsilon\left[(\cdot)^{lN^\epsilon+1} \right]\Wert_\delta &\le \delta^{(l-1)N^\epsilon} \mathcal T^\epsilon\left[(\cdot)^{N^\epsilon+1}\right](\delta).\label{Tuih2}
 \end{align}
\item \label{Thanal} Suppose that $E,R>0$, $0<\delta <R$ and consider
\begin{align*}
 H({\overline x}) = \sum_{k=N^\epsilon+1}^\infty H_k {\overline x}^k\quad \forall\, {\overline x}\in [0,\delta],
\end{align*}
% Then for $\delta>0$ small enough
% \begin{align*}
% \Wert \mathcal T^\epsilon(h)\Wert \le \frac{4}{1-\alpha^\epsilon} \sum_{k=N^\epsilon+1}^\infty \vert h_k\vert x^k\le\frac{4}{1-\alpha^\epsilon} \sum_{k=0}^\infty \vert h_k\vert x^k ,
% \end{align*}
with $\vert H_k\vert \le E R^k$. Then
\begin{align*}
 \Wert \mathcal T^\epsilon[H]\Wert_\delta &\le%\frac{K_2}{1-\alpha^\epsilon}\frac{(K\delta)^{N^\epsilon+1}}{1-K\delta}\le 
 \frac{E R^{N^\epsilon+1}}{1-R\delta} \mathcal T\left[(\cdot)^{N^\epsilon+1}\right](\delta),
\end{align*}
for all $0<\epsilon\ll 1$.
\item \label{composition} Let $C:=C_2C_1^{-1}>1$, with $C_i>0$ defined in \eqref{eq:phiest}, $E,R>0$, and suppose that $\overline Y\in \mathcal D^\epsilon_\delta$ and
\begin{align}
H({\overline x},\overline Y)= \sum_{l=2}^\infty H_{l}({\overline x}) \overline Y^l\quad \forall\, {\overline x}\in [0,\delta],\,0\le \Wert \overline Y\Wert_\delta <R^{-1},
\end{align}
with $$\Vert H_l\Vert_\delta \le  E R^{l-1}\quad \forall\, l\ge 2,$$ recall \eqref{eq:Vertdefn}.  Then 
\begin{align*}
 \Wert \mathcal T^\epsilon\left[H(\cdot,\overline Y)\right]\Wert_\delta \le  4\epsilon E K_2 C R \Wert \overline Y\Wert_\delta^2\quad \forall\,  0\le \Wert\overline Y\Wert_\delta<\frac12 (CR)^{-1}, %2\delta^2 C\overline Y+\frac{2}{N^\epsilon}\frac{C^3\Wert \overline Y\Wert^2}{1-C\Wert \overline Y\Wert},
\end{align*}
uniformly in $\alpha^\epsilon\in (0,1)$.
In particular, $\overline Y\mapsto \mathcal T^\epsilon\left[h(\cdot,\overline Y)\right]$ is $C^1$ and for all $0<\epsilon\ll 1$, it is a contraction:
\begin{align}\label{DY}
 \Wert D_{\overline Y}(\mathcal T^\epsilon\left[H(\cdot,\overline Y)\right](\overline Z)\Wert_\delta\le \mathcal O(\epsilon) \Wert \overline Z\Wert_\delta\quad \forall\,\overline Z\in\mathcal D^\epsilon_\delta,\,\Wert \overline Y\Wert_\delta<\frac12 (CR)^{-1}.
\end{align}
% uniformly in $\alpha^\epsilon\in (0,1)$.

\end{enumerate}

\end{lemma}
\begin{proof}
We prove the items \ref{phiest}--\ref{composition} successively in the following.

\textit{Proof of item \ref{phiest}}. The result follows from \cite[Lemma 7.2]{MR4445442}, see \cite[Eq. (7.10)]{MR4445442}, and it is based on the integral representation for $\mathcal T^\epsilon\left[(\cdot)^{N^\epsilon+1}\right]$:
\begin{equation}\label{eq:TN1defn}
\begin{aligned}
 \mathcal T^\epsilon\left[(\cdot)^{N^\epsilon+1}\right]({\overline x}) &= \frac{{\overline x}^{\epsilon^{-1}}}{(1-{\overline x})^{\epsilon^{-1}+ a^\epsilon}}\int_0^{\overline x}  {(1-v)^{\epsilon^{-1}+ a^\epsilon-1}}{v^{-\alpha^\epsilon}} dv\\
 &=\frac{\vert \overline x\vert^{\alpha^\epsilon} {\overline x}^{N^\epsilon}}{(1-{\overline x})^{\epsilon^{-1}+ a^\epsilon}}\int_0^{\overline x}  {(1-v)^{\epsilon^{-1}+ a^\epsilon-1}}{\vert v\vert^{-\alpha^\epsilon}} dv
 \end{aligned}
\end{equation} For completeness, we include the details (which will also be important later): 
Firstly, for $\overline x=\epsilon\overline x_2\in [-\epsilon\delta_2,\epsilon\delta_2]$, $\delta_2>0$ fixed, we use:
% For $0\le \vert u\vert\le  \delta\le \delta_2 \epsilon$, we have
\begin{equation}\label{thisest}
\begin{aligned}
 (1-{\overline x})^{\epsilon^{-1}}& = e^{\epsilon^{-1} \log (1-{\overline x})}=e^{-\overline x_2}\left(1+\mathcal O(\epsilon\overline x_2^2)\right)=\mathcal O(1)\gtrless \begin{cases} e^{-2\delta_2}\\
 e^{2\delta_2}
 \end{cases}
\end{aligned}
\end{equation}
for all $0<\epsilon\ll 1$. 
% Therefore by the definition \eqref{eq:Topdefn} of $\mathcal T^\epsilon$ we conclude that
Consequently, for ${\overline x}\in [-\epsilon\delta_2,\epsilon \delta_2]$
\begin{align*}
\mathcal T^\epsilon\left[(\cdot)^{N^\epsilon+1}\right]({\overline x}) &=\left( \frac{{\overline x}}{1-{\overline x}}\right)^{N^\epsilon+1} (1+o(1)) {\vert\overline x\vert}^{\alpha^\epsilon} \overline x^{-1} \int_0^{\overline x} \mathcal O(1) \vert v\vert^{-\alpha^\epsilon }dv\\
&=\left( \frac{{\overline x}}{1-{\overline x}}\right)^{N^\epsilon+1} (1+o(1))\frac{\mathcal O(1)}{1-\alpha^\epsilon},
%  \vert \mathcal T^\epsilon\left[(\cdot)^{N^\epsilon+1}\right]({\overline x})\vert &\le  e^{4\delta_2} \vert {\overline x}\vert ^{\epsilon^{-1}}\int_0^{\vert {\overline x}\vert}  v^{-\alpha^\epsilon} dv\\
%  &\le \frac{e^{4\delta_2}}{1-\alpha^\epsilon} \vert {\overline x}\vert ^{N^\epsilon+1},
\end{align*}
where $0<C_1<\mathcal O(1)<C_2$. Next, for $\delta_2\epsilon<{\overline x}\le \frac34$, we use a separate set of estimates: 
\begin{align*}
 \mathcal T\left[(\cdot)^{N^\epsilon+1}\right]({\overline x})\vert &\le \frac{{\overline x}^{\epsilon^{-1}}}{(1-{\overline x})^{\epsilon^{-1}+a^\epsilon}}\int_0^1 (1-v)^{\epsilon^{-1}+a^\epsilon-1} v^{-\alpha^\epsilon} dv \\
  &= \frac{{\overline x}^{\epsilon^{-1}}}{(1-{\overline x})^{\epsilon^{-1}+a^\epsilon}} \frac{\Gamma(\epsilon^{-1}+a^\epsilon)\Gamma(1-\alpha^\epsilon) }{\Gamma(\epsilon^{-1}+a^\epsilon+1-\alpha^\epsilon)}\\
  &=\frac{{\overline x}^{\epsilon^{-1}}}{(1-{\overline x})^{\epsilon^{-1}+a^\epsilon}} {(1+o(1))\Gamma(1-\alpha^\epsilon)}\epsilon^{1-\alpha^\epsilon} \\
  &\le C_2\left(\frac{{\overline x}}{1-{\overline x}}\right)^{N^\epsilon+1} ({\overline x}^{-1} \delta_2 \epsilon)^{1-\alpha^\epsilon} (1-\alpha^\epsilon)^{-1},
%   &=K_2\left(\frac{{\overline x}}{1-{\overline x}}\right)^{N^\epsilon+1} {\overline x}^{\alpha^\epsilon-1} \frac{(1+o(1))\Gamma(1-\alpha^\epsilon)}{(N^\epsilon)^{1-\alpha^\epsilon} }\
\end{align*}
and 
\begin{align*}
  \mathcal T\left[(\cdot)^{N^\epsilon+1}\right]({\overline x})\vert  &\ge \frac{{\overline x}^{\epsilon^{-1}}}{(1-{\overline x})^{\epsilon^{-1}+a^\epsilon}}\int_0^{\delta_2 \epsilon } (1-v)^{\epsilon^{-1}+a^\epsilon-1} v^{-\alpha^\epsilon}dv\\
%   &\ge C_1\frac{{\overline x}^{\epsilon^{-1}}}{(1-{\overline x})^{\epsilon^{-1}+a^\epsilon}} \frac{(\delta_2 \epsilon)^{1-\alpha^\epsilon}}{1-\alpha^\epsilon}\\
  &\ge C_1 \left(\frac{{\overline x}}{1-{\overline x}}\right)^{N^\epsilon+1} \frac{({\overline x}^{-1}\delta_2 \epsilon)^{1-\alpha^\epsilon}}{1-\alpha^\epsilon}.
\end{align*}
for some $C_1=C_1(\delta_2,a^0)$ small enough, cf. \eqref{thisest}. Here we have also used \eqref{eq:integral},  \eqref{stirling} and \eqref{eq:Gamma0}.

% 
% Write $\mathcal J(u):= \frac{u^{\epsilon^{-1}}}{(1-u)^{\epsilon^{-1}+a^\epsilon}}$ and subsequently $\mathcal I(u) := \int_0^u  \frac{1}{\mathcal J(v) v(1-v)} h(v) dv$. Then $$\mathcal T^\epsilon [H]=\mathcal J \mathcal I\quad \mbox{and}\quad \mathcal J I' = \frac{1}{u(1-u)} h(u).$$ Moreover,
% \begin{align*}
%  \mathcal J'(u) &= \mathcal J(u) \left(\epsilon^{-1} u^{-1}+(\epsilon^{-1}+a^\epsilon)(1-u)^{-1}\right) \\
% %  &=\mathcal J(u) \frac{1}{u(1-u)} \left(\epsilon^{-1}(1-u)+(\epsilon^{-1}+a^\epsilon)u\right)\\
%  &=\mathcal J(u) \frac{1+\epsilon a^\epsilon u}{\epsilon u(1-u)},
% \end{align*}
% and therefore
% \begin{align*}
%  \epsilon u(1-u) \mathcal T^\epsilon [H]'(u) =  \left(1+\epsilon a^\epsilon u\right)\mathcal J(u)\mathcal I(u)+ \epsilon h(u). 
% \end{align*}
% Consequently,
% \begin{align*}
%  \epsilon u(1-u) \mathcal T^\epsilon [H]'(u) -(1+\epsilon a^\epsilon u) \mathcal T^\epsilon [H](u) = \epsilon h(u),
%  \end{align*}
% as desired.

 \textit{Proof of item \ref{Tu2}}. For item \ref{Tu2}, we use \eqref{eq:TN1defn} with the substitution $v=\overline x \widetilde v$ and \eqref{thisest} with $\overline x=\epsilon\overline x_2\in [-\epsilon\delta_2,\epsilon\delta_2]$. This gives
 \begin{align*}
 \mathcal T^\epsilon\left[(\cdot)^{N^\epsilon+1}\right]({\overline x})% = {\overline x}^{\epsilon^{-1}} e^{\overline x_2}(1+\mathcal O(\epsilon))\epsilon^{1-\alpha^\epsilon} \int_0^{\overline x_2} e^{-(1+\epsilon(a^\epsilon-1))v_2} (1+\mathcal O(\epsilon v_2^2)) v_2^{-\alpha^\epsilon} dv_2\\
 &={\overline x}^{N^\epsilon+1} e^{\overline x_2}(1+\mathcal O(\epsilon)) \int_0^1 e^{-(1+\epsilon(a^\epsilon-1)) \widetilde v \overline x_2} (1+\mathcal O(\epsilon \widetilde v^2)) \widetilde v^{-\alpha^\epsilon} d\widetilde v\\
 &={\overline x}^{N^\epsilon+1} e^{\overline x_2}(1+\mathcal O(\epsilon)) \left(\int_0^1 e^{-(1+\epsilon(a^\epsilon-1))v\overline x_2}  v^{-\alpha^\epsilon} dv+\mathcal O(\epsilon)\right),
 \end{align*}
 with each $\mathcal O(\epsilon)$ uniform with respect to $\alpha^\epsilon$. 
 We now use integration by parts on the remaining integral:
 \begin{align*}
 \int_0^1 e^{-(1+\epsilon(a^\epsilon-1))t\overline x_2}  v^{-\alpha^\epsilon} dv &=\frac{e^{-(1+\epsilon(a^\epsilon-1))\overline x_2}}{1-\alpha^\epsilon}\left[1+(1+\epsilon(a^\epsilon-1))\overline x_2 \int_0^1 e^{(1+\epsilon(a^\epsilon-1))(1-v)\overline x_2}  v^{1-\alpha^\epsilon} dv\right]\\
 &=\frac{e^{-\overline x_2}}{1-\alpha^\epsilon}\left[1+\overline x_2\int_0^1 e^{(1-v)\overline x_2}  v^{1-\alpha^\epsilon} dv+\mathcal O(\epsilon)\right].
\end{align*}
This completes the proof.

\textit{Proof of item \ref{Tbounded}}. We estimate using \eqref{eq:Topdefn}, \eqref{eq:huest} and \eqref{eq:phiest}
\begin{align*}
 \left|\mathcal T^\epsilon[H](\overline x)\frac{\mathcal T\left[(\cdot)^{N^\epsilon+1}\right](\delta)}{\mathcal T\left[(\cdot)^{N^\epsilon+1}\right]({\overline x})}\right|&\le 
\frac{{\overline x}^{\epsilon^{-1}}}{(1-{\overline x})^{\epsilon^{-1}+ a^\epsilon}} \frac{1}{\mathcal T\left[(\cdot)^{N^\epsilon+1}\right]({\overline x})} \int_0^{\overline x} \frac{(1-v)^{\epsilon^{-1}+a^\epsilon-1}}{v^{\epsilon^{-1}+1}} \mathcal T\left[(\cdot)^{N^\epsilon+1}\right](v) dv\Wert H\Wert_\delta \\
 &\le C_2 C_1^{-1} \frac{{\overline x}^{\epsilon^{-1}}}{(1-{\overline x})^{\epsilon^{-1}+ a^\epsilon}}  \left(\frac{1-{\overline x}}{{\overline x}}\right)^{N^\epsilon+1} \sigma_\epsilon({\overline x})^{-1} \\
 &\times \int_0^{\overline x} \frac{(1-v)^{\epsilon^{-1}+a^\epsilon-1}}{v^{\epsilon^{-1}+1}} \left(\frac{v}{1-v}\right)^{N^\epsilon+1} \sigma_\epsilon(v)  dv\Wert H\Wert_\delta\\
 &\le C_2 C_1^{-1} (1-{\overline x})^{1-\alpha^\epsilon-a^\epsilon} \frac{{\overline x}^{\alpha^\epsilon-1}}{ \sigma_\epsilon({\overline x})} \\
 &\times \int_0^{\overline x} (1-v)^{\alpha^\epsilon+a^\epsilon -2} v^{-\alpha^\epsilon} \sigma_\epsilon(v) dv \Wert H\Wert_\delta\\
&\le  K_2 \frac{{\overline x}^{\alpha^\epsilon-1}}{\sigma_\epsilon({\overline x})}\int_0^{\overline x} v^{-\alpha^\epsilon} \sigma_\epsilon(v) dv \Wert H\Wert_\delta\quad \forall\,0<{\overline x}\le \delta\le \frac34,
\end{align*}
for some $K_2=K_2(\delta_2,a^0)>0$. Here we have used uniform bounds on 
\begin{align*}
 (1-{\overline x})^{1-\alpha^\epsilon-a^\epsilon}\quad \mbox{and}\quad (1-{\overline x})^{\alpha^\epsilon+a^\epsilon-2}\quad \mbox{for}\quad {\overline x}\in \left[0,\frac34\right].
\end{align*}
Due to \eqref{eq:phidefn}, we estimate $0<{\overline x}\le \delta_2\epsilon$ and $\delta_2\epsilon<{\overline x}\le \delta$ separately. The former gives an estimate
\begin{align*}
 \left|\mathcal T^\epsilon[H](\overline x) \frac{\mathcal T\left[(\cdot)^{N^\epsilon+1}\right](\delta)}{\mathcal T\left[(\cdot)^{N^\epsilon+1}\right]({\overline x})}\right|\le \frac{K_2}{1-\alpha^\epsilon}\Wert H\Wert_\delta\quad \forall\,0<{\overline x}\le \delta_2\epsilon,
 \end{align*}
 directly. We therefore consider $\delta_2\epsilon<{\overline x}\le \delta\le \frac34$ and find
\begin{align*}
 \frac{{\overline x}^{\alpha^\epsilon-1}}{\sigma_\epsilon({\overline x})}\int_0^{\overline x} v^{-\alpha^\epsilon} \sigma_\epsilon(v) dv &=  \frac{1}{(\delta_2\epsilon)^{1-\alpha^\epsilon}}\int_0^{\delta_2\epsilon} v^{-\alpha^\epsilon} dv + \int_{\delta_2\epsilon}^{\overline x} v^{-1} dv\\
 &= \frac{1}{1-\alpha^\epsilon}-\log ({\overline x}^{-1} \delta_2 \epsilon)\\
 &=\frac{1}{1-\alpha^\epsilon}(1+\log \sigma_\epsilon({\overline x})^{-1}).
\end{align*}
This completes the proof.

\textit{Proof of item \ref{boundT2i}}.
Proceeding as in the proof of item \ref{Tbounded}, we find
% \begin{align*}
%  \Wert \mathcal T^\epsilon\left[(\cdot)^i h\right]\Wert_\delta \le \Vert (\cdot)^{-N^\epsilon-1} \mathcal T^\epsilon\left[(\cdot)^{N^\epsilon+1+i}\right]\Vert_\delta \Wert H\Wert_\delta.
% \end{align*}
% For $0<\vert u\vert \le \delta\le\delta_2\epsilon$, we can estimate $u^{-N^\epsilon-1}\mathcal T^\epsilon\left[(\cdot)^{N^\epsilon+1+i}\right](u)$ as in the proof of item \ref{boundT2}. We find 
\begin{equation}\label{eq:Tepsibound}
\begin{aligned}
  \left|\mathcal T^\epsilon\left[(\cdot)^iH\right](\overline x) \frac{\mathcal T\left[(\cdot)^{N^\epsilon+1}\right](\delta)}{\mathcal T\left[(\cdot)^{N^\epsilon+1}\right]({\overline x})}\right| &\le  K_2 \frac{{\overline x}^{\alpha^\epsilon-1}}{\sigma_\epsilon({\overline x})} \int_0^{\overline x} v^{i-\alpha^\epsilon} \sigma_\epsilon(v)dv \Wert H\Wert_\delta\quad \forall\,0<{\overline x}\le \delta.
\end{aligned}
\end{equation}
As above, we estimate $0<{\overline x}\le \delta_2\epsilon$ and $\delta_2\epsilon<{\overline x}\le \delta$ separately. In the former case, we directly obtain that 
\begin{align*}
\left|\mathcal T^\epsilon[(\cdot)^iH](\overline x) \frac{\mathcal T\left[(\cdot)^{N^\epsilon+1}\right](\delta)}{\mathcal T\left[(\cdot)^{N^\epsilon+1}\right]({\overline x})}\right| &\le \frac{K_2{\overline x}^i}{i+1-\alpha^\epsilon} \Wert H\Wert_\delta\\
&\le \frac{K_2{\overline x}^i}{i}\Wert H\Wert_\delta\quad \forall\,0<{\overline x}\le \delta_2\epsilon.
\end{align*}
We are therefore left with $\delta_2\epsilon<{\overline x}\le \delta$ where
\begin{align*}
\frac{{\overline x}^{\alpha^\epsilon-1}}{\sigma_\epsilon({\overline x})} \int_0^{\overline x} v^{i-\alpha^\epsilon} \sigma_\epsilon(v)dv &\le \int_0^{\overline x} v^{i-1} dv = \frac{{\overline x}^i}{i}\quad \forall\,\epsilon\delta_2<{\overline x}\le \delta,
\end{align*}
Here we have used that 
$$\frac{1}{(\delta_2\epsilon)^{1-\alpha^\epsilon}}\int_0^{\delta_2\epsilon} v^{i-\alpha^\epsilon} dv\le \int_0^{\delta_2\epsilon} v^{i-1} dv.$$
This completes the proof.

\textit{Proof of item \ref{boundT2i2}}.
This case is easy:
\begin{align*}
 \left|\mathcal T^\epsilon\left[(\cdot)^{lN^\epsilon+1}\right]({\overline x})\right| \le \delta^{(l-1)N^\epsilon} \mathcal T^\epsilon\left[(\cdot)^{N^\epsilon+1}\right]({\overline x})\quad \forall\,0\le {\overline x}\le \delta,
\end{align*}
and therefore 
\begin{align*}
\Wert \mathcal T^\epsilon\left[(\cdot)^{lN^\epsilon+1}\right]\Wert_\delta \le \delta^{(l-1)N^\epsilon}\mathcal T^\epsilon\left[(\cdot)^{N^\epsilon+1}\right](\delta).
\end{align*}

\textit{Proof of item \ref{Thanal}}. We have
\begin{align*}
 \vert H({\overline x})\vert \le  E R^{N^\epsilon+1} {\overline x}^{N^\epsilon+1} \sum_{k=0}^\infty R^{k}\delta^k = \frac{ER^{N^\epsilon+1} {\overline x}^{N^\epsilon+1}}{1-R\delta}\quad \forall 0\le {\overline x}\le \delta,
\end{align*}
and consequently
\begin{align*}
\left|\mathcal T^\epsilon[H]({\overline x}) \right|&\le  \frac{ER^{N^\epsilon+1}}{1-R\delta}\mathcal T^\epsilon\left[(\cdot)^{N^\epsilon+1}\right]({\overline x})\quad \forall\, 0<{\overline x}\le \delta.
\end{align*}
The result follows. 

          \textit{Proof of item \ref{composition}}.
%    Now, for item \ref{composition} we use the linearity of $\mathcal T^\epsilon$ and estimate the two terms independently. For the first one we can use \eqref{Tuih} with $i=2$ and obtain:
%    \begin{align*}
%     \Wert \mathcal T^\epsilon\left[(\cdot)^2h_0(\cdot) \overline Y(\cdot)\right]\Wert_\delta \le E\frac{K_2}{2} \delta^2 \Wert \overline Y\Wert_\delta.
%    \end{align*}
% Now, for the second term 
 Now, for item \ref{composition} we use the linearity of $\mathcal T^\epsilon$ and first estimate each of the terms of the sum $\sum_{l\ge 2} \mathcal T^{\epsilon} \left[H_l(\cdot )\overline Y^l\right]$. 
By \eqref{eq:Topdefn}, \eqref{eq:huest} and \eqref{eq:phiest}, we find that 
\begin{equation}\label{eq:compositionest}
\begin{aligned}
 \left| \mathcal T^\epsilon\left[ \overline Y^l \right]({\overline x})\frac{\mathcal T\left[(\cdot)^{N^\epsilon+1}\right](\delta)}{\mathcal T\left[(\cdot)^{N^\epsilon+1}\right]({\overline x})} \right| &\le K_2 \frac{{\overline x}^{\alpha^\epsilon-1}}{\sigma_\epsilon({\overline x}) }\int_0^{\overline x} v^{-\alpha^\epsilon} \sigma_\epsilon(v)\left(\frac{\mathcal T\left[(\cdot)^{N^\epsilon+1}\right](v)}{\mathcal T\left[(\cdot)^{N^\epsilon+1}\right](\delta)}\right)^{l-1} dv\Wert \overline Y\Wert_\delta^l\\
 &\le   K_2 C^{l-1} \frac{{\overline x}^{\alpha^\epsilon-1}}{\sigma_\epsilon({\overline x})} \int_0^{\overline x}  \frac{v^{(l-1)(N^\epsilon+1)-\alpha^\epsilon} \sigma_\epsilon(v)^l}{\delta^{(l-1)(N^\epsilon+1)} \sigma_\epsilon(\delta)^{l-1}} dv \Wert \overline Y\Wert_\delta^l\quad \forall\,0<{\overline x}\le \delta,
\end{aligned}
\end{equation}
with $C=C_2/C_1$. 
We claim that 
\begin{align}\label{eq:claim}
  \Wert \mathcal T^\epsilon\left[ \overline Y^l \right]\Wert_\delta\le \frac{2\epsilon}{l-1} K_2C^{l-1} \Wert \overline Y\Wert_\delta^l.
\end{align}
In order to prove this, we only have to show that 
\begin{align}
 \frac{{\overline x}^{\alpha^\epsilon-1}}{\sigma_\epsilon({\overline x})} \int_0^{\overline x}  \frac{v^{(l-1)(N^\epsilon+1)-\alpha^\epsilon} \sigma_\epsilon(v)^l}{\delta^{(l-1)(N^\epsilon+1)} \sigma_\epsilon(\delta)^{l-1}} dv\le \frac{2\epsilon}{l-1}\quad \forall\,0<{\overline x}\le \delta,\label{eq:compositionest1}
\end{align}
cf. \eqref{eq:compositionest}. Consider first the simplest case $0<{\overline x}\le \delta\le\delta_2\epsilon$. Then $\sigma_\epsilon({\overline x})=\sigma_\epsilon(v)=\sigma_\epsilon(\delta)=1$, for all $0\le {\overline x} \le v$, and we have
\begin{align*}
  \frac{{\overline x}^{\alpha^\epsilon-1}}{\sigma_\epsilon({\overline x})} \int_0^{\overline x}  \frac{v^{(l-1)(N^\epsilon+1)-\alpha^\epsilon} \sigma_\epsilon(v)^l}{\delta^{(l-1)(N^\epsilon+1)} \sigma_\epsilon(\delta)^{l-1}} dv &= \frac{1}{(l-1)(N^\epsilon+1)+1-\alpha^\epsilon }\left( \frac{{\overline x}}{\delta}\right)^{(l-1)(N^\epsilon+1)}\\
  &\le \frac{1}{(l-1)(N^\epsilon+1)+1-\alpha^\epsilon }\\
  &\le \frac{1}{(l-1)(N^\epsilon+1)+1-\alpha^\epsilon-l(1-\alpha^\epsilon) }\\
  &=\frac{1}{(l-1)(N^\epsilon+\alpha^\epsilon)}\\
  &= \frac{\epsilon}{l-1},
\end{align*}
and \eqref{eq:compositionest1} follows.
We are left with $0<{\overline x}\le \delta_2\epsilon<\delta$ and $\delta_2\epsilon<{\overline x}\le \delta$. For the former, we have  $\sigma_\epsilon({\overline x})=\sigma_\epsilon(v)=1$, $\sigma_\epsilon(\delta)=(\delta^{-1} \delta_2\epsilon)^{1-\alpha^\epsilon}$ and
\begin{align*}
 \frac{{\overline x}^{\alpha^\epsilon-1}}{\sigma_\epsilon({\overline x})} \int_0^{\overline x}  \frac{v^{(l-1)(N^\epsilon+1)-\alpha^\epsilon} \sigma_\epsilon(v)^l}{\delta^{(l-1)(N^\epsilon+1)} \sigma_\epsilon(\delta)^{l-1}} dv &= {\overline x}^{\alpha^\epsilon-1} \int_0^{\overline x} \frac{v^{(l-1)(N^\epsilon+1)-\alpha^\epsilon}}{\delta^{(l-1)(N^\epsilon+1)} (\delta^{-1} \delta_2\epsilon)^{(l-1)(1-\alpha^\epsilon)}}dv\\
 &=  \frac{1}{(l-1)(N^\epsilon+1)+1-\alpha^\epsilon} \left(\frac{{\overline x}}{\delta}\right)^{(l-1)(N^\epsilon+\alpha^\epsilon)} \left(\frac{{\overline x}}{\delta_2\epsilon}\right)^{(l-1)(1-\alpha^\epsilon)}\\
 &\le \frac{1}{(l-1)(N^\epsilon+1)+1-\alpha^\epsilon},\\
 &\le \frac{\epsilon}{l-1},
\end{align*}
and \eqref{eq:compositionest1} follows.
We finally consider $\delta_2\epsilon<{\overline x}\le \delta$:
\begin{align*}
 \frac{{\overline x}^{\alpha^\epsilon-1}}{\sigma_\epsilon({\overline x})} \int_0^{\overline x}  \frac{v^{(l-1)(N^\epsilon+1)-\alpha^\epsilon} \sigma_\epsilon(v)^l}{\delta^{(l-1)(N^\epsilon+1)} \sigma_\epsilon(\delta)^{l-1}} dv &=\frac{1}{(\delta_2\epsilon)^{1-\alpha^\epsilon} } \int_0^{\delta_2\epsilon} \frac{v^{(l-1)(N^\epsilon+1)-\alpha^\epsilon}}{\delta^{(l-1)(N^\epsilon+\alpha^\epsilon)} (\delta_2\epsilon)^{(l-1)(1-\alpha^\epsilon)}}dv\\
 &+\int_{\delta_2 \epsilon}^{\overline x} \frac{v^{(l-1)(N^\epsilon+1)-\alpha^\epsilon-l(1-\alpha^\epsilon)}}{\delta^{(l-1)(N^\epsilon+\alpha^\epsilon)}} dv\\
 &\le\frac{1}{(l-1)(N^\epsilon+1)+1-\alpha^\epsilon}  +\frac{1}{(l-1)(N^\epsilon+\alpha^\epsilon)}\\
 &\le \frac{2}{(l-1)(N^\epsilon+\alpha^\epsilon)}\\
 &=\frac{2 \epsilon}{l-1},
\end{align*}
and \eqref{eq:compositionest1} follows. Here we have used $\delta\ge \delta_2\epsilon$ in the denominator of the first integral on the right hand side. In turn, we obtain \eqref{eq:claim} and therefore 
% We therefore conclude that
\begin{align*}
 \Wert \mathcal T^\epsilon \left[\sum_{l=2}^\infty H_l(\cdot)\overline Y^l\right]\Wert_\delta &\le 2\epsilon E K_2C^{-1} R^{-1} \sum_{l=2}^\infty(CR\Wert \overline Y\Wert_\delta)^l\\
 &= 2\epsilon E K_2 \frac{CR  \Wert \overline Y\Wert_\delta^2 }{1-C R \Wert \overline Y\Wert_\delta}\\
 &\le 4\epsilon E K_2 C R \Wert \overline Y\Wert_\delta^2\quad \forall\,\Wert \overline Y\Wert_\delta<\frac12(CR)^{-1}.
\end{align*}
% for all $$.
The fact that the mapping $\overline Y\mapsto \mathcal T^\epsilon [\sum_{l=2}^\infty H_l(\cdot)\overline Y^l]$ is $C^1$ and a contraction for all $\epsilon>0$ small enough follows from identical computations. Further details are therefore left out.

\end{proof}
% 
% Now, for each $\epsilon^{-1}\notin \mathbb N$ define
% \begin{align*}
% N^\epsilon:=\lfloor \epsilon^{-1}\rfloor\in \mathbb N,\quad \alpha^\epsilon :=\epsilon^{-1}-N^\epsilon\in (0,1),
% \end{align*}
% such that 
% \begin{align*}
% \epsilon^{-1}  =N^\epsilon+\alpha^\epsilon.
% \end{align*}
% \begin{remark}\remlab{eqvnorm}
%  An important consequence of \eqref{TuN1bound} is that 
%  \begin{align*}
%   \Vert \mathcal T^{\epsilon}((\cdot)^{N^\epsilon+1})\Vert_\delta \ge \frac{K_2}{K_1}\Wert \mathcal T^{\epsilon}((\cdot)^{N^\epsilon+1})\Wert_\delta
%  \end{align*}
% 
% \end{remark}

The following will also be important: Define
\begin{align}\label{overlineU}
 \overline U^\epsilon(\overline x)&: = (N^\epsilon + a^\epsilon)  \overline w_{N^\epsilon}^\epsilon    \mathcal T^\epsilon\left[(\cdot)^{N^\epsilon+1}\right](\overline x);
%  \overline U^\epsilon_{N^\epsilon+1}(\overline x)&: = {(-1)^{N^\epsilon}S_{\infty}^0 \epsilon^{-2} \overline w_{N^\epsilon}^\epsilon  (N^\epsilon + a^\epsilon) \mathcal T^\epsilon((\cdot)^{N^\epsilon+1})(\overline x)}.
\end{align}
this quantity corresponds to $M_0$ in \cite[Lemma 7.2]{MR4445442}. 
\begin{lemma}\label{lemma:overlineU}
Fix $\delta_2>0$ and suppose that $\epsilon^{-1}\notin \mathbb N$. Then we have the following statements regarding $\overline U$:
\begin{enumerate}
\item \label{overlineU1} $\overline U$ has an absolutely convergent power series representation 
  \begin{align}
  \overline U^\epsilon(\overline x)  =\frac{\Gamma(\alpha^{\epsilon}) \Gamma(1-\alpha^\epsilon)}{\epsilon \Gamma(\epsilon^{-1})}\sum_{k=N^\epsilon+1}^\infty \frac{\Gamma(k+a^\epsilon)}{\Gamma(k+1-\epsilon^{-1})} \overline x^k \quad \forall\, -1<\overline x<1.\label{overlineUsum}
 \end{align}
%   and 
%   \begin{align*}
%    \operatorname{sign}(\overline U^\epsilon(\overline x)) = \operatorname{sign} ( \mathcal T^\epsilon\left[(\cdot)^{N^\epsilon+1}\right](\overline x)).
%   \end{align*}
% Finally, 
\item \label{overlineU2} For all $0<\epsilon\ll 1$, $\epsilon^{-1}\notin \mathbb N$, $0<\delta\le \frac34$ the following estimate holds:
 \begin{align}\label{eq:overlineUest}
  \Vert \overline U^\epsilon \Vert_\delta = \Wert \overline U^\epsilon \Wert_\delta \gtrless \frac{1}{1-\alpha^\epsilon} \Gamma(\alpha^\epsilon) (N^\epsilon)^{a^\epsilon+2-\alpha^\epsilon}\left(\frac{\delta}{1-\delta}\right)^{N^\epsilon+1} \sigma_\epsilon(\delta) \begin{cases}
                C_1,\\                                                                                                                                                                C_2.                                                                                                                                             \end{cases}
%  \Wert  
%   &\le \frac{K_2}{1-\alpha^\epsilon} \Gamma(\alpha^\epsilon) (N^\epsilon)^{a^\epsilon+1-\alpha^\epsilon}\vert S_\infty^0\vert \delta^{N^\epsilon+1}.
%   \Wert \overline U^\epsilon\Wert_\delta &\le \frac{K_2\vert S_\infty^0\vert  \Gamma(\alpha^\epsilon)}{1-\alpha^\epsilon} (1+o(1))(N^\epsilon)^{a^\epsilon+1-\alpha^\epsilon} \delta^{N^\epsilon+1}.
 \end{align}
Here $C_i=C_i(\delta_2,a^0)>0$, $i=1,2$. 
\item \label{overlineU3} Asymptotics for $\overline x=\mathcal O(\epsilon)$: Let $\overline x=\epsilon\overline x_2\in [-\epsilon\delta_2,\epsilon\delta_2]$, $\delta_2>0$ fixed. Then:
\begin{align}\label{eq:overlineU3}
 \overline U^\epsilon(\epsilon \overline x_2) =  \frac{\Gamma(\alpha^\epsilon) }{1-\alpha^\epsilon} (N^\epsilon)^{a^\epsilon+2-\alpha^\epsilon} (\epsilon \overline x_2)^{N^\epsilon+1} \left[1+\overline x_2\int_0^1 e^{(1-v)\overline x_2}  v^{1-\alpha^\epsilon} dv+o(1)\right],
\end{align}
with $o(1)$ uniform with respect to $\alpha^\epsilon\in (0,1)$.
\end{enumerate}

%  and
%  \begin{align*}
%   \overline U^\epsilon(\overline x)- (-1)^{N^\epsilon}S_{\infty}^0 \epsilon^{-2} \overline w_{N^\epsilon}^\epsilon  \overline x^{N^\epsilon} = S_\infty^0  \frac{\Gamma(\alpha^{\epsilon}) \Gamma(1-\alpha^\epsilon)}{(1-\epsilon)\Gamma(\epsilon^{-1}-1)}\sum_{k=N^\epsilon+1} \frac{\Gamma(k+a^\epsilon}{\Gamma(k+1-\epsilon^{-1})} \overline x^k
%  \end{align*}

\end{lemma}
\begin{proof}

We prove the items \ref{overlineU1}--\ref{overlineU3} successively in the following.

\textit{Proof of item \ref{overlineU1}}.  
For \eqref{overlineUsum} we use item \ref{TuN1} of Lemma \ref{lemma:TOp0} and \eqref{eq:kinf}:
\begin{equation}\label{eq:iamdone}
\begin{aligned}
 & \overline w_{N^\epsilon}^\epsilon  \times  (N^\epsilon + a^\epsilon) \mathcal T^\epsilon\left[(\cdot)^{N^\epsilon+1}\right](\overline x) \\
 &= \frac{\Gamma(\alpha^\epsilon) \Gamma(N^\epsilon+a^\epsilon)}{\epsilon \Gamma(\epsilon^{-1})}\times \frac{(N^\epsilon + a^\epsilon) \Gamma(1-\alpha^\epsilon)}{\Gamma(N^\epsilon+1+a^\epsilon)} \sum_{k=N^{\epsilon}+1}^\infty \frac{\Gamma(k+a^\epsilon)}{\Gamma(k+1-\epsilon^{-1})}\overline x^k\\
 &= \frac{\Gamma(\alpha^\epsilon)\Gamma(1+\alpha^\epsilon)}{\epsilon \Gamma(\epsilon^{-1})}\sum_{k=N^{\epsilon}+1}^\infty \frac{\Gamma(k+a^\epsilon)}{\Gamma(k+1-\epsilon^{-1})}\overline x^k.
\end{aligned}
\end{equation}
 %  For the sign of $\overline U^\epsilon$ we use item \ref{TuN1bound} of Lemma \ref{lemma:TOp}. 
 
 \textit{Proof of item \ref{overlineU2}}. 
Next, regarding  \eqref{eq:overlineUest} we use \eqref{VertWertT}, \eqref{barwkeps}, \eqref{eq:Gamma}, \eqref{stirling}, 
\begin{equation}\label{eq:overlineU2}
\begin{aligned}
\overline U^\epsilon  &=\frac{\Gamma(\alpha^\epsilon) \Gamma(N^\epsilon+a^\epsilon+1)}{\epsilon\Gamma(\epsilon^{-1})}  \mathcal T^\epsilon\left[(\cdot)^{N^\epsilon+1}\right] = \frac{\Gamma(\alpha^\epsilon) \Gamma(N^\epsilon+a^\epsilon+1)}{(1-\epsilon) \Gamma(N^\epsilon+\alpha^\epsilon-1)} \mathcal T^\epsilon\left[(\cdot)^{N^\epsilon+1}\right]\\
  & =\Gamma(\alpha^\epsilon) (1+o(1))  (N^\epsilon)^{a^\epsilon+2-\alpha^\epsilon} \mathcal T^\epsilon\left[(\cdot)^{N^\epsilon+1}\right],
\end{aligned}
\end{equation}
and subsequently \eqref{eq:phiest}.

\textit{Proof of item \ref{overlineU3}}. Finally, for \eqref{eq:overlineU3} we use \eqref{eq:overlineU2} and Lemma \ref{lemma:TOp} item \ref{Tu2}:
\begin{align*}
 \overline U^\epsilon(\epsilon \overline x_2) &= \Gamma(\alpha^\epsilon) (1+o(1))  (N^\epsilon)^{a^\epsilon+2-\alpha^\epsilon} \mathcal T^\epsilon\left[(\cdot)^{N^\epsilon+1}\right](\epsilon \overline x_2)\\
 &=\Gamma(\alpha^\epsilon) (N^\epsilon)^{a^\epsilon+2-\alpha^\epsilon} \frac{1}{1-\alpha^\epsilon} (\epsilon \overline x_2)^{N^\epsilon+1} \left[1+\overline x_2\int_0^1 e^{(1-v)\overline x_2}  v^{1-\alpha^\epsilon} dv+o(1)\right].
\end{align*}
% Now we use 
% \begin{align*}
%  \frac{\Gamma(\alpha^\epsilon)}{1-\alpha^\epsilon} = \frac{\Gamma(1-\alpha^\epsilon)\Gamma(\alpha^\epsilon)}{\Gamma(2-\alpha^\epsilon)},
% \end{align*}
% by \eqref{eq:Gamma}.

%   \begin{align*}
%   \frac{K_1}{1-\alpha^\epsilon} \Gamma(\alpha^\epsilon) (N^\epsilon)^{a^\epsilon+1-\alpha^\epsilon}\vert S_\infty^0\vert \delta^{N^\epsilon+1} \le \Wert \overline U^\epsilon \Wert  
%   &\le \frac{K_2}{1-\alpha^\epsilon} \Gamma(\alpha^\epsilon) (N^\epsilon)^{a^\epsilon+1-\alpha^\epsilon}\vert S_\infty^0\vert \delta^{N^\epsilon+1}.
%   \end{align*}
\end{proof}

\subsection{Solving for the analytic weak-stable manifold}

In the following, we write $\overline m^\epsilon$ as (\rspp{recall the notation in \eqref{eq:jn} and \eqref{eq:rn}})
\begin{align}
 \overline m^\epsilon =  j^{N^\epsilon}[\overline m^\epsilon]+\overline M^\epsilon,\quad r^{N^\epsilon}[\overline m^\epsilon]=:\overline M^\epsilon;\label{eq:yexpansion}
\end{align}
we will use $\mathcal T^\epsilon$ to set up a fixed-point equation for $\overline M^\epsilon$. 
For this purpose, let
$$\overline G^\epsilon(\overline x,\overline Y):= \overline g^\epsilon(\overline x,j^{N^\epsilon}[\overline m^\epsilon](\overline x)+\overline Y)-\overline g^\epsilon(\overline x,j^{N^\epsilon}[\overline m^\epsilon](\overline x)).$$
% Notice that this is well-defined since $\overline g(\cdot,\cdot,0)=0$.
We clearly have
\begin{align}\label{GG}
 \overline G^\epsilon(\overline x,\overline Y) =  \overline x^2 \overline G_1^\epsilon(\overline x) \overline Y+ \sum_{l\ge 2} \overline G_l^\epsilon(\overline x) \overline Y^l,
\end{align}
and for any $\overline D>0$:
$$\Vert \overline G_l^\epsilon\Vert_\delta \le \mu C \overline D^{-l+1}\quad\forall\, l\in \mathbb N,$$ with $C>0$ independent of $\mu$ and $\epsilon$, provided that \eqref{delta1} hold and that $0<\epsilon\ll 1$.  This is essentially identical to the computation leading to \eqref{eq:radiusproof} (with $j^{N^\epsilon-1}$ replaced by $j^{N^\epsilon}$), using $\Vert j^{N^\epsilon}[\overline m^\epsilon] \Vert_\delta\le K$, see \eqref{eq:JNbound}. 

% \begin{align*}
%  \overline G(\overline 
% \end{align*}

%This then leads to the following:
The following result corresponds to \cite[Lemma 7.3]{MR4445442}.
\begin{lemma}
$\overline Y=M^\epsilon$ satisfies the fixed-point equation:
\begin{equation}\label{Yfp}
\begin{aligned}
 \overline Y(\overline x) &= %(N^\epsilon +a^\epsilon)(-1)^{N^\epsilon} \epsilon^{-2} \overline w_{N^\epsilon}^\epsilon \overline S^\epsilon_{N^\epsilon}\mathcal T^\epsilon((\cdot)^{N^\epsilon+1})(\overline x)\\
 (N^\epsilon +a^\epsilon)(-1)^{N^\epsilon} \overline w_{N^\epsilon}^\epsilon \overline S^\epsilon_{N^\epsilon}\mathcal T^\epsilon\left[(\cdot)^{N^\epsilon+1}\right](\overline x)\\
 &-\mathcal T^\epsilon\left[r^{N^\epsilon}\left[\overline g^\epsilon(\cdot,j^{N^\epsilon}[\overline m^\epsilon])\right]\right](\overline x)-\mathcal T^\epsilon\left[ \overline G^\epsilon(\cdot,\overline Y)\right](\overline x).
\end{aligned}
\end{equation}
% with $o(1)$ independent of $\alpha^\epsilon$ and $\overline Y$.
\end{lemma}
\begin{proof}
With $\overline m^\epsilon_k$ given by \eqref{barmkeps},  we obtain
 \begin{align*}
 \epsilon \overline x(\overline x-1) \frac{d\overline M^\epsilon}{d\overline x}
  +(1+\epsilon a^\epsilon\overline x)\overline M^\epsilon &= - \epsilon (N^\epsilon+ a^\epsilon)\overline m^\epsilon_{N^\epsilon} \overline x^{N^\epsilon+1}\\
  &+\epsilon\left(r^{N^\epsilon}\overline g^\epsilon(\overline x,j^{N^\epsilon}[\overline m^\epsilon](\overline x))+ \overline G^\epsilon(\overline x,\overline M^\epsilon)\right).
\end{align*}
Using  $\overline m_{N^\epsilon}^\epsilon = (-1)^{N^\epsilon} \overline w_{N^\epsilon}^\epsilon \overline S_{N^{\epsilon}}^\epsilon$, \eqref{overlineU} and Lemma \ref{lemma:TOp0} item \ref{Tdiff}, the result follows;  notice again the change of sign when comparing \eqref{Theqn} and \eqref{ybarEqn}.
\end{proof}

We denote the right hand side of \eqref{Yfp} by $\mathcal F(\overline Y)(\overline x)$:
\begin{equation}\label{eq:mathcalF}
\begin{aligned}\mathcal F(\overline Y):&=(N^\epsilon +a^\epsilon)(-1)^{N^\epsilon} \overline w_{N^\epsilon}^\epsilon \overline S^\epsilon_{N^\epsilon}\mathcal T^\epsilon\left[(\cdot)^{N^\epsilon+1}\right]\\
&-\mathcal T^\epsilon\left[r^{N^\epsilon}\left[\overline g^\epsilon(\cdot,j^{N^\epsilon}[\overline m^\epsilon])\right]\right]-\mathcal T^\epsilon\left[ \overline G^\epsilon(\cdot,\overline Y)\right],\end{aligned}\end{equation}
and
define the closed ball
  $$\mathcal B^C:=\mathcal D^\epsilon_\delta\cap \left\{\Wert Y\Wert_\delta \le C\right\},$$
of radius $C>0$.
% \fbox{Better to formulate this slightly differently:}
% 
% \fbox{$\overline Y=\overline Y_0+\overline Y_1$ where $\overline Y_0$ is the sum of the first two terms in \eqref{Yfp}}
% 
% \fbox{For any $C>0$, let $\delta>0$ be so that $\Wert \overline Y_0\Wert \le \frac12 C$.}
% 
% \fbox{Then $\mathcal F:B_C^{\mathcal D^\epsilon_\delta}\rightarrow B_C^{\mathcal D^\epsilon_\delta}$ well-defined and a contraction}
\begin{proposition}\label{prop:fixpoint}
Suppose that $S_\infty^0\ne 0$ and that $\mu\in [0,\mu_0)$ with $\mu_0>0$ small enough. Then for any $K>0$, $\alpha^\epsilon\in (0,1)$, $N^\epsilon\gg 1$, there is an $0<\overline \delta \le \frac34$ such that for any $0<\delta\le \overline \delta$ the following holds:
\begin{enumerate}
 \item \label{condK} Boundedness of the ``leading order term'': \begin{align}
 \Vert  \overline w_{N^\epsilon}^\epsilon (\cdot)^{N^\epsilon}\Vert_{\delta} +\Vert  \overline U^\epsilon\Vert_\delta \le \frac{K}{\vert S_{\infty}^0 \vert},\label{eq:condK}
\end{align}
\item \label{Fcontraction} $\mathcal F:\mathcal B^{2K}_\delta\rightarrow \mathcal B^{2K}_\delta$, defined by \eqref{eq:mathcalF}, is a contraction. 
\end{enumerate}
% and such that (b) 
% %   For any $C>0$, suppose that
%  \begin{align}\label{deltaFP}
%   \delta\le \min\left(\delta_0,2\delta_0\left(C(1-\alpha^\epsilon)\right)^{\frac{1}{N^\epsilon}},\left(\frac{C(1-\alpha^\epsilon)}{8\vert  S_{\infty}^0\vert \vert \Gamma(\alpha^\epsilon) (N^\epsilon)^{a^\epsilon+1-\alpha}}\right)^{\frac{1}{N^\epsilon}}\right),
%   \end{align}
%   where $\alpha^\epsilon\in (0,1)$ and $N^\epsilon\gg 1$. Moreover, define
%   $$\mathcal B^C:=\mathcal D^\epsilon_\delta\cap \left\{\Vert Y\Vert \le C\right\}.$$
%     Then 
%  $\mathcal F:\mathcal B^C\rightarrow \mathcal B^C$ is well-defined and a contraction.
\end{proposition}
\begin{proof}
%  
%  
% \end{proof}
% 
% Firstly, we notice that
% \begin{equation}\label{sterlingestimate}
%  \begin{aligned}
%  (N^\epsilon +a^\epsilon)\epsilon^{-2} \overline w_{N^\epsilon}^\epsilon  &=   \frac{\Gamma(\alpha^\epsilon)\Gamma(N^\epsilon+1+a^\epsilon)}{(1+\epsilon) \Gamma(N^\epsilon + \alpha^\epsilon)} \\
%   &=\Gamma(\alpha^\epsilon) (1+o(1))(N^\epsilon)^{a^\epsilon+1-\alpha^\epsilon},
%  \end{aligned}
%  \end{equation}
%  using $$\epsilon^2 \Gamma(\epsilon^{-1}+1)=\epsilon (\epsilon^{-1}-1)\Gamma(\epsilon^{-1}-1)=(1-\epsilon)\Gamma(N^\epsilon+\alpha^\epsilon),$$ and Stirling's formula for $N^\epsilon\rightarrow \infty$ (see \eqref{this2}); here $o(1)$ is independent of $\alpha^\epsilon$. 
 
 The first claim in item \ref{condK} is obvious as $\overline U^\epsilon(0)=0$, see also Lemma \ref{lemma:overlineU} and \eqref{mNeps}. We therefore proceed to proof item \ref{Fcontraction} regarding $\mathcal F:\mathcal B^{2K}_\delta\rightarrow \mathcal B^{2K}_\delta$ being a contraction. For this purpose, we estimate each of the three terms on the right hand side \eqref{Yfp} in the norm $\Wert \cdot\Wert_\delta$, recall \eqref{Cepsnorm}.

 \textbf{The first term}:
 \begin{align*}
%    (N^\epsilon +a^\epsilon)(-1)^{N^\epsilon} \epsilon^{-2} \overline w_{N^\epsilon}^\epsilon \overline S^\epsilon_{N^\epsilon}\mathcal T^\epsilon((\cdot)^{N^\epsilon+1})(\overline x).
 (N^\epsilon +a^\epsilon)(-1)^{N^\epsilon} \overline w_{N^\epsilon}^\epsilon \overline S^\epsilon_{N^\epsilon}\mathcal T^\epsilon\left[(\cdot)^{N^\epsilon+1}\right](\overline x).
 \end{align*}
%  Using \eqref{sterlingestimate} and item \ref{Thanal} of Lemma \ref{lemma:TOp}, we obtain
%  \begin{align*}
%  \Wert \cdots \Wert  
%   &\le \frac{2}{1-\alpha^\epsilon} \Gamma(\alpha^\epsilon) (1+o(1))(N^\epsilon)^{a^\epsilon+1-\alpha^\epsilon}\vert S_\infty^0\vert \delta^{N^\epsilon+1}.
%  \end{align*}
We write the first term as
\begin{align*}
(1+o(1))(-1)^{N^\epsilon} S_{\infty}^0 \overline U^\epsilon(\overline x),
\end{align*}
using Lemma \ref{lemma:SNepslimit} and \eqref{overlineU}. Then by using \eqref{overlineUsum} and the assumption on $\delta$, see \eqref{eq:condK}, we have $$\Wert (N^\epsilon +a^\epsilon)(-1)^{N^\epsilon} \overline w_{N^\epsilon}^\epsilon \overline S^\epsilon_{N^\epsilon}\mathcal T^\epsilon\left[(\cdot)^{N^\epsilon+1}\right]\Wert_\delta \le \frac43 K.$$ 
% We use Lemma \ref{lemma:overlineU}.
%  By \eqref{deltaFP}, the right hand side is $\le \frac14(1+o(1)) C\le \frac13 C$.
 
  \textbf{The second term}:
  \begin{align*}
  -\mathcal T^\epsilon\left[r^{N^\epsilon}\left[\overline g^\epsilon(\cdot,j^{N^\epsilon}[\overline m^\epsilon])\right]\right](\overline x).
  \end{align*}  
  For the second term, we first use \rspp{Lemma \ref{lemma:tildeg}, setting 
  \begin{align*}
  q=\overline m^\epsilon_{N^\epsilon} \overline x^{N^\epsilon}
%    \rspp{\overline g^\epsilon(\overline x,j^{N^\epsilon}[\overline m^\epsilon](\overline x)))=\overline g^\epsilon(\overline x,j^{N^\epsilon-1}[\overline m^\epsilon](\overline x) +\overline m^\epsilon_{N^\epsilon} \overline x^{N^\epsilon} ).}
  \end{align*}
  in \eqref{tildeg}.}
  Therefore
  \begin{align*}
r^{N^\epsilon}\left[\overline g^\epsilon(\cdot,j^{N^\epsilon}[\overline m^\epsilon]))\right](\overline x)&= \sum_{k=N^{\epsilon}+1}^\infty \rsp{\overline{\mathcal G}^\epsilon[j^{N^\epsilon-1}[\overline m^\epsilon]]_k} \overline x^k +\rspp{\overline g_1^\epsilon}(\overline x) \overline m_{N^\epsilon}^\epsilon \overline x^{N^\epsilon+2}\\
&+\sum_{l= 2}^\infty \rspp{\overline g_{l}^\epsilon}(\overline x) (\overline m_{N^\epsilon}^\epsilon)^l \overline x^{ lN^\epsilon+1},
\end{align*}
using \eqref{tildeg}.
We now use item \ref{Thanal} of Lemma \ref{lemma:TOp}:
\begin{align*}
\Wert \mathcal T^\epsilon \left[\sum_{k=N^{\epsilon}+1}^\infty \rsp{\overline{\mathcal G}^\epsilon[j^{N^\epsilon-1}[\overline m^\epsilon]]_k} (\cdot)^k\right]\Wert_\delta &\le \frac{ER^{N^\epsilon+1}}{1-R\delta} \mathcal T\left[(\cdot)^{N^\epsilon+1}\right](\delta).
% &\le \frac{2K_2}{1-\alpha^\epsilon} ER^{N^\epsilon+1}\delta^{N^\epsilon+1}.
\end{align*}
Here $$E=C (\overline w_2^\epsilon)^2e^{-4Q_4^\epsilon}, \quad R=e^{Q_4^\epsilon},$$ cf. \eqref{eq:gkest} and \eqref{AepsBeps}\rspp{.} Consequently,
\begin{align*}
 ER^{N^\epsilon+1} = C (\overline w_2^\epsilon)^2 e^{Q_4^\epsilon (N^\epsilon-3)}=C  \overline w_2^\epsilon \overline w_{N^\epsilon-1}^\epsilon,
\end{align*}
using \eqref{AepsBeps} and therefore $R\delta\le \frac45$ for $0<\delta\le \frac34$ for all $0<\epsilon\ll 1$. \rspp{We then} arrive at 
\begin{align*}
\Wert \mathcal T^\epsilon \left[\sum_{k=N^{\epsilon}+1}^\infty \rsp{\overline{\mathcal G}^\epsilon[j^{N^\epsilon-1}[\overline m^\epsilon]]_k} (\cdot)^k\right]\Wert_\delta&\le 5 C  \overline w_2^\epsilon\overline w_{N^\epsilon-1}^\epsilon \mathcal T\left[(\cdot)^{N^\epsilon+1}\right](\delta)\\
&\le \epsilon^2 \overline C  K,
\end{align*}
with $\overline C>0$ large enough,
for all $0<\epsilon\ll 1$. Here we have also used \eqref{eq:overlineUest}, \eqref{eq:iamdone}, \eqref{eq:condK}, $\Gamma(\alpha^\epsilon)>1$ and $ \overline w_2^\epsilon=\mathcal O(\epsilon)$ to estimate
\begin{align*}
 5 C  \overline w_2^\epsilon\overline w_{N^\epsilon-1}^\epsilon \mathcal T\left[(\cdot)^{N^\epsilon+1}\right](\delta)\le   \epsilon^2 \overline C \vert S_\infty^0\vert \Wert \overline U^\epsilon \Wert_\delta\le \epsilon^2\overline C K,
\end{align*}
in the last inequality.
Proceeding completely analogously, we obtain
\begin{align*}
  \Wert \mathcal T^\epsilon \left[\rspp{\overline g_1^\epsilon}(\cdot) \overline m_{N^\epsilon}^\epsilon (\cdot)^{N^\epsilon+2}\right]\Wert_\delta &\le  \mu C \vert S_\infty^0\vert  \overline w_{N^\epsilon}^\epsilon \delta \mathcal T\left[(\cdot)^{N^\epsilon+1}\right](\delta) \le \mu \epsilon \overline C,
\end{align*}
and by Lemma \ref{lemma:TOp} item \ref{boundT2i2}
\begin{align*}
\\
%  &\le \mu \overline C ,\\
 \Wert \mathcal T^\epsilon \left[\sum_{l=2}^\infty\rspp{\overline g_l^\epsilon}(\cdot)  (\overline m_{N^\epsilon}^\epsilon)^{l} (\cdot)^{lN^\epsilon+1}\right]\Wert_{\delta} &\le \mu C K_2 \sum_{l=2}^\infty \overline D^{-l+1} \vert \overline m_{N^\epsilon} \vert^l \delta^{N^\epsilon(l-1)}\mathcal T^\epsilon\left[(\cdot)^{N^\epsilon+1}\right](\delta)\\
 &= \mu C K_2 \sum_{l=2}^\infty \left( \overline D^{-1} \vert \overline m_{N^\epsilon} \vert \delta^{N^\epsilon}\right)^{l-1} \vert \overline m_{N^\epsilon}\vert \mathcal T^\epsilon\left[(\cdot)^{N^\epsilon+1}\right](\delta)\\
 &=\mu C K_2  \frac{\overline D^{-1}\vert \overline m_{N^\epsilon} \vert \delta^{N^\epsilon}}{1- \overline D^{-1} \vert \overline m_{N^\epsilon} \vert \delta^{N^\epsilon}} \vert \overline m_{N^\epsilon}\vert \mathcal T^\epsilon\left[(\cdot)^{N^\epsilon+1}\right](\delta)\\
 &\le \mu \epsilon \overline C,
%  &= \mu C K_2 \frac{\vert \overline m_{N^\epsilon}\vert \delta^{N^\epsilon} }{1-\overline D^{-1} \vert \overline m_{N^\epsilon}\vert \delta^{N^\epsilon} } \vert \overline m_{N^\epsilon} \ver\mathcal T^\epsilon\left[(\cdot)^{N^\epsilon+1}\right](\delta)
%  &\le \mu \overline C,FIX THISSS
%  
%  
%  
%  (\vert \overline m_{N^\epsilon}^\epsilon\vert \delta^{N^\epsilon})^{l-1} \vert \overline m_{N^\epsilon}^\epsilon\vert  \delta \mathcal T\left[(\cdot)^{N^\epsilon+1}\right](\delta) \\
%  &\le \frac{\mu C(1+o(1))K_2}{N^\epsilon} \frac{(\overline D^{-1} K)^2}{1-\overline D^{-1} K}.
%  &\le \overline C \epsilon,
\end{align*}
for some $\overline C>0$ independent of $\mu$ and $\epsilon$. Here we have again used \eqref{eq:condK}
% \begin{align*}
%  \vert \overline m_{N^\epsilon}\vert \mathcal T^\epsilon\left[(\cdot)^{N^\epsilon+1}\right](\delta) = \mathcal O(\epsilon)
% \end{align*}
and taken $\overline D>0$ large enough. %$\overline C>0$ is large enough.
% Here we have again used \eqref{eq:condK}. 
% In total, using that $0<\delta\le \overline \delta\le \delta_2\epsilon$, we conclude that
In total, we conclude that
\begin{align*}
 \Wert -\mathcal T^\epsilon\left[r^{N^\epsilon} \left[\overline g^\epsilon(\cdot,j^{N^\epsilon}[\overline m^\epsilon])\right]\right] \Wert_\delta &\le \mathcal O(\epsilon).% \frac{2c K_2}{1-\alpha^\epsilon} (\epsilon^{-2} \overline w_2^\epsilon) (\epsilon^{-2}\overline w_{N^\epsilon-1}) \delta^{N^\epsilon+1} \\
%  &+\overline C \delta^2 \vert  S_{\infty}^0\vert \Vert (\epsilon^{-2} \overline w_{N^\epsilon}^\epsilon) (\cdot)^{N^\epsilon} \Vert_\delta ,
\end{align*}
 %Here we have also used that $\Wert \overline m_{N^\epsilon}^\epsilon \overline x^{ N^\epsilon+1}\Wert_\delta=\Vert \overline m_{N^\epsilon}^\epsilon \overline x^{ N^\epsilon+1}\Vert_\delta$. 
It follows that the second term is bounded by $\frac13 K$ for all $0<\epsilon\ll 1$.
% assuming again that $S_\infty^0\ne 0$, provided that $\vert \overline m_{N^\epsilon}^\epsilon \delta^{N^\epsilon+3}\vert< \overline D$, $0<\delta\le \delta_0$ and $0<\epsilon\ll 1$. Here $c>0$ is large enough but independent of $\epsilon$. By \eqref{deltaFP}, the right hand side is $\le \frac13 C$.

\textbf{The third term}:
\begin{align*}
 -\mathcal T^\epsilon\left[  \overline G^\epsilon(\cdot,\overline Y)\right](\overline x).
\end{align*}
For the third term, we also use items \ref{boundT2i} (with $i=2$) and \ref{composition} of Lemma \ref{lemma:TOp}. Specifically, by \eqref{GG}, we directly obtain $$\Wert-\mathcal T^\epsilon\left[  \overline G^\epsilon(\cdot,\overline Y)\right] \Wert_\delta = O(\mu)\le \frac13 K,$$ for $\Wert Y\Wert_\delta \le 2K$ by taking $\overline D>0$ large enough. 

In total, $\mathcal F:\mathcal B^{2K}_\delta\rightarrow \mathcal B^{2K}_\delta$ is well-defined when \eqref{eq:condK} holds and $0<\epsilon\ll 1$, $\epsilon^{-1}\notin \mathbb N$ for $\mu>0$  small enough. The fact that $\mathcal F$ is a contraction follows directly from using item \ref{composition} of Lemma \ref{lemma:TOp} on the third term, see  \eqref{DY}. In fact, we find that the Lipschitz constant is $O(\mu)$ for all $0<\epsilon\ll 1$ (independently of $\alpha^\epsilon\in (0,1)$). 
\end{proof}

\begin{proposition}\label{prop:final}
Assume that the conditions of Proposition \ref{prop:fixpoint} hold true and that \eqref{eq:condK} holds. Let $\overline M^\epsilon$ denote the solution of the fixed-point equation \eqref{Yfp} \rsp{and define
\begin{align*}
\overline B^\epsilon(\overline x):=&j^{N^\epsilon-1}[\overline m^\epsilon](\overline x) -\mathcal T^\epsilon\left[r^{N^\epsilon}\left[\overline g^\epsilon(\cdot,j^{N^\epsilon}[\overline m^\epsilon])\right]\right](\overline x)-\mathcal T^\epsilon \left[\overline G^\epsilon(\cdot,\overline M^\epsilon)\right](\overline x),\\
 \overline V^\epsilon(\overline x): =&\overline w_{N^\epsilon}^\epsilon \overline x^{N^\epsilon} +\overline U^\epsilon(\overline x).
 \end{align*}
Then 
%  This then leads to the following form of the weak analytic manifold
\begin{equation}\label{eq:mepsExpansion}
\begin{aligned}
 \overline m^\epsilon(\overline x) =& \overline B^\epsilon(\overline x)+(1+o(1))(-1)^{N^\epsilon}S_{\infty}^0\overline V^\epsilon(\overline x), 
%  
%  r^{N^\epsilon}\overline g^\epsilon(\overline x,j^{N^\epsilon}[\overline m^\epsilon](\overline x))+ \overline G^\epsilon(\overline x,\overline M^\epsilon(\overline x))
%  \underbrace{\left\{j^{N^\epsilon-1}[\overline m^\epsilon](\overline x) -\mathcal T^\epsilon\left[\rsp{Z^\epsilon}\right](\overline x)\right\}}_{:=\overline B^\epsilon(\overline x)}+(1+o(1))(-1)^{N^\epsilon}S_{\infty}^0 \underbrace{\left\{  \overline w_{N^\epsilon}^\epsilon \overline x^{N^\epsilon} +\overline U^\epsilon(\overline x)\right\}}_{:=\overline V^\epsilon(\overline x)},
%  &+  (1+o(1)) \underbrace{(-1)^{N^\epsilon}S_{\infty}^0 (\epsilon^{-2} \overline w_{N^\epsilon}^\epsilon ) \left[\overline x^{N^\epsilon}+ (N^\epsilon + a^\epsilon) \mathcal T^\epsilon((\cdot)^{N^\epsilon+1})(\overline x)\right]}_{=\overline V^\epsilon(\overline x)},
%  &:=\overline P^\epsilon(\overline x)+\overline U^\epsilon^\epsilon(\overline x)
\end{aligned}
\end{equation}
for all $0\le \overline x \le \delta$.}
% , with this equation defining $\overline B^\epsilon$ and $\overline V^\epsilon$, and
% where
% \begin{align*}
%  \rsp{Z^\epsilon (\overline x):=r^{N^\epsilon}\overline g^\epsilon(\overline x,j^{N^\epsilon}[\overline m^\epsilon](\overline x))+ \overline G^\epsilon(\overline x,\overline M^\epsilon(\overline x)).}
% \end{align*}
Moreover, the following estimates hold true:
\begin{align}\label{eq:BVest}
 \Vert \overline B^\epsilon\Vert_{\delta}=\mathcal O(\epsilon),\quad \Vert\overline V^\epsilon(\overline x) \Vert_\delta \le \frac{K}{\vert S_\infty\vert}.
\end{align}

\end{proposition}
\begin{proof}
\eqref{eq:mepsExpansion} follows directly from \eqref{eq:yexpansion} with $\overline Y=\overline M^\epsilon$ given by \eqref{Yfp}.  Subsequently,  the estimate for $\overline B^\epsilon$ follows from \eqref{eq:Vertdefn}
and the proof of Proposition \ref{prop:fixpoint} (see the second and third terms). Finally for the estimate for $\overline V^\epsilon$, we use that $\overline U^\epsilon$ is an increasing function of $\overline x\in [0,\delta]$ and
\begin{align*}
 \Vert \overline V^\epsilon\Vert_{\delta} &=\Vert \overline w_{N^\epsilon}^\epsilon (\cdot)^{N^\epsilon} + \overline U^\epsilon\Vert_{\delta}=\overline w_{N^\epsilon}^\epsilon \delta^{N^\epsilon} + \overline U^\epsilon(\delta)
 =\Vert \overline w_{N^\epsilon}^\epsilon (\cdot)^{N^\epsilon} \Vert_\delta+\Vert \overline U^\epsilon\Vert_{\delta}\le \frac{K}{\vert S_\infty^0\vert},
\end{align*}
using \eqref{eq:condK} in the final inequality.
\end{proof}

$\overline V^\epsilon$ has the following absolutely convergent power series expansion
\begin{align} \label{eq:Veps}
 \overline V^\epsilon(\overline x) = \frac{\Gamma(\alpha^\epsilon)\Gamma(1-\alpha^\epsilon) }{\epsilon\Gamma(\epsilon^{-1})}\sum_{k=N^\epsilon}^\infty \frac{\Gamma(k+a^\epsilon)}{\Gamma(k+1-\epsilon^{-1})}\overline x^k\quad -1< \overline x< 1.
\end{align}
This follows from \eqref{overlineUsum}.

\begin{remark}
  The quantity $\overline V^\epsilon$ corresponds to $\Theta_{\text{main}}$ in \cite[Eq. (7.25)]{MR4445442}.
 \end{remark}

% subsection{Completing the proof}
 \subsection{Completing the proof of \lemmaref{lemma:Veps}}\label{sec:nodefinal}
 
We prove the items \ref{Vepsprop}--\ref{expansionV} successively in the following.

\textit{Proof of item \ref{Vepsprop}}. The power series expansion \eqref{eq:Veps}  of $\overline V^\epsilon$ in \eqref{eq:mepsExpansion} was proven in Proposition \ref{prop:final}. 
 The absolute convergence of this power series expansion follows from \eqref{stirling}:
 \begin{align}\label{eq:Vepscoeff}
  \frac{\Gamma(k+a^\epsilon)}{\Gamma(k+1-\epsilon^{-1}) } = (1+o(1)) k^{\epsilon^{-1}+a^\epsilon-1}.
 \end{align}
  Each term of the series $\overline V^\epsilon$ is positive for $\overline x>0$, which implies \eqref{eq:Vepsprop}.
  
  \textit{Proof of item \ref{Vepsprop2}}. For the lower bound, we use $\overline w_{N^\epsilon}^\epsilon >0$, the definition of $\overline V^\epsilon$:
  \begin{align*}
   \overline V^\epsilon(\overline x) = \overline w_{N^\epsilon}^\epsilon \overline x^{N^\epsilon} + \overline U^\epsilon(\overline x)\ge \overline U^\epsilon(\overline x) \quad \forall\, 0\le \overline x<1,
  \end{align*}
  and the lower bound of $\overline U^\epsilon$ (see \eqref{eq:overlineUest}):
  \begin{align*}
   \overline V^\epsilon(\overline x)\ge C_1 \frac{1}{1-\alpha^\epsilon}\Gamma(\alpha^\epsilon) (N^\epsilon)^{a^\epsilon+2-\alpha^\epsilon} \left(\frac{\overline x}{1-\overline x}\right)^{N^\epsilon+1}\sigma_\epsilon(\overline x)\quad \forall\, 0\le \overline x\le \frac34. 
  \end{align*}
  We take $\delta_2=\frac34$ in \eqref{eq:phibound} and obtain
  \begin{align*}
   \sigma_\epsilon(\overline x) \ge \epsilon^{1-\alpha^\epsilon}\quad \forall\, 0\le \overline x\le \frac34, 
  \end{align*}
  which together with 
  \begin{align*}
    (N^\epsilon)^{a^\epsilon+2-\alpha^\epsilon} \ge \frac12 \epsilon^{\alpha^\epsilon-a^\epsilon-2} \quad \mbox{and} \quad \frac{1}{1-\alpha^\epsilon}\Gamma(\alpha^\epsilon)>1 \quad \forall\,\alpha^\epsilon\in (0,1),\,0<\epsilon\ll 1,
    \end{align*}
    leads to 
\begin{align*}
\overline V^\epsilon(\overline x)\ge \frac12 C_1  \epsilon^{-a^\epsilon-1} \left(\frac{\overline x}{1-\overline x}\right)^{N^\epsilon+1} \quad \forall\, 0\le \overline x\le \frac34. 
\end{align*}
We now use that $a^0>-2$, recall \rsp{Hypothesis} \ref{assa}, to conclude that 
\begin{align*}
 \frac12 C_1  \epsilon^{-a^\epsilon-1}\ge \epsilon \Longleftrightarrow \frac12 C_1 \ge \epsilon^{2+a^\epsilon}\quad \forall\,0<\epsilon\ll 1.
\end{align*} 
Indeed, let $\nu=2+a^0>0$. Then by taking $0<\epsilon\ll 1$, we have that $\vert a^\epsilon-a^0\vert \le \frac12 \nu $ and hence
\begin{align*}
 \epsilon^{2+a^\epsilon}\le \epsilon^{\nu} \epsilon^{-\frac12 \nu} = \epsilon^{\frac12 \nu}\rightarrow 0\quad \mbox{for}\quad \epsilon\rightarrow 0.
\end{align*}
% as $\epsilon\rightarrow 0$. 
This completes the proof of item \ref{Vepsprop2}.

\textit{Proof of item \ref{Vepsblowup}}. The divergence with respect to $\alpha^\epsilon\rightarrow 0^+$ and $1^-$ is a direct consequence of the factors $\Gamma(\alpha^\epsilon)\Gamma(1-\alpha^\epsilon)$ in the definition of $\overline V^\epsilon$, recall \eqref{eq:Gamma0}.

\textit{Proof of item \ref{expansionV}}.
%   and also $\overline V^\epsilon(\overline x)\rightarrow \infty$ for $\overline x\rightarrow 1^-$ when combined with \eqref{eq:Vepscoeff}. 
%  The lower bound \eqref{eq:Vepsprop2} for $\epsilon\le \overline x\le \frac34$ follows from \eqref{eq:mepsExpansion}, $\Gamma(\alpha^\epsilon)\Gamma(1-\alpha^\epsilon)>1$ for all $\alpha^\epsilon\in (0,1)$ and \eqref{eq:overlineUest} with $\sigma_\epsilon(\overline x) = (\overline x^{-1}\epsilon)^{1-\alpha^\epsilon}$ for $\delta_2=0$. 
  Finally, in order to obtain \eqref{eq:expansionV} we use the form in \eqref{eq:mepsExpansion}:
 \begin{align*}
  \overline V^\epsilon(\overline x)=\overline w_{N^\epsilon}^\epsilon \overline x^{N^\epsilon}+\overline U^{\epsilon}(\overline x),
 \end{align*}
\eqref{eq:overlineU3} and
\begin{align*}
 \overline w_{N^\epsilon}^\epsilon \overline x^{N^\epsilon} = (1+o(1)) \Gamma(\alpha^\epsilon) (N^\epsilon)^{a^\epsilon+1-\alpha^\epsilon} \overline x^{N^\epsilon}.
\end{align*}
This gives
\begin{align*}
 \overline V^\epsilon(\epsilon \overline x_2) &= (1+o(1)) \Gamma(\alpha^\epsilon) (N^\epsilon)^{a^\epsilon+1-\alpha^\epsilon} (\epsilon \overline x_2)^{N^\epsilon}\\
 &+ \frac{\Gamma(\alpha^\epsilon) }{1-\alpha^\epsilon} (N^\epsilon)^{a^\epsilon+2-\alpha^\epsilon} (\epsilon \overline x_2)^{N^\epsilon+1} \left[1+\overline x_2\int_0^1 e^{(1-v)\overline x_2}  v^{1-\alpha^\epsilon} dv+o(1)\right]\\
 &=(1+o(1)) \Gamma(\alpha^\epsilon) (N^\epsilon)^{a^\epsilon+1-\alpha^\epsilon} (\epsilon \overline x_2)^{N^\epsilon} \\
 &\times \left(1+ \frac{\overline x_2}{1-\alpha^\epsilon}  \left[1+\overline x_2\int_0^1 e^{(1-v)\overline x_2}  v^{1-\alpha^\epsilon} dv+o(1)\right]\right).
\end{align*}

\subsection{Completing the proof of \thmref{main2}}\label{sec:completing}

We consider $I$ of the forms $[0,\delta]$, $0<\delta\le \frac32$, and $[-\delta,0]$, $0<\delta\le \delta_2 \epsilon$, separately in the following and prove the statement of \thmref{main2} in these cases. It will then follow that the statement is true for any $I\subset [-\delta_2\epsilon,\frac34]$ satisfying \eqref{VcondK}.

\textbf{The case $I=[0,\delta]$}. If $I=[0,\delta]$, $\delta\le \frac34$, satisfies \eqref{VcondK}, then
\begin{align*}
 \Vert \overline V^\epsilon \Vert_\delta = \Vert \overline w_{N^\epsilon}^\epsilon (\cdot)^{N^\epsilon}\Vert_\delta+\Vert \overline U^\epsilon\Vert_\delta \le K.
\end{align*}
% We therefore apply 
%  Upon using \eqref{eq:expansionV}, we obtain from \eqref{} that
% \begin{align*}
%  K\ge \Vert \overline V^\epsilon \Vert_\delta & \ge \Vert \overline w_{N^\epsilon} (\cdot)^{N^\epsilon}\Vert_\delta + \frac{K_2}{K_1}\Wert \overline U^\epsilon \Wert_\delta\\
%  &\ge \frac{K_2}{K_1} \left(\Vert \overline w_{N^\epsilon} (\cdot)^{N^\epsilon}\Vert_\delta + \Wert \overline U^\epsilon \Wert_\delta\right).
% \end{align*}
% Here $0<K_1<K_2$ are defined by \eqref{eq:TuN1bound}, recall \remref{eqvnorm}. 
% 
We therefore apply 
Proposition \ref{prop:fixpoint} with $K$ replaced by $K\vert S_\infty^0\vert$. In this case, \thmref{main2} then follows from Proposition \ref{prop:final}, see \eqref{eq:mepsExpansion} and \eqref{eq:BVest}.

\textbf{The case $I=[-\delta,0]$}.
Now, we consider case $I=[-\delta,0]$ with $0<\delta\le\delta_2\epsilon$. In this case, we adapt the space $\mathcal D^\epsilon_\delta$ and the norm $\Wert \cdot\Wert_{\delta}$ in the following way:
\begin{align*}
\mathcal D^\epsilon_\delta:= \{H:[-\delta,0] \rightarrow \mathbb R\,\,\text{analytic}\,:\,\Wert H\Wert_{\delta} <\infty\},
\end{align*}
where
\begin{align}
 \Wert H\Wert_{\delta}:=\sup_{\overline x\in [-\delta,0)}\left| H(u) \frac{\mathcal T\left[(\cdot)^{N^\epsilon+1}\right](-\delta)}{\mathcal T\left[(\cdot)^{N^\epsilon+1}\right](\overline x)}\right|,  \label{Cepsnorm2}
\end{align}
and importantly:
\begin{align}
 0<\delta\le \delta_2\epsilon.\label{eq:delta2cond}
\end{align}
By \eqref{eq:phiest}, we have
\begin{align}
 C_1 \left(\frac{\vert \overline x\vert}{1- \overline x}\right)^{N^\epsilon+1}  \le (1-\alpha^\epsilon)\left|\mathcal T^\epsilon\left[(\cdot)^{N^\epsilon+1}\right](\overline x)\right|\le C_2 \left(\frac{\vert \overline x\vert }{1-\overline x}\right)^{N^\epsilon+1} \quad \forall\, -\delta_2\epsilon\le \overline x\le 0,\label{eq:phiest2}
\end{align}
see \eqref{eq:phidefn}.
\begin{lemma}
Fix $\delta_2>0$ and suppose that \eqref{eq:delta2cond} holds. Then the items \ref{Tbounded}--\ref{composition} in Lemma \ref{lemma:TOp} also hold true with $\Wert\cdot\Wert_{\delta}$ given by \eqref{Cepsnorm2}.
\end{lemma}
\begin{proof}
The proof of Lemma \ref{lemma:TOp} carries over since \eqref{eq:phiest2} holds.
%  The proof is simpler in the context of \eqref{Cepsnorm2} upon using that 
%  \begin{align*}
\end{proof}
In this way, we obtain similar versions of Proposition \ref{prop:fixpoint} (which only relies on the estimates in items \ref{Tbounded}--\ref{composition} in Lemma \ref{lemma:TOp}) and Proposition \ref{prop:final} with 
\begin{align*}
 \Vert H\Vert_\delta := \sup_{u\in [-\delta,0]} \vert H(u)\vert,
\end{align*}
using \eqref{eq:phiest2} to estimate $\overline V^\epsilon$ in the sup-norm. 
Then, by proceeding as above for $I=[0,\delta]$, 
% using 
% \begin{align*}
%  \Vert \overline V^\epsilon \Vert_\delta \le  \Vert \overline w_{N^\epsilon} (\cdot)^{N^\epsilon}\Vert_\delta+\Vert \overline U^\epsilon\Vert_\delta \le K.
% \end{align*}
we complete the proof of \thmref{main2} in the case $I=[-\delta,0]$, $0<\delta\le \delta_2\epsilon$, 
 %turn, the result holds for any $I\subset [-\delta_2\epsilon,\frac34]$ for which \eqref{VcondK} holds true, 

% Specifically, we replace $K$ in Proposition \ref{prop:fixpoint} by $\vert S_\infty^0 \vert K$, with $K$ b

% implies \thmref{main2}. 

% 
% TDB
% 
% Put ``mismatch here''. We need the following: Let $\overline x_*\in (0,1)$ be so that $\overline V^\epsilon(\overline x_*) =K$. Then $j^{N^\epsilon-1}\overline m^\epsilon(\overline x_*)=o(1)$. We can apply the same fixed-point argument with functions defined on $\overline x\in [0,\overline x_*]$. In this way, we find that the analytic weak-stable manifold always intersects either $\overline y=-\frac{K}{2}$ or $\overline y=\frac{K}{2}$ between $\overline x\in (0,1)$. Consequently, the unstable manifold of the saddle and the analytic weak-stable manifold never intersect. 
\section{Discussion}\label{sec:discussion}
In this paper, we have provided a detailed description of analytic weak-stable manifolds near  analytic saddle-nodes (under a certain smallness assumption of the quantity $\mu=\sup u^\epsilon$ in the general normal form, see \rsp{Hypothesis} \ref{deltacond}). In further details, we have identified the quantity $S_\infty^0$, with the property that $S_\infty^0\ne 0$ implies the following (cf. Theorems \ref{theorem:main1} and \ref{theorem:main2}):
\begin{enumerate}
\item[(R1)] \label{it:cm} The center manifold is nonanalytic.
\item[(R2)] \label{it:flapping} A certain flapping mechanism of the analytic weak-stable manifold $W^{ws}$.
\item[(R3)] \label{it:non} $W^{ws}$ does not intersect the unstable manifold of the saddle $W^u$ for all $0<\epsilon\ll 1$, $\epsilon^{-1}\notin \mathbb N$.
\end{enumerate}
 
The quantity $S_\infty^0$ is reminiscent of a Stokes constant, see e.g. \cite{stokes1902a}.
Overall our approach is inspired by \cite{MR4445442}, proving statement (R2) for a specific (rational) system. In summary, \cite{MR4445442} performs a blowup (scaling) transformation, writes the weak analytic manifold as a power series in the scaled coordinates, truncated at an order just below the resonant term, and then treats the remainder through a certain integral operator ($\mathcal T^\epsilon$). We also follow this strategy, but have brought the method into a form that is more in tune with dynamical systems theory (through normal forms, center manifolds and fixed-point arguments). 
In this way, we established (R1)--(R3) as a generic phenomena for analytic saddle-nodes (albeit still within the context of \rsp{Hypotheses} \ref{assa} and \ref{deltacond}). \rspp{We emphasize that (R3) does not follow from the results of \cite{MR4445442} when applied to their specific nonlinearity.}  We also feel that we have obtained a deeper understanding of the underlying phenomena and also streamlined the method of \cite{MR4445442} along the way. For an example of the latter, in our treatment of $\mathcal T^\epsilon$ we have used a Banach space of analytic functions $H(\overline x)=\mathcal O(\overline x^{N^\epsilon+1})$ and set up a fixed-point argument for the remainder; this approach does not depend upon \rsp{Hypothesis} \ref{deltacond} (see discussion below).

% The results of the present paper show that this approach carries over to an analytic normal form of the saddle-node, provided that \rsp{Hypotheses} \ref{assa} and \ref{deltacond} hold true. To appreciate the significance of this, it is also important to emphasize that some results of \cite{MR4445442}  use the rational form of their equations explicitly (e.g. in the proof of \cite[Lemma 7.7]{MR4445442}). Moreover, we have brought the method into a form that is more in tune with dynamical systems theory (through center manifolds and fixed-point arguments) and feel that we have not only extended but also simplified some of the arguments along the way. 

\rsp{We conjecture that \rsp{Hypotheses} \ref{assa} and \ref{deltacond} can be removed so that our statements hold true for any analytic and generic unfolding of a saddle-node (by virtue of the normal form, see Theorem \ref{thm:nf}). In fact, we believe that $a^0\le -2$  can be removed with only a few changes to our argument (we just have to handle the finite sums $${\sum_{k\ge 2\,:\,k+a^0\le 0} (\cdots),}$$ separately).  }

%For example, we showed that the important quantity $S_\infty^0$ to the anayticity of the center manifold at the saddle-node $\epsilon=0$. 
% Whereas a condition like \rsp{Hypothesis} \ref{assa} is also present in \cite{MR4445442}, \rsp{Hypothesis} \ref{deltacond} does not appear there.

% At the same time, 

Let us emphasize where \rsp{Hypothesis} \ref{deltacond} is needed: It is used in the proof of $\widehat m^0\in \mathcal D^0$, see Proposition \ref{prop:m0D0}, and in the proof of the uniform boundedness of $\overline m^\epsilon$ in the semi-norm \eqref{eq:seminorm}, see Proposition \ref{prop:seminormmeps}. 

At the same time, it is important to emphasize that \textit{it is NOT needed in the treatment of the remainder, see Proposition \ref{prop:fixpoint}}. To see this, notice that in the proof of  Proposition \ref{prop:fixpoint},   we only  need that $\overline G$ has small Lipschitz-norm, see \rspp{the estimate of the third term in the proof of Proposition \ref{prop:fixpoint}}. To show this, we can first use the final condition of \eqref{eq:gkcond0}, see Remark \ref{remark:proph1}, to estimate the $\overline Y$-linear part of $\overline G$, and for the nonlinear part of $\overline G$  we can use Lemma \ref{lemma:TOp} item \ref{composition} (it is, in particular, $\mathcal O(\epsilon)$ in the norm $\Wert\cdot\Wert_\delta$ for all $\mu>0$). 
% The $\overline Y
% $-linear part  linear part  holds independently of \rsp{Hypothesis} \ref{deltacond}, due to the final condition of \eqref{eq:gkcond0}, see Remark \ref{remark:proph1}. Notice also here that by Lemma \ref{lemma:TOp} item \ref{composition}, % \textbf{The third term} in the proof of Proposition \ref{prop:fixpoint}. 

In other words, in order to remove \rsp{Hypothesis} \ref{deltacond} we just need to find alternative proofs of Proposition \ref{prop:m0D0} and Proposition \ref{prop:seminormmeps}. 

Let us focus on the former and $$\widehat m^0(x) = \sum_{k=2}^\infty {m_k^0}x^k\in \mathcal D^0\Longleftrightarrow \sup_{k\ge 2} \frac{\vert m_k^0\vert}{w_k^0} <\infty.$$ 
In \cite{MR4445442}, the authors show that their corresponding sequence $\frac{m_k^0}{w_k^0}$, $k\in \mathbb N \setminus\{1\}$, is bounded by essentially setting up a majorant equation for $S_k^0$ defined by
\begin{align}
 m_k^0 =: (-1)^k w_k^0 S_k^0,\label{eq:fuckthis}
\end{align}
see \cite[Lemma 5.4]{MR4445442}. %Recall that $S_k^0$ is defined in \eqref{eq:S0k} repeated here for convenienc
It follows from \eqref{eq:recmk0} (with $f_k^0$ replaced by $\rsp{\mathcal G^0_k:=\mathcal G^0[\widehat m^0]_k}$) that $S_k^0$ satisfies the following recursion relation
\begin{align*}
 S_k^0 = S_{k-1}^0 + \frac{(-1)^k \rsp{\mathcal G^0_k} }{w_k^0},\quad w_k^0 = \Gamma(k+a^0),
\end{align*}
with $\rsp{\mathcal G^0_k}$ depending on $S_2^0,\ldots,S_{k-3}^0$ (upon using \eqref{eq:fuckthis}).
The proof of \cite[Lemma 5.4]{MR4445442} first consists of showing (using the majorant equation for $S_k^0$) that $\vert S_k^0\vert \le K_0 C_0^k$ for some $K_0>0$, $C_0>1$ and all $k$. This is Step 3 of their proof. One can take $K_0=1,C_0=e^{n}$ for some $n\in \mathbb N$. 

Then in Step 4 of the proof of \cite[Lemma 5.4]{MR4445442}, the authors show that the exponential bound can be improved: There is some $\delta=\frac{1}{m}>0$, $m\in \mathbb N$ large enough, and a $K_1>0$ large enough such that $\vert S_k^0\vert \le K_1 (C_0e^{-\delta})^k=K_1 e^{k\left(n-\frac{1}{m}\right)}$ for all $k$. Importantly, the authors of \cite{MR4445442} show that this process can be iterated (with $n$ and $m$ fixed) in the following sense: For each $l\in \mathbb N$ with $n-\frac{l}{m}\ge 0$, there is a $K_l>0$ such that $$\vert S_k^0\vert \le K_{l}e^{k(n-l\delta)}=K_l e^{k\left(n-\frac{l}{m}\right)},$$ for all $k$. Here $l\in \mathbb N$ is the number of applications of the improvement. Setting $l=nm$ then gives $\vert S_k^0\vert \le K_{nm}$ (uniform bound) for all $k$ as desired. 

For our general normal form in Theorem \ref{thm:nf}, it is straightforward to reproduce Step 3 and the existence of $C_0$ from \eqref{gk12} (without using $\mu>0$ small). However, we have not been able to reproduce the argument in Step 4 in the general framework. %\footnote{In a private conversation, the authors of \cite{MR4445442} acknowledged a mistake in the equation just before Step 5 in their proof of Lemma 5.4. But the authors corrected this mistake in our correspondence.} 
We will pursue this further (along with alternative approaches to majorize $S_k^0$) in future work. \rspp{In fact, at the time of writing, the first author of the present paper posted a preprint on arXiv, see \cite{kristiansen2024improvedgevrey1estimatesformal}, on the existence of \eqref{eq:Sinfty00} in the generic case (i.e. without the assumptions on $\mu$ and $a^0$). The author uses a separate approach based upon Borel-Laplace and the number $S_\infty^0$ is connected with a singularity in the Borel plane. The preprint \cite{kristiansen2024improvedgevrey1estimatesformal} does not address $S_\infty^0=0$, recall Remark \ref{rem:converse}.}

It is our belief that our results will find use in different areas of dynamical systems, in particular in the area of singularly perturbed systems where weak-stable manifolds play an important role. \rsp{Here we are for example thinking of the weak canard of the folded node, see \cite{beno1990a,szmolyan_canards_2001}. Having said that, our approach is inherently planar and the minimal dimension of the folded node is three (with a two-dimensional critical manifold), so progress in this direction is therefore not just a simple incremental step. At the same time, we are confident that the overall approach of the paper (using power series expansion, identifying leading order terms and setting up fixed-point formulations) has potential (and is novel) in the context of the folded node. A natural starting point to extensions in higher dimensions, could also be to consider saddle-nodes in $(x,y)\in \mathbb R^{1+n}$; this would also be interesting at the level of $\epsilon=0$ (where $S_\infty^0$ would have to be reinterpreted). As other examples of future research directions, we mention extensions to different planar bifurcations with higher co-dimension, including the unfolding of the pitchfork. This would also require an extension of $S_\infty^0$ to Poincar\'e rank $r=2$ (where $\dot x=x^{3}$), which we believe is an interesting topic in itself. Here the results of \cite{new} could be relevant. Finally, we would also like to explore connections in the future to the interesting results of C. Rousseau \cite{rousseau2005a}, see also \cite{rousseau2023a} and references therein, on the analytic classification of unfoldings of saddle-nodes.} 

\subsection*{Acknowledgement} The second author would like to acknowledge inspiring and pleasant discussions with Heinrich Freist{\"u}hler in an early phase of this research project. 
% \newpage

% \newpage
\newpage 
\bibliography{refs}
\bibliographystyle{plain}
\newpage 

\appendix

\section{Basic properties of the gamma function}\label{sec:Gamma}
The gamma function $z\mapsto \Gamma(z)$, defined for $\operatorname{Re}(z)>0$ by
\begin{align}
 \Gamma(z) = \int_0^\infty t^{z-1} e^{-t}dt,\label{eq:Gammadef}
\end{align}
will play a crucial role. We therefore collect a few well-known facts (see e.g. \cite[Chapter 5]{NIST}) that will be used throughout the manuscript. 

First, we recall that $\Gamma(n+1)=n!$ for all $n\in \mathbb N_0$, {which follows from $\Gamma(1)=1$ and the basic property}
\begin{align}
 \Gamma(z+1)=z\Gamma(z) \quad \forall \, \operatorname{Re}(z)>0.\label{eq:Gamma}
\end{align}
The gamma function can be analytically extended to the whole complex plane except zero and the negative integers (which are all simple poles); specifically, 
\begin{align}\label{eq:Gamma0}
 \lim_{x\rightarrow 0} x\Gamma(x) = \Gamma(1)=1,
\end{align}

In this paper, we will use Stirling's well-known formula:
\begin{align}\label{stirling0}
 \Gamma(x+1) = (1+o(1))\sqrt{2\pi x} \left(\frac{x}{e}\right)^x,
\end{align}
for $x\rightarrow \infty$. The following form
\begin{align}
 \frac{\Gamma(x+b)}{\Gamma(x)} = (1+o(1))  x^b,\label{stirling}
\end{align}
for $b\in \mathbb R$ and $x\rightarrow \infty$,  which can be obtained directly from \eqref{stirling0}, will also be needed. We will also use 
the reflection formula:
\begin{align}\label{reflection}
 \Gamma(z) \Gamma(1-z)= \frac{\pi}{\sin \pi z}\quad \forall\,z\notin \mathbb Z.
\end{align}
and the Euler integral of the first kind:
\begin{align}\label{eq:integral}
 \int_0^{1} (1-v)^{x-1} v^{y-1} dv = \frac{\Gamma(x)\Gamma(y)}{\Gamma(x+y)}\quad \forall\,x,y>0.
\end{align}
% This equality can be proven by using the definition \eqref{eq:Gammadef} for the quantity on  the right hand side and then by using a change of variables for the integration. 

Finally, the digamma function $\phi$ is defined as the logarithmic derivative of the gamma function:
\begin{align}
 \phi(z): = \frac{\Gamma'(z)}{\Gamma(z)}.\label{eq:digamma}
\end{align}
It has a unique positive zero at $z\approx 1.4616312...$ and $\phi(z)$ is positive for all $z$-values larger than this number. It will be particularly important to us that $\phi$ is an increasing function of $z>0$:
\begin{align}
 \phi'(z)>0.\label{eq:digammaprop}
\end{align}

\section{Proof of Theorem \ref{thm:nf}}\applab{app:normalform}
 We first use \cite[Theorem 2.2]{rousseau2005a}, \rsp{where the following normal form (with unfolding parameter $\lambda$) is provided 
 \begin{equation}\nonumber
 \begin{aligned}
 \dot x &= \lambda-x^2,\\
  \dot y &=-y(1-a^{\lambda} x)+g^{\lambda}(x,y),
 \end{aligned}
 \end{equation}
 with
 \begin{align}
 g^{\lambda} (x,y) = (\lambda-x^2) f^{\lambda}(x)+y^2 u^{\lambda}(x,y).\label{eq:nfgapp0}
 \end{align}
 We focus on the singularity side of the bifurcation, i.e. $\lambda \ge 0$. The normal form is then analytic with respect to $\sqrt{\lambda}\ge 0$. 
Notice in comparison with \cite[Theorem 2.2]{rousseau2005a} that we reverse the direction of time and have replaced their $(a(\lambda),o(y))$ by $(-a^{\lambda},-y^2 u^{\lambda}(x,y))$, respectively. 
 We then put 
 \begin{align*}
  x=:-\widetilde x+\sqrt{\lambda},
 \end{align*}
so that 
\begin{equation}\nonumber
 \begin{aligned}
 \dot{\widetilde x} &= \widetilde x(\widetilde x-2\sqrt{\lambda}),\\
  \dot y &=-y(1-a^{\lambda} \sqrt{\lambda}+a^{\lambda}\widetilde x)+\widetilde g^{\lambda }(\widetilde x,y),
 \end{aligned}
 \end{equation}
 where $\widetilde g^{\lambda }$ is obtained from \eqref{eq:nfgapp0} and takes the same form (see \eqref{eq:nfgapp02}). We proceed to drop the tildes and then define (new tildes)
\begin{align}
\epsilon := \frac{2\sqrt{\lambda}}{\kappa(\sqrt{\lambda})},\quad \widetilde x := \frac{1}{\kappa(\sqrt{\lambda})} x.\label{eq:finaleps}
\end{align}
where
\begin{align*}
 \kappa(\sqrt{\lambda}): = 1-a^{\lambda} \sqrt{\lambda},
\end{align*}
which are all well-defined for all $0\le \lambda\ll 1$. We then obtain that
\begin{equation} \label{eq:nfapp0}
\begin{aligned}
 \widetilde x'  & =\widetilde x(\widetilde x-  \epsilon),\\
 y' &=- y(1+\widetilde a^{\epsilon} \widetilde x) +\widetilde g^{\epsilon}(\widetilde x,y),
\end{aligned}
\end{equation}
after dividing the right hand side by ${\kappa(\sqrt{\lambda})}$. This corresponds to a reparameterization of time. Here we have defined
\begin{align}
 \widetilde a^{\epsilon} = a^{\lambda( \epsilon)},\quad 
 \widetilde g^{ \epsilon} (\widetilde x,y) = \widetilde x(\widetilde x- \epsilon) \widetilde f^{ \epsilon}(\widetilde x)+y^2 \widetilde u^{ \epsilon}(x,y),\label{eq:nfgapp02}
 \end{align}
 and used that the first equation in \eqref{eq:finaleps} can be (analytically) inverted for $\sqrt{\lambda}=\sqrt{\lambda}( \epsilon)$, $\sqrt{\lambda}(0)=0$, $\frac{d\sqrt{\lambda}}{d\epsilon}(0)=\frac12$. We then drop the tildes again.}
 
%   In comparison with \cite[Theorem 2.2]{rousseau2005a} we have here replaced their $(x,\epsilon,t,g_0,o(y))$ by $$(-x +\sqrt{\epsilon}, \frac14 \epsilon^2,-t, -f^\epsilon,u^\epsilon(x,y)=y^2 \widetilde u^\epsilon(x,y)),$$
%  respectively.
 Now, in order to achieve the desired normal form in Theorem \ref{thm:nf}, we then apply three elementary transformations (T1)--(T3) to obtain \eqref{eq:nfapp0} (with tildes dropped) with $g^\epsilon$ given by \eqref{eq:gcondapp} and satisfying \eqref{eq:gkcond0app}, repeated here for convenience:
\begin{equation}\label{eq:gcond0}
\begin{aligned} 
 g^\epsilon(x,y) &= f^\epsilon(x)+u^\epsilon(x,y)\\
 f^\epsilon(x) = \sum_{k=2}^\infty f_{k}^\epsilon x^k, \quad & u^\epsilon(x,y) = \sum_{k=2}^\infty u_{k,1}^\epsilon x^k y+\sum_{k=1}^\infty\sum_{l=2}^\infty u_{k,l}^\epsilon x^k y^l,
\end{aligned}
\end{equation}
and \begin{align}
 \vert f_{k}^\epsilon\vert \le B \rho^{-k},\quad \vert u_{k,l}^\epsilon\vert \le  \mu \rho^{-k-l}\quad \mbox{and}\quad u_{k,1}^0= 0\quad \forall\,k,l\in \mathbb N,\,\epsilon \in[0,\epsilon_0),\label{eq:gkcond0fuck}
\end{align}
respectively.
The purposes of each of these successive transformations are:
\begin{enumerate} 
\item[(T1)] \label{t1} Remove the $x$-linear term of $g^\epsilon$ from \eqref{eq:nfgapp02}  ($-\epsilon x f^\epsilon(0)$) on the right hand side of the $y$-equation in \eqref{eq:nfapp0}. 
 \item[(T2)] \label{t2} Remove the $y$-linear term of the resulting nonlinearity $g^\epsilon=f^\epsilon+u^\epsilon$ obtained after application of (T1) for $\epsilon=0$: $u_{k,1}^0=0$ for all $k\in \mathbb N\setminus\{1\}$, see the last condition in \eqref{eq:gkcond0fuck}.
 \item[(T3)] \label{t3} Remove the $x=0$ part of the resulting  nonlinearity $g^\epsilon=f^\epsilon+u^\epsilon$ obtained after application of (T2): $u_{0,l}^\epsilon=0$ for all $l\in \mathbb N\setminus\{1\}$ and all $\epsilon\in [0,\epsilon_0)$, see  \eqref{eq:gcond0}.
\end{enumerate}

For (T1), we define $\widetilde y=y+\epsilon \frac{f^\epsilon(0)}{1-\epsilon}x$. 
This gives the following system
\begin{equation}\label{eq:normalformapp0}
\begin{aligned}
 \dot x &= (x-\epsilon)x,\\
 \dot{\widetilde y} &=-\widetilde y(1+\widetilde a^\epsilon x)+\widetilde g^\epsilon(x,\widetilde y),
\end{aligned}
\end{equation}
with 
\begin{align}\label{eq:gcondapp0}
\widetilde g^\epsilon(x,\widetilde y) = \sum_{k=2}^\infty \widetilde f_{k}^\epsilon x^k +\sum_{k=2}^\infty \widetilde u_{k,1}^\epsilon x^k y+\sum_{k=0}^\infty\sum_{l=2}^\infty \widetilde u_{k,l}^\epsilon x^k y^l,
\end{align}
% satisfying 
% \begin{align*}
%   \widetilde g_{0,0}^\epsilon = \widetilde g_{1,0}^\epsilon =  \widetilde g^\epsilon_{1,1}=0,
% \end{align*}
and
\begin{align*}
 \widetilde a^\epsilon =a^\epsilon+2\epsilon \frac{f^\epsilon(0)}{1-\epsilon} h^\epsilon(0,0).
\end{align*}
This completes (T1). We drop the tildes.

For (T2), 
% In comparison with \eqref{eq:gcondapp0} and \eqref{eq:gkcond0app}, the last sum in \eqref{eq:gcondapp0} starts with $k=0$ as opposed to $k=1$ in \eqref{eq:gcond0} and $\widetilde u_{k,1}^0\ne 0$ in general. To achieve the last property ($u_{k,1}^0=0$), 
we introduce a new $x$-fibered diffeomorphism defined by
\begin{align*}
(x,y)\mapsto  \widetilde y = e^{-\psi(x)} y,\quad \psi(x) := \sum_{k=2}^\infty \frac{u_{k,1}^0}{k-1} x^{k-1}\Longrightarrow x^2\psi'(x) = \sum_{k=2}^\infty u_{k,1}^0 x^{k}. 
\end{align*}
A simple calculation then shows that in the new $(x,\widetilde y)$-coordinates, we obtain a system of the form \eqref{eq:normalformapp0} with $\widetilde g^\epsilon$ given by \eqref{eq:gcondapp0}, for a new $\widetilde f^\epsilon$ and a new $\widetilde u^\epsilon$ now satisfying $u_{k,1}^0=0$ for all $k\in \mathbb N$, upon dropping the tildes.  This completes (T2). 

Now, finally for (T3) we analytically linearize the $x=0$-subsystem: There is a locally defined analytic near-identity diffeomorphism $y\mapsto \widetilde y=\psi^\epsilon(y)$, $\psi^\epsilon(0)=0$, $\frac{d}{dy}\psi^\epsilon(0)=1$, depending analytically on $\epsilon\in [0,\epsilon_0)$, such that 
\begin{align*}
 \dot y = -y+\sum_{l=2}^\infty u_{0,l}^\epsilon y^l\Longrightarrow \dot{\widetilde y} &=-\widetilde y.
\end{align*}
In the coordinates $(x,\widetilde y)$, we therefore obtain the desired form \eqref{eq:nfapp0} with $g^\epsilon$ given by \eqref{eq:gcond0} and satisfying \eqref{eq:gkcond0fuck} upon dropping the tildes a final time. In particular, the estimates in \eqref{eq:gkcond0fuck} follow from Cauchy's estimate for all $\rho>0$ small enough. 
\end{document}